\documentclass[11pt,a4paper]{article}

\usepackage{amsthm,amssymb,amsmath,geometry}
\usepackage[utf8]{inputenc}
\usepackage[usenames,dvipsnames]{color}
\usepackage{xstring}
\usepackage[titletoc,title]{appendix}
\usepackage{authblk}
\usepackage[normalem]{ulem}

\usepackage{float}
\usepackage{tikz}
\usepackage{tikz-3dplot}
\usepackage{standalone}
\usetikzlibrary{calc,decorations.pathreplacing,patterns}

\numberwithin{equation}{section}

\theoremstyle{plain}
\begingroup
\newtheorem{theorem}{Theorem}[section]
\newtheorem{lem}[theorem]{Lemma}
\newtheorem{proposition}[theorem]{Proposition}

\endgroup
\theoremstyle{definition}
\begingroup

\endgroup
\theoremstyle{remark}
\newtheorem{rmk}[theorem]{Remark}

\newcommand{\p}{\partial}

\newcommand{\eps}{\varepsilon}

\newcommand{\N}{\mathbb{N}}
\newcommand{\R}{\mathbb{R}}
\renewcommand{\S}{\mathbb{S}}
\newcommand{\T}{\mathbb T}
\newcommand{\Z}{\mathbb{Z}}

\newcommand{\EE}{{\mathcal E}}
\newcommand{\DD}{{\mathcal D}}
\newcommand{\HH}{{\mathcal H}}

\newcommand{\Gam}{\Gamma}

\newcommand{\dd}{\, {\rm d}}

\DeclareMathOperator*{\dist}{dist}
\DeclareMathOperator*{\argmin}{{\rm argmin}}
\DeclareMathOperator*{\Int}{\int}

\usepackage[textwidth=3cm]{todonotes}
\presetkeys{todonotes}{color=orange!85}{}

\renewcommand{\d}[2][0]{\,{\,\rm d^{\IfStrEqCase{#2}{{1}{}}[#2]}}{\IfStrEqCase{#1}{{0}{x}}[#1]}}

\newcommand{\ds}[1]   {\d[s]{#1}}
\newcommand{\dphi}[1]   {\d[\phi]{#1}}
\newcommand{\drho}[1]   {\d[\rho]{#1}}
\newcommand{\dx}[1]   {\d[x]{#1}}
\newcommand{\dy}[1]   {\d[y]{#1}}
\newcommand{\dz}[1]   {\d[z]{#1}}

\newcommand{\lupref}[2]{\hspace{0ex} \stackrel{\eqref{#1}}{#2}} 
\newcommand{\lupupref}[3]{\hspace{0ex}\stackrel{\eqref{#1},\eqref{#2}}{#3}}

\def\XXint#1#2#3{{\setbox0=\hbox{$#1{#2#3}{\int}$}
        \vcenter{\hbox{$#2#3$}}\kern-.5\wd0}}

\title{Magnetic domains in thin ferromagnetic films with strong perpendicular anisotropy}

\author[*]{\rm Hans Knüpfer} \author[$\dag$]{\rm Cyrill B. Muratov}
\author[*]{\rm Florian Nolte} \affil[*]{Institut für Angewandte
  Mathematik and Interdisciplinary Center for Scientific Computing
  (IWR), Universität Heidelberg, 69120 Heidelberg, Germany}
\affil[$\dag$]{Department of Mathematical Sciences and Center for
  Applied Mathematics and Statistics, New Jersey Institute of
  Technology, Newark, NJ 07102, USA}

\begin{document}

\maketitle
\begin{abstract}
  We investigate the scaling of the ground state energy and optimal
  domain patterns in thin ferromagnetic films with strong uniaxial
  anisotropy and the easy axis perpendicular to the film
  plane. Starting from the full three-dimensional micromagnetic model,
  we identify the critical scaling where the transition from single
  domain to multidomain ground states such as bubble or maze patterns
  occurs. Furthermore, we analyze the asymptotic behavior of the
  energy in two regimes separated by a transition. In the single
  domain regime, the energy $\Gamma$-converges towards a much simpler
  two-dimensional and local model. In the second regime, we derive the
  scaling of the minimal energy and deduce a scaling law for the
  typical domain size.
\end{abstract}

\newpage

\tableofcontents

\section{Introduction}
\label{sec:introduction}

Ferromagnetic materials are an important class of solids which have
played an indispensable role in data storage technologies of the
digital age \cite{moser02,eleftheriou10,stamps14}. Their utility for
technological applications stems from the basic physical property of
ferromagnets to exhibit spatially ordered magnetization patterns --
magnetic domains -- under a variety of conditions
\cite{HS-1998-magnetic_domains}. The mechanisms behind the magnetic
domain formation can be quite complex, but usually domain patterns may
be understood from the energetic considerations based on the
micromagnetic modeling framework \cite{HS-1998-magnetic_domains,brown,
  DKMO-2006-micromagnetics_review}. Starting with the early works of
Landau and Lifshitz \cite{landau1935theory} and Kittel
\cite{kittel46}, ground states of various ferromagnetic systems have
been the subject of extensive studies in the physics community (see
\cite{HS-1998-magnetic_domains} and references therein), and more
recently in the mathematical literature (for a review, see
\cite{DKMO-2006-micromagnetics_review}). In particular, within the
micromagnetic framework the ground state domain structure of
macroscopically thick uniaxial ferromagnetic films is by now fairly
well understood mathematically in terms of the energy and length
scales, as well as some of the qualitative properties of the domains
\cite{CK-1998-2D, CKO-1999-3D, conti00, KM-2011-branching, otto10}. In
contrast, apart from only a handful of studies
\cite{gioia1997micromagnetics, condette, melcher14, ms:prcla17}, the
vast majority of mathematical treatments of microscopically thin
ferromagnetic films deal with the situation in which the magnetization
prefers to lie in the film plane (see, e.g., \cite{GC-1999-magnetic,
  Carbou, DKMO, Moser, KohnSlastikov, DKO, Kurzke, ignat08,
  ignat2010vortex, ignat12, cm:non13}; this list is certainly not
complete). Thus, one of the fundamental open problems in the theory of
uniaxial ferromagnets is to rigorously characterize their ground
states in the case of films of vanishing thickness when the
magnetization prefers to align normally to the film plane (for various
ansatz-based computations in the physics literature, see \cite{kooy60,
  druyvesteyn71, Kaplan1993111, ng95}). This problem is the main
subject of the present paper.

Recent advances in nanofabrication allow an unprecedented degree of
spatial resolution, with features of only a few atomic layers in
thickness and tens of nanometers laterally for planar structures
\cite{stepanova}, enabling synthesis of ultrathin ferromagnetic films
and multilayer structures with novel material properties. Over the
last decade, there has been a major focus on films with thickness of
only a few atomic layers, primarily due to their promising
applications in spintronics \cite{bader10}. One of the important
features of these films is the emergence of perpendicular
magnetocrystalline anisotropy due to the increased importance of
surface effects \cite{heinrich93,ikeda10}, favoring the magnetization
vector to lie along the normal to the film plane. As a result, the
magnetization may exhibit either stripe or bubble domain phases
depending on the applied external field and other factors
\cite{speckmann95, saratz10, huang00, yamanouchi11, navas14}. We note
that studies of magnetic bubble domains in relatively thick films have
a long history in the context of magnetic memory devices (see, e.g.,
\cite{kooy60} and the book \cite{malozemoff}). However, the occurrence
of additional physical effects in ultrathin films, such as spin
transfer torque \cite{brataas12,fert13,jiang15}, Dzyaloshinskii-Moriya
interaction \cite{bogdanov94,rohart13} and electric field-controlled
perpendicular magnetic anisotropy \cite{endo10, matsukura15} allow for
much greater manipulation of the domain patterns, resulting in a
renewed attention to bubble domains from experimentalists
\cite{jiang15, kronseder2015real, woo16, schott16,
  2016arXiv160606034S}. In particular, the topological characteristics
of the bubble domain patterns in these materials are of great current
interest \cite{braun12, fert13, nagaosa13}. These considerations
further motivate the present study of the basic problem noted at the
end of the preceding paragraph. 

In this paper, we are interested in deriving a reduced two-dimensional
model for ultrathin ferromagnetic films with perpendicular anisotropy
and using it to asymptotically characterize the observed ground states
and, more generally, all low energy states in films of large spatial
extent. Our starting point is the three-dimensional micromagnetic
energy functional, coming from the continuum theory of uniaxial bulk
ferromagnets \cite{landau8}.  In a partially non-dimensionalized form,
the micromagnetic energy is given by
\begin{align} \label{eq:EE}
	\begin{aligned}
          \mathcal E[m] & = \int_{\Omega} \left( l_{\rm ex}^2|\nabla
            m|^2 + Q (m_1^2 + m_2^2) - 2 h_\mathrm{ext} \cdot m
          \right) \d3 + \int_{\R^3} \left|h\right|^2 \d3.
	\end{aligned}
\end{align}
In \eqref{eq:EE}, $\Omega \subset \R^3$ denotes the region in space
occupied by the ferromagnet and $\EE$ is minimized among all
$m \in H^1(\Omega;\S^2)$. The stray field $h:\R^3 \to \R^3$ is
determined by the static Maxwell's equations in matter
\begin{align}
    \nabla \cdot (h + m) = 0 && \text{and}&& \nabla \times h = 0,
\end{align}
so that the energy density depends in a nonlocal way on
$m$. Furthermore, $h_{ext}: \R^3 \to \R^3$ describes an external
magnetic field. The exchange length $l_{\rm ex}$ and the
non-dimensional quality factor $Q$ are material parameters. For an
introduction to micromagnetic modeling we refer to
e.g. \cite{HS-1998-magnetic_domains,DKMO-2006-micromagnetics_review}.
Note that additional physical effects due to the film surfaces may be
easily incorporated and would lead to the same type of a reduced
two-dimensional model \cite{ms:prcla17}.

Since our focus is on materials with perpendicular anisotropy, we
assume that the parameter $Q$ in \eqref{eq:EE} is greater than one
(for a detailed explanation, see the following section). The high
anisotropy leads to magnetizations that are predominantly
perpendicular to the film plane. It is well-known that such materials
feature magnetizations that consist of one or many regions of nearly
constant magnetization, called \textit{magnetic domains}, separated by
interfaces, called \textit{domain walls}. In this work, we identify
the critical scaling for the size of the sample where a transition
from single domain states to multidomain states occurs. Moreover, we
analyze the asymptotic behavior of the energy in the two regimes
separated by this transition. In the subcritical regime, the global
minimizers are the single domain states $m=\pm {e_3}$. We derive the
asymptotic behavior of the energy in this regime in the framework of
$\Gamma$-convergence. The reduced energy turns out to be much simpler
than the full energy, in particular, it is two-dimensional and local.
In the supercritical regime, which lies beyond the transition towards
multidomain configurations, we establish the scaling of the energy (up
to a multiplicative constant) and characterize sequences that achieve
this scaling. Our analysis shows that the magnetization in this regime
consists of several domains and suggests that the typical distance
between domain walls scales as
\begin{align}
  \text{typical domain size }S \sim
  \frac{e^{\frac{2\pi l_{\rm ex} \sqrt{Q-1}}{T}}}{\sqrt{Q-1}}l_{\rm ex} 
\end{align}
where $T$ is the thickness of the film.

We will show that in the regimes we consider the leading order of the
micromagnetic energy, upon rescaling and subtracting a constant, is
given by the following two-dimensional functional defined for
$m \in H^1(\T^2; \S^2)$:
\begin{align}\label{eq:F_intro}
  F_{\eps, \lambda}[m] =
  \int_{\T^2} \left(\frac{\eps}2 |\nabla m|^2 + \frac1{2\eps}
  (1-m_3^2)\right)\d2- \frac{\lambda}{|\log
  \eps|}\int_{\T^2}|\nabla^{1/2} m_3|^2 \d2. 
\end{align}
In \eqref{eq:F_intro}, $\T^2 = \R^2/\Z^2$ denotes the square flat
torus of unit side length, and we have assumed periodicity to avoid
boundary effects for simplicity (see also the next section), $\eps$ is
the renormalized Bloch wall width and $\lambda$ is the renormalized
film thickness (see the following section for the precise
definitions). We note that a similar result for a closely related
  problem of a Ginzburg-Landau energy with dipolar interactions has been
  obtained in \cite{m17}, where the meaning of the asymptotic
  equivalence between the full energy of three-dimensional
  configurations and the reduced energy of their $e_3$-averages is
  discussed in more detail.

The main part of our analysis is concerned with the
asymptotic behavior of \eqref{eq:F_intro} as $\eps \to 0$ for
different values of $\lambda >0$. Note that the last term in
\eqref{eq:F_intro} occurs with a negative sign and hence prefers
oscillations of $m_3$. As it turns out, the value of the parameter
$\lambda$ is crucial - in fact, we will show that the asymptotic
behavior changes at $\lambda = \lambda_{c}$, where
$\lambda_c = \frac{\pi}{2}$, which is a singular point in the
terminology of \cite{BraidesTruskinovsky2008}. For
$\lambda < \lambda_{c}$ the $\Gam$-limit
$F_{*,\lambda} := \Gam(L^1)$-$\lim_{\eps \to 0}F_{\eps, \lambda}$
measures the length of the interface separating regions with
$m \approx \rm e_3$ and $m \approx -e_3$ (see Theorem
\ref{th:SUBcrit-F})
\begin{align}\label{eq:F_gam}
    F_{*,\lambda}[m] =
    \begin{cases}
      \left(1-\frac{\lambda}{\lambda_{c}}\right)\int_{\T^2}|\nabla
      m_3| \d2, & \text{for } m \in BV(\T^2;\{\pm e_3\}),\\
      + \infty,& \text{otherwise}.
    \end{cases} 
\end{align}
Note that the last term in \eqref{eq:F_intro} leads to a reduction of
the interfacial cost by $\frac{\lambda}{\lambda_c}$ compared to the
classical result \cite{ABV} for $\lambda=0$.  On the other hand, for
$\lambda > \lambda_{c}$, the scaling of the minimal energy changes
(see Theorem \ref{th:SUPERcrit-F})
\begin{align}
  \min F_{\eps,\lambda} \sim -\frac{\lambda \eps^{\frac{\lambda_c -
  \lambda}{\lambda}}}{|\log \eps|} \stackrel{\eps \to
  0}\longrightarrow - \infty, 
\end{align}
and sequences $(m_\eps)$ which achieve the optimal scaling
$F_{\eps,\lambda}[m_\eps] \sim \min F_{\eps,\lambda}$ are highly
oscillatory in the sense that
\begin{align}
  \int_{\T^2} |\nabla \left(m_\eps\right)_3| \d3 \sim
  \eps^{\tfrac{\lambda_c - \lambda}{\lambda}}  \stackrel{\eps \to
  0}\longrightarrow +\infty. 
\end{align}
Furthermore, for $\lambda\ge \lambda_c$, the leading order
contributions of all three terms in \eqref{eq:F_intro} cancel. The
main difficulty in the proof is to find asymptotically optimal
estimates for the non-local term.

A reduction of the full three-dimensional micromagnetic energy to a
local two-dimensional model in the thin film limit was first
established rigorously in
\cite{gioia1997micromagnetics}. Subsequently, several thin film
regimes for for magnetically soft materials have been identified and
analyzed, see e.g. \cite{Carbou,DKMO,Moser,KohnSlastikov,Kurzke,
  ignat2010vortex}. However, since we consider materials with high
perpendicular anisotropy, our setting is considerably different, as we
now explain.  For thin films of the form $\Omega = \T^2 \times (0,t)$,
the leading order contribution of the stray field energy penalizes the
out-of-plane component of the magnetization. Neglecting boundary
effects, we have (see e.g. Theorem \ref{th:h-m2d})
\begin{align*}
  \left| \int_{\T^2 \times \R} |h|^2 \d3 - \int_{\T^2 \times (0,t)} m_3^2
  \d3 \right| \lesssim t \int_{\T^2 \times (0, t)} |\nabla m|^2 \d3.
\end{align*}
To our knowledge, the first result in this direction is contained in
\cite{gioia1997micromagnetics}. In the absence of high perpendicular
anisotropy or a sufficiently strong external field (as in the
previously mentioned papers) the micromagnetic energy forces the
out-of-plane component $m_3$ to vanish asymptotically. In our setting,
the anisotropy energy
$Q\int_{\Omega} ( m_1^2 + m_2^2 ) \d3 = Q\int_{\Omega} ( 1-m_3^2 )
\d3$
is however sufficiently strong (recall that $Q>1$) such that low
energy configurations require $m \approx \pm
e_3$ on most of the domain.

The behavior of the material changes when the film can no longer be
considered to be thin. In \cite{CK-1998-2D} the scaling of the ground
state energy was identified for the two-dimensional micromagnetic
model and in \cite{CKO-1999-3D} for the three-dimensional
model. Magnetizations with optimal energy involve so-called branching
domain patterns which become finer and finer as they approach the
boundary of the sample. When the ferromagnetic sample is exposed to a
critical external field, a transition between a uniform and a
branching domain pattern occurs. The critical field strength and the
scaling of the micromagnetic energy for this regime were derived in
\cite{KM-2011-branching}. In our regime, the thickness of the film is
so small that this does not only exclude the branching patterns that
occur in bulk samples, but actually forces the magnetization to become
constant in the direction normal to the film plane.

\paragraph{Notation:} For $x \in \R^3$ we write $x = (x',x_3)$, where
$x'$ is the projection of $x$ onto the first two components. The
square flat torus with side length $\ell>0$ is denoted by
$\T_\ell^2 := (\R^2/\ell\Z^2)$, and we abbreviate $\T^2:=\T^2_1$. We
frequently identify functions $u:\T_\ell^2 \to \R$ with periodic
functions $v:\R^2 \to \R$ by means of the natural projection
$p:\R^2 \to \T_\ell^2$, i.e. $u=v\circ p$.

For $u \in L^1(\T_\ell^2 \times (0,t))$ we write
$\overline{u} \in L^1(\T_\ell^2)$ to denote the ${e_3}$-average, given
by
\begin{align}
  \overline{u}(x') = \frac{1}{t}\int_0^t u(x',x_3) \dd x_3.
\end{align}
Moreover, for every $v \in L^1(\T_\ell^2)$ we write
$\chi_{(0,t)}v \in L^1(\T^2_\ell \times (0,t))$ to denote the function
$(\chi_{(0,t)}v)(x',x_3) = \chi_{(0,t)}(x_3)v(x')$. 

Unless stated otherwise, the expression $f(x) \lesssim g(x)$ means
that there exists a universal constant $C>0$ such that the inequality
$f(x) \le C g(x)$ holds for every $x$. The symbol $\gtrsim$ is defined
analogously with $\ge$ instead of $\le$ and we write $\sim$ if both
$\lesssim$ and $\gtrsim$ hold.

For future reference, we now fix the constants in the definition of
the Fourier coefficients. For $f \in L^2(\T_\ell^2)$, we write
\begin{align}\label{eq:def_fourier}
  \widehat f_k = \int_{\T_\ell^2} e^{-i k \cdot x} f(x) \dx2, 
  && \text{where } k \in \frac{2 \pi}{\ell} \Z^2.  
\end{align}
The inverse Fourier transform is then given by
\begin{align}
  f(x) = \frac{1}{\ell^2} \sum_{k \in \frac{2 \pi}{\ell} \Z^2} e^{i k
  \cdot x} \widehat f_k, 
\end{align}
where convergence is understood in the $L^2(\T_\ell^2)$
sense. Parseval's Theorem then states that
\begin{align}
  \int_{\T^2_\ell} f^*(x)g(x) \d2 = \frac{1}{\ell^2} \sum_{k \in
  \frac{2 \pi}{\ell} \Z^2} {\widehat f}^*_k \widehat g_k 
  &&  \text{for } f,g \in L^2(\T_\ell^2), 
\end{align}
where ``$^*$'' denotes complex conjugation.  Furthermore, we use the
symbol $\nabla^su$ to denote
\begin{align}
  \int_{\T_\ell^2} |\nabla^{s}u|^2 \d2:= \frac{1}{\ell^2}\sum_{k \in
  \frac{2\pi}{\ell}\Z^2} |k|^{2s} |\widehat u_k|^2 
\end{align}
for $s \in \R$. For $s=1/2$ we will also use the following well-known
real space representation of the (square of the) homogeneous
$H^{1/2}(\T_\ell^2)$-norm
\begin{align}
  \int_{\T_\ell^2} |\nabla^{1/2}u|^2 \d2 =
  \frac{1}{4\pi}\int_{\T_\ell^2} \int_{\R^2} \frac{|u(x+y) -
  u(x)|^2}{|y|^3} \dy2 \dx2. 
\end{align}
For the convenience of the reader, a proof is contained in the
appendix.

Lastly, with the usual abuse of notation, for $\eps \to 0$ we will
refer to $(m_\eps) \in H^1(\T^2_\ell; \mathbb S^2)$ as a sequence,
implying the sequence of $m_{\eps_k} \in H^1(\T^2_\ell; \mathbb S^2)$
for some sequence of $\eps_k \to 0$ as $k \to \infty$. Similarly, when
dealing with the family of functionals $\{ F_{\eps,\lambda} \}$ we are
always dealing with sequences $F_{\eps_k, \lambda}$.

\section{Setting}
\label{sec:setting}

In order to non-dimensionalize the micromagnetic energy, we express
lengths as multiples of the exchange length $l_{\rm ex}$ and rescale
(effectively this amounts to setting $l_{\rm ex}=1$).  We are
interested in thin ferromagnetic films of uniform
(non-dimensionalized) thickness $t$. For simplicity, we assume that
the film extends infinitely in the film plane and that its
magnetization is periodic in both in-plane coordinates with period
$\ell$. This means that we neglect boundary effects in the case of a
finite sample of large spatial extent.

The film is composed of a uniaxial ferromagnetic material whose easy
axis is perpendicular to the film plane, i.e. parallel to
$e_3$. Furthermore, we assume that the external field $h_{\rm ext}$ is
parallel to $e_3$ and hence independent of $x_3$ (due to
$\nabla \cdot h_{\rm ext}=0$).  By a slight abuse of notation, from
now on, we consider $h_{ext}: \T_\ell^2 \to \R$ as a scalar
function. The non-dimensionalized energy per unit-cell
$\T_{\ell}^2 \times (0,t)$ then reads
\begin{align}\label{eq:E}
\begin{aligned}
  E[m] &:=\int_{\T_{\ell}^2\times (0, t)} \left( |\nabla m|^2 + Q
    (m_1^2 +m_2^2) - 2 m_3h_{\mathrm{ext}} \right) \d3 +
  \int_{\T_{\ell}^2 \times \R} |h|^2 \d3.
\end{aligned}
\end{align}
In the last term of \eqref{eq:E}, the stray field is the unique
distributional solution $h \in L^2(\T^2_\ell \times \R ; \R^3)$ of
\begin{align}\label{eq:h}
  \nabla \times h = 0 \qquad \text{and} \qquad \nabla \cdot (h+m) = 0
  \quad \text{in } \T^2_{\ell} \times \R, 
\end{align}
where $m \in H^1(\T^2_\ell \times (0,t))$ is extended by zero to
$\T^2_\ell \times \R$. Hence, up to a sign, $h$ equals the Helmholtz
projection of $m$ onto the space of gradients. We also use the
notation $h=h[m]$ to denote the solution of \eqref{eq:h}.

Note that \eqref{eq:E} depends on the three dimensionless parameters
$\ell, t$ and $Q$. We are interested in the asymptotic behavior of the
energy in \eqref{eq:E} for thin films (i.e. $t \ll 1$) with large
extension in the film plane (i.e. $\ell \gg 1$) and high anisotropy
(i.e. $Q>1$).

\begin{figure}%
    \includegraphics[width=0.9\linewidth]{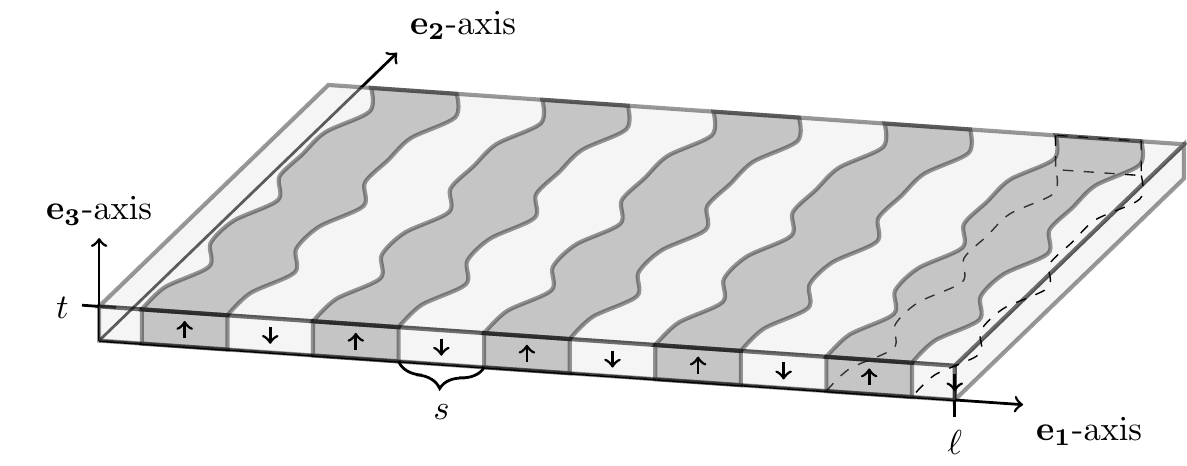}%
    \caption{Typical magnetization pattern ("stripe pattern") in a
      unit cell $\T_\ell^2 \times (0,t)$ of the ferromagnetic
      film. The arrows represent the value of the magnetization $m(x)$
      at $x$, which is approximately constant across regions of the
      same color. The domains are separated by continuous domain walls
      of vanishing thickness, depicted as lines.}%
    \label{fig:tikz:my}%
\end{figure}%

\subsection{Identification of the regimes and the reduced energy $F$}
\label{sec:heuristic}

In this section, we motivate the rigorous results contained in section
\ref{sec:main-results}. We use heuristic arguments to identify the
scaling of the transition between monodomain and multi-domain states,
and to explain how the micromagnetic energy $E$ in \eqref{eq:E} is
related to the two-dimensional reduced energy $F$ in
\eqref{eq:F_intro}. Roughly speaking, we will argue that (upon
rescaling) $F$ is a prototype for the next-to-leading-order term in
the $\Gamma$-development of $E$, cf. \cite{Anzellotti1993}.

To simplify the exposition, we neglect the energy contribution due to
the external field $h_{\rm ext}$. Furthermore we make two assumptions
(for this section only), stated below. These assumptions are actually
consequences of the thin film regime (see \eqref{eq:poincare_3} and
Theorem \ref{th:h-m2d}).  Our assumptions are:
\begin{enumerate}
\item[(i)] The magnetization $m$ is constant in the direction normal
  to the film, i.e.
  \begin{align}\label{eq:I}
    m(x',x_3) = \chi_{(0, t)}(x_3) \overline m(x') \qquad \text{for }
    x=(x',x_3) \in \T_\ell^2 \times (0,t). \tag{i} 
  \end{align}
\item[(ii)] The stray field energy can be approximated by
  \begin{align}
    \label{eq:II}\tag{ii}
    \int_{\T^2_\ell \times \R} |h[m]|^2\d3\approx
    t\int_{\T^2_\ell}\overline m_3^2\d2 - \frac{t^2}{2}
    \int_{\T^2_\ell}|\nabla^{1/2}\overline m_3|^2\d2.  
  \end{align}
\end{enumerate}
Assumption \eqref{eq:I} can be understood as a consequence of the
vanishing thickness of the film which is smaller than the thickness of
optimal domain walls (so-called Bloch walls).

We will now motivate Assumption \eqref{eq:II}. For magnetizations that
are constant in the normal direction of the film, i.e.
$m(x',x_3)= \chi_{(0,t)}(x_3)\overline m(x')$, it is well-known that
the stray field energy splits into a contribution due to the normal
component $\overline m_3$ and a contribution due to the in-plane
divergence
$\nabla' \cdot \overline m' = \p_1 \overline m_1 + \p_2 \overline
m_2$,
see e.g. \cite{A-1966-energy,GC-1999-magnetic}. With the aid of the
Fourier transform, a direct calculation yields (see also Theorem
\ref{th:h-m2d})
\begin{align}\label{eq:h_split_intro}
    \begin{aligned}
      \int_{\T_\ell^2\times \R} |h[m]|^2\d3 &= \frac{1}{\ell^2}\sum_{
        k \in \frac{2 \pi}{\ell}\Z^2}t\sigma(t |k|)|\widehat{\overline
        m}_{3,k}|^2\\
      &\quad + \frac{1}{\ell^2}\sum_{ k \in \frac{2
          \pi}{\ell}\Z^2}t\left(1-\sigma(t
        |k|)\right)\left|\frac{k}{|k|} \cdot \widehat{\overline
          m}'_k\right|^2,
    \end{aligned}
\end{align}
where the Fourier multiplier $\sigma$ is given by
$\sigma(s) = \frac{1-e^{-s}}{s}$.  In the electrostatics analogy, the
first term on the right hand side can be understood as the
contribution of surface charges proportional to $\overline m_3$ at the
top and bottom surface of the film, whereas the second term describes
the contribution due to volume charges proportional to
$\nabla '\cdot \overline m'$.  Since the strong anisotropy requires
$|m_3|\approx 1$ on most of the domain, a scaling argument indicates
that only the contribution due to $m_3$ is relevant. Indeed, since
$|1-\sigma(t|k|)| \le t |k| \le t(1+|k|^2)$ the contribution due to
$m'$ may be estimated by the exchange and anisotropy energy at lower
order
\begin{align}\label{eq:lo}
  \frac{1}{\ell^2}\sum_{ k \in \frac{2
  \pi}{\ell}\Z^2}t\left(1-\sigma(t |k|)\right)\left|\frac{k}{|k|}
  \cdot \widehat{\overline m}'_k\right|^2 \le t^2 \int_{\T_\ell^2}
  \left( |\nabla m|^2 + |m'|^2 \right) \d2. 
\end{align}
The right hand side of \eqref{eq:II} is obtained by neglecting the
second term on the right hand side of \eqref{eq:h_split_intro} and
approximating $\sigma(s)\approx 1-\frac{s}{2}$ in the first term (see
Theorem \ref{th:h-m2d} for a rigorous version).

With \eqref{eq:I}, \eqref{eq:II} and $h_{\rm ext} = 0$, the energy
\eqref{eq:E} can now be written as
\begin{align}\label{eq:E_approx}
  E[m] \approx t\int_{\T^2_\ell} \left( |\nabla \overline m|^2 + Q
  \left(\overline{m}_1^2 + \overline{m}_2^2 \right) \right) \d2 +
  t\int_{\T^2_\ell} \overline{m}_3^2 \dd x -
  \frac{t^2}{2}\int_{\T_\ell^2}|\nabla^{1/2}\overline m_3|^2\d2. 
\end{align}
We use the constraint $|\overline m|=1$ to combine the leading order
stray-field energy term with the anisotropy energy
\begin{align}\label{eq:combine-terms}
   \int_{\T^2_\ell} \overline{m}_3^2 \d2 + \int_{\T^2_\ell}Q \left(\overline{m}_1^2 + \overline{m}_2^2\right)
  \d2 = \ell^2  +
  \int_{\T^2_\ell}(Q-1) \left(\overline{m}_1^2 +
  \overline{m}_2^2\right) \d2.
\end{align}
Inserting \eqref{eq:combine-terms} into \eqref{eq:E_approx} allows to
extract the leading order constant
\begin{align}
  E[m] \approx \ell^2 t + t\left(\int_{\T^2_\ell} \left( |\nabla \overline
  m|^2 + (Q-1) \left(\overline{m}_1^2 + \overline{m}_2^2 \right)
  \right) \d2 -
  \frac{t}{2}\int_{\T_\ell^2}|\nabla^{1/2}\overline m_3|^2\d2\right). 
\end{align}
Upon rescaling $\T_\ell^2$ to the fixed domain $\T^2$ and
renormalizing the energy, we obtain
\begin{align}\label{eq:E-heuristic-3}
\begin{aligned}
  \frac{E[m(\ell \cdot)]-\ell^2 t}{\ell t \sqrt{Q-1}} &\approx
  \int_{\T^2} \left( \frac{1}{\ell \sqrt{Q-1}}|\nabla \overline m|^2 +
    \ell\sqrt{Q-1} \left(\overline{m}_1^2 + \overline{m}_2^2\right)
  \right) \d2\\ &\qquad-
  \frac{t}{2\sqrt{Q-1}}\int_{\T^2}|\nabla^{1/2}\overline m_3|^2\d2.
    \end{aligned}
\end{align}
In order to determine the critical scaling where minimizers of
\eqref{eq:E-heuristic-3} cease to be constant and start to oscillate,
we ask for which $\ell,t$ and $Q$ it is possible to control the last
term by the first integral
\begin{align}
  \frac{t}{2\sqrt{Q-1}}\int_{\T^2}|\nabla^{1/2}\overline m_3|^2\d2
  \stackrel{\textbf{?}}\lesssim \int_{\T^2} \left( \frac{1}{\ell
  \sqrt{Q-1}}|\nabla \overline m|^2 + \ell\sqrt{Q-1}
  \left(\overline{m}_1^2 + \overline{m}_2^2\right) \right) \d2. 
\end{align}

\begin{figure}
    \centering \includegraphics[width=0.9\linewidth]{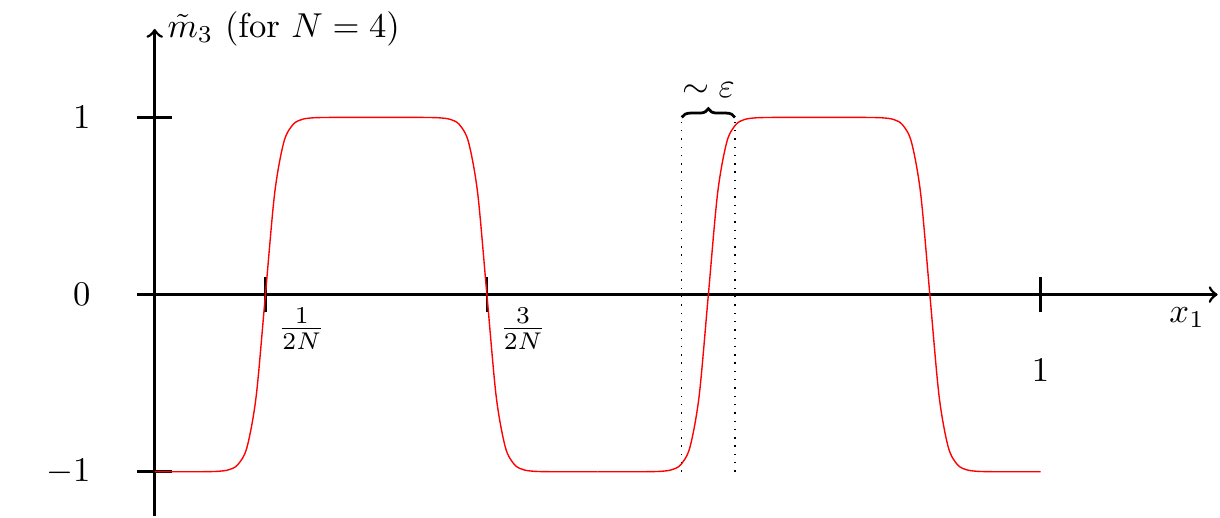}%
    \caption{One-dimensional ansatz modeling a stripe pattern.}
    \label{fig:ansatz}
\end{figure}

We make a one-dimensional ansatz $\tilde m$ corresponding to $N$
domains separated by smooth domain walls of width $\eps$, see Figure
\ref{fig:ansatz}.  For the nonlocal term, a straightforward
computation yields (see Lemma \ref{le:opt})
\begin{align}
  \int_{\T^2} |\nabla^{1/2}\tilde m_3|^2 \dx2 =
  \frac{1}{4\pi}\int_{\T^2} \int_{\R^2} \frac{|\tilde m_3(x+z)-\tilde
  m_3(x)|^2}{|z|^3} \dz2 \dx2 \approx {\frac{4}{\pi}}
  \log\left(\frac{1}{\eps N}\right)N.
\end{align}
Since the nonlocal term depends only logarithmically on the transition
layer, we optimize the width and internal structure of the transition
layer for the first two terms in the energy by choosing
$\eps = \frac{1}{\ell \sqrt{Q-1}}$.  For the corresponding Bloch wall
profiles \cite{HS-1998-magnetic_domains}, we obtain
\begin{align}
  \int_{\T^2} \left( \frac{1}{\ell \sqrt{Q-1}}|\nabla \tilde m|^2 +
  \ell\sqrt{Q-1}(\tilde m_1^2 + \tilde m_2^2) \right) \dx2
  \approx 2\int_{\T^2} 
  |\nabla \tilde m_3| \dx2 \approx 4N. 
\end{align}
Hence
\begin{align}\label{eq:E_N}
  \frac{E[\tilde m(\ell \cdot)]-\ell^2 t}{\ell t \sqrt{Q-1}} \approx
  N\left(4 - \frac{2 t}{\pi \sqrt{Q-1}}\log\left(\frac{\ell
  \sqrt{Q-1}}{N}\right)\right). 
\end{align}
The (renormalized) energy of our ansatz \eqref{eq:E_N} becomes
negative, i.e. smaller than the energy of the constant configurations
$m\equiv \pm e_3$, if
$8\sqrt{Q-1} < \frac{4}{\pi} t\log\left(\frac{\ell
    \sqrt{Q-1}}{N}\right)$.
By monotonicity in $N$, we expect that the critical scaling occurs for
$N=1$ and $t \sim t_c$, where
\begin{align}\label{eq:scaling-approx}
  t_c \approx \frac{2 \pi \sqrt{Q-1}}{\log\left(\ell
  \sqrt{Q-1}\right)} 
\end{align}
is the critical thickness of the onset of multidomain states.

Inserting \eqref{eq:scaling-approx} into \eqref{eq:E-heuristic-3} and
abbreviating 
\begin{align}
  \label{eq:epslam}
  \eps = \frac{1}{\ell \sqrt{Q-1}},  \qquad \lambda = \frac{t \log \left(
  \ell \sqrt{Q - 1} \right)}{4 \sqrt{Q - 1}},
\end{align}
we are led to study the asymptotic behavior for $\eps \to 0$ of the
family of functionals
$F_{\eps, \lambda}: L^1(\T^2; \S^2) \to \R \cup \{+\infty\}$, given by
\begin{align}\label{eq:F}
\begin{aligned}
  F_{\eps, \lambda}[m] &=
    \begin{cases}
    \begin{aligned}
        \int_{\T^2} \left(\frac{\eps}2 |\nabla m|^2 + \frac1{2\eps}
          \left(1-m_3^2\right)\right)\d2- \frac{\lambda}{|\log
          \eps|}\int_{\T^2}|\nabla^{1/2} m_3|^2\d2 
      \end{aligned} & \\
      \hspace{7.5cm} \text{if } m \in H^1(\T^2; \S^2),\\
      +\infty \hspace{6.8cm} \text{otherwise},
\end{cases}
\end{aligned}
\end{align}
where $\lambda \sim 1$ is a fixed parameter and with
$\min E \approx \ell^2 t + 2 \ell t \sqrt{Q - 1} \, \min F_{\eps,
  \lambda}$.

\begin{rmk}(Natural cut-off in the stray field
  energy)\label{eq:nat-scales} For thin films, a natural approximation
  for the stray field energy is given by
    \begin{align}
      \int_{\T_\ell^2 \times \R} |h[m_3e_3]| \d3 \approx
      \int_{\T_\ell^2 \times (0,t)} m_3^2 \d3 - \frac{t^2}{8\pi}
      \int_{\T_\ell^2}\int_{\R^2 \setminus B_{t}} \frac{|\overline
      m_3(x+z) - \overline m_3(x)|^2}{|z|^3} \dz2 \dx2, 
    \end{align}
    i.e. the region $|z|\le t$ is excluded in the last
    integral. However, our approximations in \eqref{eq:II} and in
    Theorem \ref{th:h-m2d} ignore this cut-off. We will now explain,
    that due to periodicity, this cut-off is not relevant in our
    setting. Roughly speaking, the reason is that the length scale of
    the cut-off is much smaller than the width of domain walls, which
    is the smallest length scale on which $m$ varies. More precisely,
    we have (see Lemma \ref{le:main_ineq})
    \begin{align}
        \begin{aligned}\label{eq:cutoff}
          t^2\int_{\T_\ell^2}\int_{B_{t}} \frac{|\overline m_3(x+z) -
            \overline m_3(x)|^2}{|z|^3} \dx2 \dz2 &\lesssim
          t^3\int_{\T_\ell^2 } |\nabla \overline m_3|^2 \d2\\ 
          &\lesssim t^2\int_{\T_\ell^2 \times (0,t)} |\nabla m|^2 \d3,
        \end{aligned}
    \end{align}
    so that the effect due to the cut-off is controlled by the
    exchange energy at lower order. Here we have implicitly used that
    the film is periodic and hence does not have boundaries. On the
    other hand, if the ferromagnetic material is modeled by a finite
    domain $\Omega \times (0,t)$, exploiting the cut-off in the stray
    field energy becomes crucial: At the boundary $\partial \Omega$,
    the out-of-plane component $m_3$ may have a jump so that
    $\|m_3\|_{H^{1/2}(\R^2)}$ would be infinite. Since the exchange
    energy is oblivious to this jump at the boundary,
    \eqref{eq:cutoff} does not hold for $\Omega$ instead of
    $\T_\ell^2$.
\end{rmk}

\section{Main results and overview of the proof}
\label{sec:main-results}

Our main result is the identification of two thin-film regimes
separated by a transition and the derivation of the asymptotic
behavior of the energy in the regimes. We will state the results for
the full energy $E$ in Section \ref{sec:results-for-E} and for the
reduced energy $F$ in Section \ref{sec:results-for-F}.

\subsection{Results for the full energy $E$}
    \label{sec:results-for-E}

In terms of $\ell$, $t$ and $Q$, the regimes may be expressed by
\begin{align}\label{eq:regime}
  Q > 1, \qquad \ell \gg 1 \qquad \text{ and }\qquad \frac{t|\log
  \left(\ell\sqrt{Q-1}\right)|}{4\sqrt{Q-1}} = \lambda  
\end{align}
and $\lambda_c:=\pi/2$, where
\begin{enumerate}
\item[\,] $\lambda < \lambda_c$ corresponds to the subcritical regime
  featuring single domain states,
\item[\,] $\lambda = \lambda_c$ corresponds to the transition,
\item[\,]
  $\lambda_c < \lambda < \gamma \frac{|\log(\ell \sqrt{Q-1})|}{Q-1}$,
  for some universal $\gamma >0$, corresponds to the multidomain
  state.
\end{enumerate}
The upper bound $\lambda < \gamma \frac{|\log(\ell \sqrt{Q-1})|}{Q-1}$
is necessary because in general magnetizations may not be
approximately two-dimensional beyond this threshold.

It is convenient to rescale the domain of the ferromagnetic film to a
fixed domain by means of the anisotropic transformation
\begin{align}\label{eq:transform}
  \T^2_\ell \times (0,t) \to \T^2 \times (0,1) \quad \text{with }
  (x_1,x_2,x_3) \mapsto \left(\frac{x_1}{\ell},\frac{x_2}{\ell},
  \frac{x_3}{t}\right), 
\end{align}
and study the renormalized energy
$J: L^1(\T^2\times (0,1); \S^2) \to \R \cup \{+\infty\}$, defined by
\begin{align}\label{eq:J}
  J[m]:= \begin{cases}
    \displaystyle \frac{E[m(\ell \cdot, \ell \cdot, t \cdot)] - \ell^2
      t}{\ell t \sqrt{Q-1}} & \text{for }m \in H^1(\T^2\times (0,1);
    \S^2),\\ 
    +\infty & \text{otherwise}.
\end{cases}
\end{align}
The asymptotic behavior of \eqref{eq:J} in the subcritical regimes is
characterized in the following theorem.

\begin{theorem}[Subcritical regime]\label{th:SUBcrit-E}
  Let $\lambda_c:={\frac{\pi}{2}}$, $\lambda \in [0,\lambda_c)$, $Q>1$
  and $(\ell_k,t_k, h_{{\rm ext},k})_{k \in \N}$ be a sequence with
    \begin{align}\label{eq:SUBcrit-E-regime}
      \ell_k \to \infty, \quad \frac{t_k|\log
      \left(\ell_k\sqrt{Q-1}\right)|}{ 4\sqrt{Q-1}} = \lambda \quad
      \text{and} \quad \frac{\ell_k}{\sqrt{Q-1}} \, h_{\rm ext,k}(\ell_k
      \cdot) \to  g 
    \end{align}
    for some $g \in L^1(\T^2)$ and for all $k \in \N.$ Then the
    sequence of renormalized energies $\{J_k\}_{k \in \N},$ defined by
    \eqref{eq:J} with $(\ell,t,h_{\rm ext})$ replaced by
    $(\ell_k,t_k,h_{{\rm ext},k})$, satisfies
    \begin{enumerate}
    \item \textbf{Compactness}: For every sequence
      $(m_k) \in L^1(\T^2\times (0,1);\S^2)$ with
      \begin{align}\label{eq:E_bound}
        \limsup_{k \to \infty}J_k[m_k] < + \infty,
      \end{align}
      there exists a sub-sequence (not relabeled) and
      $\overline m \in BV(\T^2;\{\pm e_3\})$ such that
      \begin{align}\label{eq:m_conv}
        \int_{\T^2 \times (0,1)}|m_k(x)- \overline m(x')| \dx3 \to 0
        \text{ for }k \to \infty. 
      \end{align}
    \item \textbf{$\Gam$-Convergence}: The sequence of functionals
      $\{J_k\}_{k \in \N}$ $\Gamma$-converges towards\\
      $J_*: L^1(\T^2;\{\pm e_3\}) \to \R \cup \{+\infty\}$ given by
      \begin{align}
        J_*[\overline m] 
        &= \left\{
          \begin{array}{ll}
            \displaystyle
            2\left(1-\frac{\lambda}{\lambda_c}\right)\int_{\T^2}
            |\nabla \overline m_3|\d2 - 2
            \int_{\T^2}g\overline{m}_3 \d2  
            & \text{
              if
              }\overline
              m \in
              BV(\T^2;\{\pm
              e_3\}),\\ 
            +\infty 
            & \text{ otherwise}.
          \end{array}
              \right.
      \end{align}
      This means
      \begin{enumerate}
      \item liminf - Inequality: Every sequence
        $(m_k) \in L^1(\T^2\times (0,1);\S^2)$ that converges towards
        $\overline m \in L^1(\T^2;\{\pm e_3\})$ in the sense of
        \eqref{eq:m_conv} satisfies
        \begin{align}
          \liminf_{k \to 0} J_k[m_k]\ge J_*[\overline m].
        \end{align}
      \item Recovery Sequence: For every
        $\overline m \in L^1(\T^2,\{\pm e_3\})$ there exists a
        sequence of magnetizations
        $(m_k) \in L^1(\T^2 \times (0,1);\S^2)$ which converges
        towards $\overline m$ in the sense of \eqref{eq:m_conv} and
        satisfies
        \begin{align}
          \limsup_{k \to 0} J_k[m_k]\le J_*[\overline m].
        \end{align}
      \end{enumerate}
    \end{enumerate}
\end{theorem}

Whereas the energy favors single domain states in the subcritical
regime, our next theorem shows that the energy leads to pattern
formation in the supercritical regime.

\begin{theorem}[Supercritical regime]\label{th:SUPERcrit-E}
  Let $h_{\rm ext}=0$. There are universal constants $\delta, K > 0$
  such that for $Q,\ell,t>0$ in the regime 
    \begin{align}\label{eq:re2}
      Q>1, \quad t \le \delta \min\left\{\sqrt{Q-1},
      \frac{1}{\sqrt{Q-1}}\right\} 
      \quad \text{and} \quad 
      \ell \ge K \frac{e^{2\pi t^{-1} \sqrt{Q-1}}}{\sqrt{Q-1}}
    \end{align}
    the minimal renormalized energy $J$ in \eqref{eq:J} scales as
    \begin{align}
      -Ct\ell e^{-2\pi t^{-1} \sqrt{Q-1}} \le \min J[m] \le -ct\ell
      e^{-2\pi t^{-1} \sqrt{Q-1}},
    \end{align}
    for some universal constants $0<c<C$.
\end{theorem}
Furthermore, profiles achieving the optimal scaling in the regime
\eqref{eq:re2} can be characterized as follows.
\begin{proposition}\label{prop:properties}
  Let $\delta,K$ be as in Theorem \ref{th:SUPERcrit-E},
  $h_{\rm ext}=0$ and let $\ell,t,Q$ satisfy \eqref{eq:re2}. For any
  $\gamma > 0$ and all $m \in H^1(\T^2 \times (0,1); \S^2)$ which
  satisfy
    \begin{align}\label{eq:t-opt-gamma}
        J[m] \le - \gamma t\ell e^{-2\pi t^{-1} \sqrt{Q-1}},
    \end{align}
    we have 
        \begin{align}\label{eq:t2-i1}
          (i)&& &\int_{\T^2\times (0,1)} |m - \overline m|^2 \d3 \le
                  c_\gamma t^3 e^{-2\pi t^{-1}
                  \sqrt{Q-1}} \sqrt{Q-1},\\ 
            (ii)&& \label{eq:oop}&\int_{\T^2\times (0,1)} \left(m_1^2 + m_2^2\right)
                     \d3 \le c_\gamma e^{-2\pi t^{-1} \sqrt{Q-1}},&&\\ 
          (iii)&& \label{eq:wl}
                &c_\gamma \ell e^{-2\pi t^{-1} \sqrt{Q-1}} \sqrt{Q-1}
                  \le \int_{\T^2} |\nabla' \overline m_3| \d2 \le
                  C_\gamma \ell e^{-2\pi t^{-1}
                  \sqrt{Q-1}} \sqrt{Q-1}\\ 
          (iv)&& &\int_{\T^2 \times (0,1)} \left( { |\nabla
                   m|^2 \over \ell \sqrt{Q - 1}} +\ell
                   \sqrt{Q - 1} \, (1-m_3^2) \right) \dx3-
                   2\int_{\T^2} |\nabla \overline m_3|
                   \dx2 \notag \\ 
             &&  \label{eq:bw} &\hspace{4cm}\le c_\gamma
                                 \frac{t}{\sqrt{Q-1}}\int_{\T^2} |\nabla \overline
                                 m_3| \dx2, 
     \end{align}
     where $0<c_\gamma<C_\gamma$ are constants (changing from line to
     line) which may depend only on $\gamma$.
\end{proposition}
We take a moment to interpret the statements (i)--(iv) in Proposition
\ref{prop:properties} above. Item $(i)$ shows that the magnetization
is approximately two-dimensional, i.e. independent of the thickness
variable. Moreover, since $|m|=1$, Item $(ii)$ means that the
magnetization is mostly perpendicular to the film (i.e.
$m \approx \pm e_3$). Furthermore, Item $(iii)$ is an estimate for the
total length of the domain walls in a unit cell. Back in the original,
physical variables, this quantity for the unit cell
$(0,L)^2 \times (0,T)$ is
\begin{align}\label{eq:oscil}
  W:= L\int_{\T^2} |\nabla \overline m_3| \d2 \lupref{eq:wl}\sim
  \frac{L^2}{l_{\rm ex}} \, e^{-\frac{2\pi l_{\rm
  ex} \sqrt{Q-1}}{T}} \sqrt{Q-1}. 
\end{align}
We expect that the stray field energy induces a repulsive interaction
of (nearest) neighboring domain walls and leads to an approximately
equidistant spacing of the walls. In view of $(iii)$, the typical
distance of neighboring walls should be
\begin{align}\label{eq:t2-i4b}
  S:= \frac{\text{length of the film}}{\#\text{ of walls on cross
  section}}\sim \frac{\ell}{\int_{\T_2} |\nabla' \overline m_3| \d2} \,
  l_{\rm ex} \lupref{eq:wl}\sim   \frac{l_{\rm ex} e^{\frac{2\pi
  l_{\rm ex} \sqrt{Q-1}}{T}}}{\sqrt{Q-1}}.
\end{align}
The exponential dependence of the typical distance between neighboring
walls on the inverse thickness in \eqref{eq:t2-i4b} was already
observed in ansatz-based computations in \cite{Kaplan1993111} for a
two-dimensional sharp interface model.  Item $(iv)$ in Proposition
\ref{prop:properties} indicates that domain walls approximate
Bloch walls of thickness proportional to
$\eps L = \frac{l_{\rm ex}}{\sqrt{Q-1}}$ for which the left hand side
of \eqref{eq:bw} is exactly zero. Note that \eqref{eq:bw} also implies
that $m$ approximately satisfies the optimal profile ODE in an
$L^2$-sense
\begin{align}\label{eq:scaling_profile_supercrit}
  \begin{aligned}
    \int_{\T^2 \times (0,1)}\left(\frac{|\nabla m_3|}{\sqrt{\ell
          \sqrt{Q - 1} \, (1-m_3^2)}} - \sqrt{\ell \sqrt{Q -
          1} \, (1-m_3^2)}\right)^2 \dx3  \\
    \lesssim \frac{t}{\sqrt{Q-1}}\int_{\T^2} |\nabla m_3| \dx2,
  \end{aligned}
\end{align}
with the convention $\frac{|\nabla m_3|}{\sqrt{1-m_3^2}}=0$ if
$|m_3|=1$.  Finally, we want to mention that the estimate of the
in-plane magnetization in Item $(i)$ is consistent with the in-plane
magnetization of a Bloch wall of length $W$ (see \eqref{eq:oscil}) and
thickness $\frac{l_{\rm ex}}{\sqrt{Q-1}}$.

Our third theorem addresses the transition where the cross-over from
constant to non-constant global minimizers occurs and which separates
the two previously described regimes.

\begin{theorem}[Critical scaling]\label{th:CRIT-E}
  Let $h_{\rm ext} = 0$ and let $\delta>0$ be as in Theorem
  \ref{th:SUPERcrit-E}. Then the following holds
    \begin{enumerate}
    \item \textbf{Cross-over of global minimizers} There are constants
      $c,C>0$ such that for $\ell,t,Q$ which satisfy
        \begin{align}
          Q>1, 
          && t \le \delta \min\left\{\sqrt{Q-1},
             \frac{1}{\sqrt{Q-1}}\right\} 
          && \text{and}&& \ell
                          \le c
                          \frac{e^{2\pi t^{-1} \sqrt{Q-1}}}{\sqrt{Q-1}} 
        \end{align}
        the renormalized energy $J$ is non-negative and
        $m\equiv\pm e_3$ are the only global minimizers, whereas for
        $\ell,t,Q$ which satisfy
        \begin{align}
          Q>1, && t \le \delta \min\left\{\sqrt{Q-1},
                  \frac{1}{\sqrt{Q-1}}\right\} && \text{and}&&                                                               \ell \ge C \frac{e^{2\pi t^{-1} \sqrt{Q-1}}}{\sqrt{Q-1}}
        \end{align}
        the minimal rescaled energy $\min J$ is strictly negative and
        minimizers cannot be constant.
      \item \textbf{$\Gamma$-convergence} For
        $\frac{t \log(\ell \sqrt{Q-1})}{\sqrt{Q-1}} = 2\pi$, $J$
        $\Gamma$-converges for $\ell\sqrt{Q-1} \to \infty$ towards
        \begin{align}
        J_*[m] =\begin{cases}
        0 &\text{if } m \in L^1(\T^2; \{{\pm e_3}\}),\\
        +\infty &\text{otherwise}.
        \end{cases}
        \end{align}
      \item \textbf{Compactness upon rescaling} For $C > 0$ and
        $\ell \sqrt{Q-1} \to \infty$, sequences with
        \begin{align}
          J[m] \leq \frac{C}{\log (\ell \sqrt{Q-1})}
        \end{align}
        are compact in $L^1(\T^2 \times (0,1))$ with a limit of the
        form $\chi_{(0,1)}\overline m$ where
        $\overline m \in BV^1(\T^2; \{{\pm e_3}\})$.
    \end{enumerate}
\end{theorem}

\subsection{Results for the simplified energy $F$}
\label{sec:results-for-F}

In this section, we will formulate results analogous to the ones in
the previous section, but for the reduced energy $F$. The relation
between the full energy $E$ and the reduced two-dimensional energy $F$
was explained heuristically in section \ref{sec:heuristic} and will be
made rigorous in section \ref{sec:reduction-E-to-F}. The reason to
formulate our results also in terms of $F$ is mainly expositional: We
believe that the main ideas are easier to understand when they are not
obscured by additional difficulties arising from the reduction to a
two-dimensional model and the stray-field energy approximation.

The behavior of the reduced energy $F$ in the subcritical regime is
summarized in the following theorem.
\begin{theorem}[Subcritical regime]\label{th:SUBcrit-F}
  Let $\lambda < \lambda_{c} := \frac{\pi}{2}$ and $F_{\eps,\lambda}$
  as defined in \eqref{eq:F}. Then the following holds
  \begin{enumerate}
  \item \textbf{Compactness:} Every sequence $(m_\eps)$ in
    $H^{1}(\T^2; \S^2)$ with
    \begin{align}
      {\limsup_{\eps \to 0} F_{\eps, \lambda}[m_\eps] < + \infty}
    \end{align}
    converges in $L^1(\T^2)$ (up to extracting a subsequence) towards
    a limit in $BV(\T^2;\{\pm e_3\})$.
  \item \textbf{$\Gamma$-convergence:} As $\eps \to 0$, the family of
    functionals $\{F_{\eps,\lambda}\}$ $\Gamma$-converges with respect
    to the $L^1(\T^2)$-topology towards $F_{*, \lambda}$, given by
    \begin{align}\label{eq:F_st}
      F_{*,\lambda}[m] = \left\{
        \begin{array}{ll}
          \left(1 - \frac{\lambda}{\lambda_{c}} \right)\int_{\T^2}
          |\nabla m_3|\d2 & \text{ for } m \in BV(\T^2;\{\pm e_3\})\\ 
          +\infty & \text{ otherwise.}
        \end{array}
        \right.
    \end{align}
\end{enumerate}
\end{theorem}

The next theorem is concerned with the minimal energy and the
structure of low energy states in the supercritical regime.
\begin{theorem}[Supercritical regime]\label{th:SUPERcrit-F}
  Let $\lambda_{c} := \frac{\pi}{2}$ and $F_{\eps,\lambda}$ as defined
  in \eqref{eq:F}. There are constants $\delta<1<K$ such that for
\begin{align}\label{eq:SUPERcrit-F-regime-th}
  0<\eps < K^{-\frac{\lambda}{\lambda - \lambda_c}}\qquad \text{and}
  \qquad \lambda_c < \lambda < \delta|\log \eps|,  
\end{align}
the minimal energy of the family of functionals
$\{F_{\eps, \lambda}\}$ satisfies
\begin{align}\label{eq:scaling_supercrit}
  -C\, \frac{\lambda \eps^{\frac{\lambda_c - \lambda}{\lambda}}}{|\log
  \eps|} \le \min F_{\eps,\lambda} \le -c\, \frac{\lambda
  \eps^{\frac{\lambda_c - \lambda}{\lambda}}}{|\log \eps|} 
\end{align}
for some universal constants $0<c<C$. Moreover, the profiles achieving
the optimal scaling can be characterized as follows. For any
$\gamma >0$ and all $m \in H^1(\T^2;\S^2)$ which satisfy
\begin{align}\label{eq:scaling-opt-th}
  F_{\eps, \lambda}[m] \le -\gamma \, \frac{\lambda
  \eps^{\frac{\lambda_c - \lambda}{\lambda}}}{|\log \eps|} , 
\end{align}
the quantities
\begin{align}\label{eq:scaling_var_supercrit}
    \begin{aligned}
      \int_{\T^2} |\nabla m_{3}| \dx2\le \int_{\T^2} \left(
        \frac{\eps}{2} |\nabla m|^2 + \frac{1-m_{3}^2}{2\eps} \right)
      \dx2 \le \frac{\lambda}{|\log \eps|} \int_{\T^2}|\nabla^{1/2}
      m_{3}|^2 \dx2
    \end{aligned}
\end{align}
agree to the leading order and scale as
$\eps^{\frac{\lambda_c - \lambda}{\lambda}}$, i.e. if $A$ and $B$ are
any of the three quantities in \eqref{eq:scaling_var_supercrit}, we
have
\begin{align}\label{eq:scaling_var_supercrit_2}
  c_{\gamma} \eps^{\frac{\lambda_c - \lambda}{\lambda}} \le A  \le
  C_{\gamma} \eps^{\frac{\lambda_c - \lambda}{\lambda}} \qquad
  \text{and} \qquad |A-B| \le \tilde C_{\gamma} \frac{\lambda}{|\log
  \eps|}A, 
\end{align}
for some positive constants $c_\gamma, C_\gamma$ and $\tilde C_\gamma$
which depend only on $\gamma$.
\end{theorem}

Under the assumptions of Theorem \ref{th:SUPERcrit-F}, statements
analogous to \eqref{eq:t2-i1} -- \eqref{eq:bw} in Proposition
\ref{prop:properties} hold as well, they are simple consequences of
the stronger statement \eqref{eq:scaling_var_supercrit_2}.

The next theorem addresses the structure of minimizers in a
neighborhood of the transition.

\begin{theorem}[Critical scaling]\label{th:crit-F}
  Let $\lambda_{c} := \frac{\pi}{2}$ and $F_{\eps,\lambda_c}$ as
  defined in \eqref{eq:F}. Then the following holds
  \begin{enumerate}
  \item \textbf{Cross-over of global minimizers: }There are two
    constants $0<\beta_1<1<\beta_2$ such that for
        \begin{align}\label{eq:beta_1}
          \lambda \le \lambda_-(\eps):=\lambda_c\left(1-\frac{|\log
          \beta_1|}{|\log \eps|}\right) 
        \end{align}
        the minimal energy $\min F_{\eps,\lambda}$ is zero and only
        attained by the constant configurations $m \equiv \pm e_3$,
        whereas for
        \begin{align}\label{eq:beta_2}
          \lambda \ge \lambda_+(\eps):= \lambda_c\left(1+\frac{|\log
          \beta_2|}{|\log \eps|}\right) 
        \end{align}
        the minimal energy is strictly negative and minimizers cannot
        be constant.
      \item \textbf{$\Gamma$-convergence:} As $\eps \to 0$, the family
        of functionals $\{F_{\eps,\lambda_c}\}$ $\Gamma$-converges
        with respect to the $L^1(\T^2)$-topology towards
        $F_{*, \lambda_c}$, given by
        \begin{align}\label{eq:F_st_crit}
          F_{*,\lambda_c}[m] = \begin{cases}
            0, & \text{if }m \in L^1(\T^2; \{\pm e_3\})\\
            +\infty & \text{otherwise},
        \end{cases}
        \end{align}
      \item \textbf{Lack of compactness:} There is a sequence
        $(m_\eps)$ in $H^{1}(\T^2; \S^2)$ with
        \begin{align}
        \limsup_{\eps \to 0} F_{\eps, \lambda_c}[m_\eps] \to 0
        \end{align}
        which is not precompact in $L^1(\T^2)$.
      \item \textbf{Compactness upon rescaling:} For every $C > 0$,
        every sequence $(m_{\eps})$ with
        \begin{align}
          F_{\eps,\lambda_c}[m_\eps] \leq C |\log \eps|^{-1}
        \end{align}
        converges in $L^1(\T^2)$ (up to extracting a subsequence) to a
        limit in $BV(\T^2;\{\pm e_3\})$.
    \end{enumerate}
\end{theorem}

Theorem \ref{th:crit-F} suggests that $|\log \eps|F_{\eps,\lambda_c}$
is the appropriate rescaling for the critical case. Unfortunately, it
seems not possible to obtain the $\Gamma$-limit of
${|\log \eps|}F_{\eps,\lambda_c}$ with our $H^{1/2}$-estimate
\eqref{eq:main_ineq_1d} of the following section, because the constant
$c_*$ there is not optimal.

We illustrate our results in a phase diagram (Figure \ref{fig:pd}). It
is not difficult to see that for each $0 < \eps < 1$ there is a sharp
threshold value $\lambda = \lambda_c(\eps) > 0$ at which a transition
from monodomain ($m \equiv +e_3$ or $m \equiv -e_3$) to multidomain
($m \not\equiv const$) states as global energy minimizers occurs, with
$\lambda_c(\eps)$ a Lipschitz-continuous function on
$[\delta, 1 - \delta]$ for every $0 < \delta < \frac12$ (for the
reader's convenience, a proof of this fact is presented in Lemma
\ref{le:cont-lambda} in the appendix).
While we do not know the precise value of $\lambda_c(\eps)$ for
$\eps > 0$, we show in Theorem \ref{th:crit-F} that
$\lambda_-(\eps) \le \lambda_c(\eps) \le \lambda_+(\eps)$ and
$\lim_{\eps \to 0} \lambda_c(\eps)= \frac{\pi}{2}$, i.e. the
definition above agrees with
$\lambda_c := \lambda_c(0) = \frac{\pi}{2}$. Furthermore, global
minimizers $m_{(\eps,\lambda)}$ of $F_{\eps,\lambda}$ with
$(\eps,\lambda)$ between the red (dashed) curves of the form
$\lambda(\eps) = \lambda_c + \gamma |\log \eps|^{-1}$ satisfy a
uniform bound of the form
$c \le \int_{\T^2}|\nabla m_{(\eps,\lambda),3}| \d2 \le C$, with
constants $C > c > 0$ depending only on the values of $\gamma > 0$ for
these curves.

\begin{figure}
  \centering\includegraphics[width=0.7\linewidth]{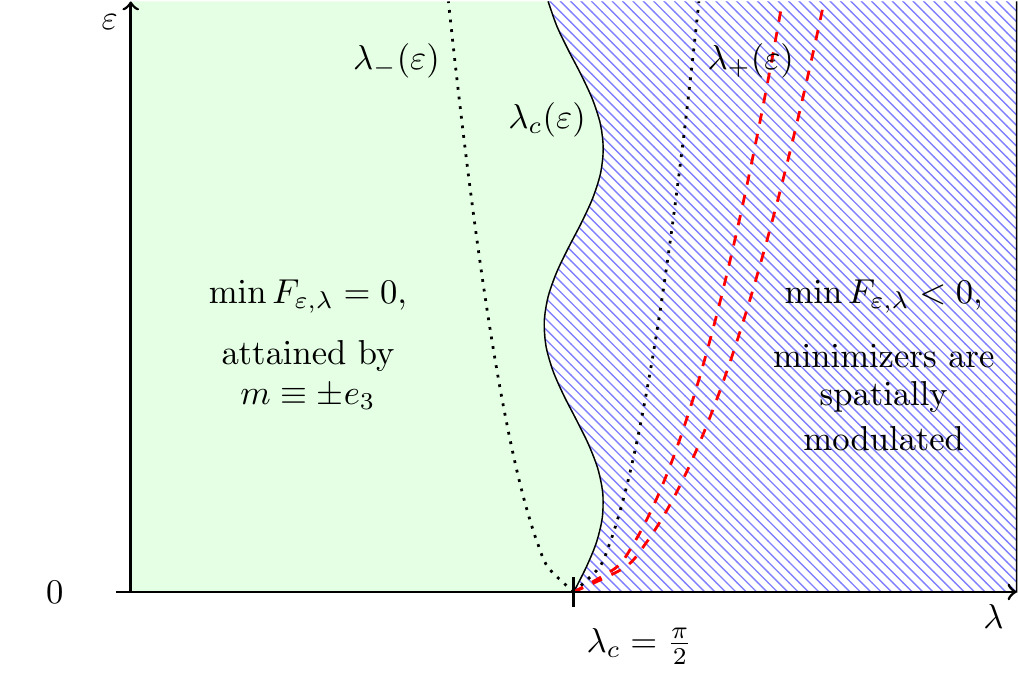}
    \caption{Sketch of the phase diagram for minimizers of
      $F_{\eps,\lambda}$ in terms of $\lambda > 0$ and
      $0 < \eps \ll 1$.}
    \label{fig:pd}
\end{figure}

\section{A bound on the homogeneous $H^{1/2}$-norm}
\label{sec:H12-bounds}

Since all three terms in $F$ contribute in highest order to the limit,
it is important to estimate the negative term
$\int_{\T^2}|\nabla^{1/2}m_3|^2 \d2$ with precise leading order
constant. In this section we will establish an upper bound for the
homogeneous $H^{1/2}$-norm which is the key ingredient for the lower
bounds (recall that the $H^{1/2}$-term occurs in the energy with a
negative sign).

We will prove the following
\begin{lem}\label{le:main_ineq}
  There is a universal constant $c_* \geq 1$ such that for every
  $f \in C^\infty(\T^2)$ and every $\eps > 0$ we have
	\begin{align}\label{eq:main_ineq_1d}
		\begin{aligned}
                  &\int_{\T^2} |\nabla^{1/2} f|^2\d2 \le
                  \frac{\eps}{2} \int_{\T^2}|\nabla f|^2 \d2\\
                  &\qquad +\frac{2}{\pi}\log
                  \left(c_*\max\left\{1,\min\left\{\frac{\|f\|_\infty}{\eps
                          \int_{\T^2}|\nabla f| \dd x},
                        \frac{1}{\eps}\right\}\right\}\right)\,
                  \|f\|_{\infty} \int_{\T^2}|\nabla f|\d2.
		\end{aligned}
	\end{align}
\end{lem}
In Lemma \ref{le:main_ineq}, we improve an inequality established in
\cite{DKO}. Expressed in our setting, the inequality proved in
\cite{DKO} asserts that for every $\delta >0$ there exists
$M_\delta \gg 1$ such that for all $\eps\le R$ and all
$f: \T^2 \to \R$, we have
\begin{align}\label{eq:dko}
  \sum_{k \in 2\pi\Z^2} \min\left\{\frac{1}{\eps},
  |k|,R|k|^2\right\}|\widehat f_k|^2 \le
  (1+\delta)\frac{2}{\pi}\log\left(\frac{2M_\delta R}{\eps}\right)
  \|f\|_{\infty} \int_{\T^2}|\nabla f|\d2. 
\end{align}
Note that \eqref{eq:main_ineq_1d} implies a similar estimate
\begin{align}\label{eq:main_ineq_1d_w}
\begin{aligned}
  &\int_{\T^2} |\nabla^{1/2} f|^2\d2 \le \frac{\eps}{2}
  \int_{\T^2}|\nabla f|^2 \d2 +\frac{2}{\pi}\log
  \left(c_*/{\eps}\right)\, \|f\|_{\infty} \int_{\T^2}|\nabla f|\d2
\end{aligned}
\end{align}
for all $\eps \leq 1$, which is weaker than
\eqref{eq:main_ineq_1d}. Estimate \eqref{eq:dko} is an inequality
  for a regularized $\mathring{H}^{1/2}$-norm, whereas
  \eqref{eq:main_ineq_1d_w} estimates the full
  $\mathring{H}^{1/2}$-norm, but needs an additional
  $\mathring{H}^1$-term.  It ceases to be optimal for functions which
oscillate significantly. Indeed, let $\alpha \in (0,1)$ and consider
functions $f$ with
\begin{align}\label{eq:osci}
  \int_{\T^2}|\nabla f| \dx2 \gtrsim \eps^{-\alpha}\|f\|_{\infty}.
\end{align}
Then the second term in \eqref{eq:main_ineq_1d} is smaller than the
second term in \eqref{eq:main_ineq_1d_w} by a factor of $(1-\alpha)$
for all $f$ which satisfy \eqref{eq:osci}. Asymptotic optimality in
the case of strong oscillation is crucial to obtain the results on the
supercritical regime.

The proof of Lemma \ref{le:main_ineq} uses similar ideas as in
\cite{DKO} and is based on a separate treatment of distinct
scales. However, our proof does not involve any Fourier Analysis.

\begin{proof}[Proof of Lemma \ref{le:main_ineq}]
  We will show that the following estimates hold for all
  $f \in C^{\infty}(\T^2)$ and all $0<r\le R$:
  \begin{align}
    \label{eq:main_ineq_small} \int_{\T^2} \int_{B_r}
    \frac{|f(x+z)-f(x)|^2}{|z|^3} \dz2 \dx2 
    &\le \pi r\int_{\T^2}
      |\nabla f|^2 \dx2,\\ 
    \label{eq:main_ineq_medium}	\int_{\T^2} \int_{B_R \setminus B_r}
    \frac{|f(x+z)-f(x)|^2}{|z|^3} \dz2 \dx2 
    &\le 8 \log(R/r) \|f\|_{\infty} \int_{\T^2} |\nabla f| \dx2,\\
    \label{eq:main_ineq_large} \int_{\T^2} \int_{\R^2 \setminus B_R}
    \frac{|f(x+z)-f(x)|^2}{|z|^3} \dz2 \dx2 
    &\le \frac{2\pi\|f\|_{\infty}}{R}\min\left\{4\|f\|_\infty,
      \int_{\T^2} |\nabla f| \dx2\right\}. 
    \end{align}
    The claim of the Lemma will follow by adding
    \eqref{eq:main_ineq_small} -- \eqref{eq:main_ineq_large} and a
    suitable choice of $r$ and $R$. Before we start with the proofs of
    estimates \eqref{eq:main_ineq_small} --
    \eqref{eq:main_ineq_large}, we first record an auxiliary
    inequality for further use. By the Fundamental Theorem of
    Calculus, Jensen's inequality and Fubini's theorem we get
	\begin{align}\label{eq:Lp_translation}
	\begin{aligned}
          &\int_{\T^2} |f(x+z) - f(x)|^p \dx2 = \int_{\T^2}
          \left|\int_0^1 \nabla f(x+sz) \cdot z\, \ds1\right|^p \dx2\\
          &\le \int_0^1 \int_{\T^2} \left| \nabla f(x+sz) \cdot
            z\right|^p \dx2 \ds1 \le \int_{\T^2}|\nabla f(x) \cdot
          z|^p \dx2
	\end{aligned}
	\end{align}
	for all $z \in \R^2$ and all $1 \le p < \infty$. In order to
        prove \eqref{eq:main_ineq_small}, we use Fubini's Theorem and
        apply \eqref{eq:Lp_translation} with $p=2$ to get
	\begin{align}\label{eq:main_ineq_small_p1}
          \int_{\T^2} \int_{B_r} \frac{|f(x+z) - f(x)|^2}{|z|^3} \dz2
          \dx2 \lupref{eq:Lp_translation}{\le} \int_{B_r} \int_{\T^2}
          \frac{|\nabla f(x)\cdot z|^2}{|z|^3} \dx2 \dz2. 
	\end{align}
	We apply Fubini's Theorem again and evaluate the integral with
        respect to $z$ in polar coordinates
	\begin{align}\label{eq:main_ineq_small_p2}
		\begin{aligned}
                  \int_{B_r} \int_{\T^2} \frac{|\nabla f(x)\cdot
                    z|^2}{|z|^3} \dx2 \dz2 &= \left(\int_{0}^r
                    \int_{0}^{2\pi} \cos^2 \phi \, \dd \phi \dd \rho
                  \right) \left(\int_{\T^2}|\nabla f(x)|^2
                    \dx2\right)\\
                  &= \pi r\int_{\T^2}|\nabla f|^2\d2.
		\end{aligned}
	\end{align}
	Together, \eqref{eq:main_ineq_small_p1} and
        \eqref{eq:main_ineq_small_p2} yield the first estimate
        \eqref{eq:main_ineq_small}.
	
	For the estimate involving intermediate distances
        \eqref{eq:main_ineq_medium}, we use Fubini's Theorem (twice)
        and \eqref{eq:Lp_translation} with $p=1$ to conclude
	\begin{align}
	\begin{aligned}
	\label{eq:int2_a}
	\int_{\T^2} \int_{B_R \setminus B_r} \frac{|f(x+z) -
          f(x)|^2}{|z|^3} \dz2 \dx2 &\lupref{eq:Lp_translation}{\le}
        2\|f\|_\infty \int_{\T^2} \int_{B_R \setminus B_r}
        \frac{|\nabla f(x) \cdot z|}{|z|^3} \dz2 \dx2.
	\end{aligned}
	\end{align}
	As in the proof of \eqref{eq:dko} in \cite{DKO}, we evaluate
        the inner integral in polar coordinates
	\begin{align}\label{eq:int2_inner}
          \int_{B_R \setminus B_r} \frac{|\nabla f(x) \cdot
          z|}{|z|^3}\dz2 = \int_r^R \int_0^{2\pi} \frac{|\nabla
          f(x)|\;|\cos \phi|}{\rho} \dphi1 \drho1 =4 \log
          \left(\frac{R}{r}\right)|\nabla f(x)|. 
	\end{align}
	Inserting \eqref{eq:int2_inner} into \eqref{eq:int2_a} yields
        the claim \eqref{eq:main_ineq_medium}.
	
	In order to prove \eqref{eq:main_ineq_large}, we first show
	\begin{align}\label{eq:poincare}
          \int_{\T^2}|f(x + z) - f(x)| \dx2 \le
          \min\left\{2\|f\|_\infty, \frac{1}{2} \int_{\T^2} |\nabla f|
          \dx2 \right\} \quad\text{for all }z \in \R^2. 
	\end{align}
	Indeed, the upper bound of $2\|f\|_\infty$ in
        \eqref{eq:poincare} is trivial. Furthermore, since $f$ is
        periodic, it is sufficient to show the second upper bound in
        \eqref{eq:poincare} only for
        $z \in \left(-\tfrac12,\tfrac12\right)^2$. Thus the second
        bound in \eqref{eq:poincare} follows from
        \eqref{eq:Lp_translation} with $p=1$
	\begin{align*}
          \int_{\T^2}|f(x + z) - f(x)| \dx2
          \lupref{eq:Lp_translation}\le \int_{\T^2} |\nabla f(x) \cdot
          z| \dx2 \le \frac{1}{2}\int_{\T^2} |\nabla f(x)| \dx2 
	\end{align*}
	so that the proof of \eqref{eq:poincare} is complete. With
        \eqref{eq:poincare} at hand, estimate
        \eqref{eq:main_ineq_large} now follows by direct integration
	\begin{align}
	\begin{aligned}\label{eq:int3}
          \int_{\T^2} \int_{\R^2 \setminus B_R} \frac{|f(x+z) -
            f(x)|^2}{|z|^3} \dz2 \dx2 &\le 2\|f\|_{L^\infty}\int_{\R^2
            \setminus B_R} \int_{\T^2}\frac{|f(x+z) - f(x)|}{|z|^3}
          \dx2  \dz2\\ 
          &{\lupref{eq:poincare}\le} \frac{2\pi\|f\|_\infty}{R}
          \min\left\{4\|f\|_\infty, \int_{\T^2} |\nabla f| \dx2
          \right\}.
	\end{aligned}
	\end{align}

	It remains to prove \eqref{eq:main_ineq_1d}, for which we use
        the real-space representation of the homogeneous
        $H^{1/2}$-norm
	\begin{align}\label{eq:h12_real}
          \int_{\T^2} |\nabla^{1/2} f|^2 \d2 =
          \frac{1}{4\pi}\int_{\T^2} \int_{\R^2} \frac{|f(x+z) -
          f(x)|^2}{|z|^3} \dz2 \dx2. 
	\end{align}
	A proof of \eqref{eq:h12_real} is given in the appendix for
        completeness of the presentation. Without loss of generality,
        we may assume that $f$ is not equal to a constant in
        $\T^2$. Adding \eqref{eq:main_ineq_small} --
        \eqref{eq:main_ineq_large} to estimate the right hand side of
        \eqref{eq:main_ineq_1d}, we therefore get
	\begin{align}\label{eq:int_sum}
		\begin{aligned}
                  \int_{\T^2} |\nabla^{1/2} f|^2\d2 &\le
                  \frac{r}{4}\int_{\T^2}|\nabla f|^2\d2\\
                  &+ \left(\frac{2}{\pi}\log \left(\frac{R}{r}\right)
                    +
                    \frac{1}{2R}\min\left\{\frac{4\|f\|_\infty}{\int_{\T^2}|\nabla 
                        f| \dd x}, 1\right\}\right)\|f\|_\infty
                  \int_{\T^2}|\nabla f|\d2.
		\end{aligned}
	\end{align}
        For $r=2\eps$ and
        $R=\max\left\{2 \eps,\min\left\{\frac{4
              \|f\|_\infty}{\int_{\T^2}|\nabla f| \dd x},
            1\right\}\right\}$
        the claim \eqref{eq:main_ineq_1d} now follows from
        \eqref{eq:int_sum}.
\end{proof}

\section{Proofs for the reduced energy $F$}
\label{sec:proofs-F}
In this section we give the proofs of the Theorems involving the
reduced energy $F$.  The proof of Theorem \ref{th:SUBcrit-F} is a
direct consequence of Lemma \ref{le:subcrit-lb-comp} and Lemma
\ref{le:constr}. Similarly the proof of Theorem \ref{th:SUPERcrit-F}
follows immediately from Lemma \ref{le:F-SUPERcrit-lb} and Lemma
\ref{le:F-SUPERcrit-ub}. Finally, the proof of Theorem \ref{th:crit-F}
is presented at the end of this section.

\subsection{Proof of Theorem \ref{th:SUBcrit-F}}
\begin{lem}[Lower bound and compactness in the subcritical
  regime]\label{le:subcrit-lb-comp}
  Let $\lambda < \lambda_{c} := \frac{\pi}{2}$ and $F_{\eps,\lambda}$
  as defined in \eqref{eq:F}. Then every sequence $(m_\eps)$ in
  $H^{1}(\T^2; \S^2)$ with
  \begin{align} {\limsup_{\eps \to 0} F_{\eps, \lambda}[m_\eps] < +
      \infty}
\end{align}
converges in $L^1(\T^2;\R^3)$ (up to extracting a subsequence) towards
a limit in $BV(\T^2;\{\pm e_3\})$. Furthermore, for every sequence
$(m_\eps)$ in $L^1(\T^2;\S^2)$ with $m_\eps \to m$ for some $m$ in
$L^1(\T^2;\R^3)$ we have
\begin{align}\label{eq:subcrit-liminf}
\liminf_{\eps \to 0} F_{\eps,\lambda}[m_\eps] \ge \left\{
\begin{array}{ll}
  \displaystyle\left(1 - \frac{\lambda}{\lambda_{c}} \right)\int_{\T^2} |\nabla
  m_3|\d2 & \text{ for } m \in BV(\T^2;\{\pm e_3\}),\\ 
  +\infty & \text{ otherwise.}
\end{array}
\right.
\end{align}
\end{lem}

\begin{proof}[Proof of Lemma \ref{le:subcrit-lb-comp}]
  We first show that for sufficiently small $\eps > 0$ we have
    \begin{align}\label{eq:F_SUBcrit_lb}
      F_{\eps,\lambda}[m]\ge \left(1 - \frac{\lambda|\log
      c\eps|}{\lambda_{c}|\log \eps|}\right) \int_{\T^2} |\nabla m_3
      |\d2 
    \end{align}
    for all $m \in H^1(\T^2; \S^2)$, where $c>$ is a universal
    constant. Indeed, for $\lambda < \lambda_c$ we expect
    $\int_{\T^2}|\nabla m_3| \dd x$ to be small and hence it is
    sufficient to use Lemma \ref {le:main_ineq} for $m_3$ in the
    weaker form \eqref{eq:main_ineq_1d_w}. Recalling that
    $\|m_3\|_\infty \le 1$ and $\lambda_c = \frac{\pi}{2}$, we get
    \begin{align}\label{eq:main_ineq_1d_SUBcrit_simple}
      &\frac{\lambda}{|\log \eps|}\int_{\T^2} |\nabla^{1/2} m_3|^2\d2
        \lupref{eq:main_ineq_1d_w}\le \frac{\lambda}{|\log \eps|}
        \int_{\T^2}\frac{\eps}{2}|\nabla m_3|^2 \d2
        +\frac{\lambda}{\lambda_c}\frac{\log \left(
        c_*/\eps\right)}{|\log \eps|}\,\int_{\T^2}|\nabla m_3|\d2. 
    \end{align}
    We also use the constraint $|m|=1$ in the form of the well-known
    estimate
    \begin{align}\label{eq:modica_mortola_p1}
      |\nabla m_3| \lupref{eq:A2}\le \frac{\eps}{2} |\nabla m|^2 +
      \frac1{2\eps}(1-m_3^2). 
    \end{align}
    which is obtained by differentiating $|m|^2=1$ and applying
    Young's inequality (see \eqref{eq:A2} in the Appendix for a
    proof). Now the claimed lower bound \eqref{eq:F_SUBcrit_lb}
    follows from \eqref{eq:main_ineq_1d_SUBcrit_simple} and
    \eqref{eq:modica_mortola_p1}:
    \begin{align}\label{eq:F_SUBcrit_lb_p1}
    \begin{aligned}
      F_{\eps,\lambda}[m] &= \int_{\T^2} \left(\frac{\eps}2 |\nabla
        m|^2 + \frac{1}{2\eps}(1 - m_3^2)\right)\d2-
      \frac{\lambda}{|\log \eps|}\int_{\T^2}|\nabla^{1/2} m_3|^2\d2\\
      &\lupref{eq:main_ineq_1d_SUBcrit_simple}\ge \left(1 -
        \frac{\lambda}{|\log \eps|}\right) \int_{\T^2}
      \left(\frac{\eps}2 |\nabla m|^2 + \frac{1}{2\eps}(
        1 - m_3^2)\right)\d2\\
      & \qquad -\frac{\lambda}{\lambda_c}\frac{\log \left(
          c_*/\eps\right)}{|\log \eps|}\,\int_{\T^2}|\nabla m_3|\d2\\
      &\lupref{eq:modica_mortola_p1}\ge\left(1
        -\frac{\lambda}{\lambda_c}\frac{\log \left(
            e^{\lambda_c}c_*/\eps\right)}{|\log \eps|}\right)
      \,\int_{\T^2} \left(\frac{\eps}2 |\nabla m|^2 +
        \frac{1}{2\eps}(1 - m_3^2)\right)\d2\\ 
      &\lupref{eq:modica_mortola_p1}\ge\left(1
        -\frac{\lambda}{\lambda_c}\frac{\log \left(
            e^{\lambda_c}c_*/\eps\right)}{|\log \eps|}\right)
      \,\int_{\T^2}|\nabla m_3|\d2.
    \end{aligned}
    \end{align}
    
    Let $m_\eps$ be a sequence in $H^1(\T^2; \S^2)$ with bounded
    energy
    $\limsup_{\eps \to 0} F_{\eps, \lambda}[m_\eps] < + \infty$.  From
    the penultimate line in \eqref{eq:F_SUBcrit_lb_p1}, $|m_\eps|=1$
    and $\lambda < \lambda_c$ we obtain
    \begin{align}\label{eq:F_lb_b}
      0 = \limsup_{\eps \to 0} \eps F_{\eps, \lambda} 
      [m_\eps] \lupref{eq:F_SUBcrit_lb_p1}\ge \frac{1}{2}\left(1 -
      \frac{\lambda}{\lambda_{c}}\right) \limsup_{\eps \to
      0} \int_{\T^2} \left( m_{\eps,1}^2 + m_{\eps,2}^2 \right) \d2,
    \end{align}
    implying that the first two components $m_{\eps,1}$ and
    $m_{\eps,2}$ converge to zero in $L^2(\T^2)$ as $\eps \to 0$.
    Moreover, \eqref{eq:F_SUBcrit_lb_p1} yields a uniform bound for
    $m_{\eps,3}$ in $BV$, which by compactness of $BV(\T^2)$ in
    $L^1(\T^2)$ implies the existence of a convergent
    subsequence. Passing to another subsequence, we may assume that
    $m_\eps$ converges pointwise almost everywhere. Since
    $|m_\eps|=1$, we obtain $m = \pm e_3$ almost everywhere.
 
    For the liminf inequality \eqref{eq:subcrit-liminf}, we may
    assume without loss of generality that \\
    $\liminf_{\eps \to 0}F_{\eps,\lambda}[m_\eps]< + \infty$.  But
    then there is a subsequence (not relabelled) such that
    $\limsup_{\eps \to 0}F_{\eps,\lambda}[m_\eps]< + \infty$ and by
    the compactness result and uniqueness of the limit we have
    $m \in BV(\T^2;\{\pm e_3\})$. Now the liminf inequality follows
    directly from \eqref{eq:F_SUBcrit_lb}, the fact that
    $\lim_{\eps \to 0}\frac{\lambda|\log c\eps|}{\lambda_{c}|\log
      \eps|} = \frac{\lambda}{\lambda_{c}}<1$
    and lower semi-continuity of the $BV$-seminorm.
\end{proof}

Before we begin with the construction of the upper bound, we define a
family of asymptotically optimal profiles and record some of their
properties (see Fig. \ref{fig:tikz:xi}).

\begin{figure}[h]
    \includegraphics[width=0.9\linewidth]{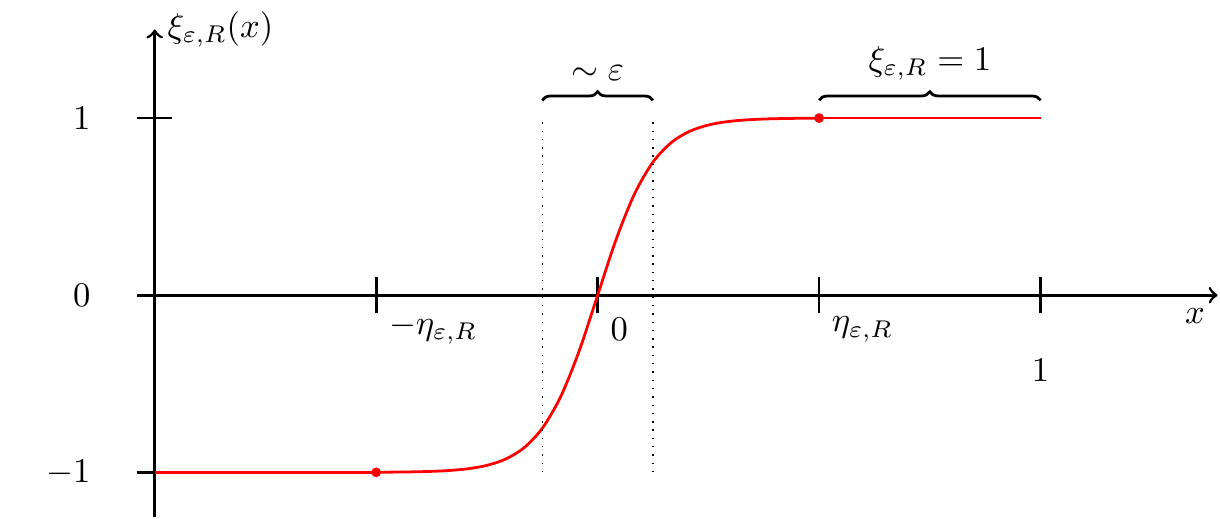}%
    \caption{Family of asymptotically optimal profiles $\xi_{\eps,R}$}%
    \label{fig:tikz:xi}%
\end{figure}%

\begin{lem}[Estimates for a family of asymptotically optimal
  profiles]\label{le:opt}
  For $R \in (0, + \infty]$ and $\eps>0$, let
  $\xi_{\eps,R} : \R \to [-1,1]$ be the unique solution to the initial
  value problem
    \begin{align}\label{eq:xi_eps_R}
      \xi_{\eps,R}(0) = 0 \qquad \text{and} \qquad \xi_{\eps,R}' =
      \frac{1}{\eps} (1-\xi_{\eps,R}^2)^{1/2}\left(1-\xi_{\eps,R}^2 +
      \left( \frac{\pi\eps}{2 R} \right)^2\right)^{1/2}. 
    \end{align}
    Then $\xi_{\eps,R}$ is non-decreasing and satisfies
    \begin{align}\label{eq:xi_eps_R_decay}
    \xi_{\eps,R}(x)=-\xi_{\eps,R}(-x) && \text{and}&&
    |\xi_{\eps,R}(x) - {\rm sign}(x)| \le 2e^{-2|x|/\eps}.
    \end{align}
    Moreover, $\xi_{\eps,R}(x)=1$ if $x \ge \eta_{\eps,R}$ and
    $\xi_{\eps,R}(x)=-1$ if $x\le -\eta_{\eps,R}$, for some
    $\eta_{\eps,R} \in (0,R]$. The contribution to the local part of
    the energy may be estimated as
    \begin{align}\label{eq:constr_local_energy}
      \frac12 \int_{-\eta_{\eps,R}}^{\eta_{\eps,R}} \left( 
      \frac{\eps |\xi_{\eps,R}'|^2}{1-\xi_{\eps,R}^2} +
      \frac{1-\xi_{\eps,R}^2}{\eps} \right) \dd x \le 2 +
      \frac{\pi^2\eps}{4R}. 
    \end{align}
    Lastly, there is a universal constant $c >0$ such that
    \begin{align}\label{eq:constr_nonlocal_energy}
      \int_{-X}^{X} \int_{-X}^X
      \frac{|\xi_{\eps,R}(x)- \xi_{\eps,R}(y)|^2}{|x-y|^2}
      \dd x \dd y \ge 8 \log(c X/\eps) \quad \text{for }X \ge
      2\eps.  
    \end{align}
\end{lem}

\begin{proof}
  The existence, uniqueness and monotonicity of $\xi_{\eps,R}$ follows
  by direct integration. In particular, for $R< + \infty$, there
  exists a unique real number $\eta_{\eps,R} >0$, such that the
  solution of \eqref{eq:xi_eps_R} satisfies
  $\xi_{\eps,R}(s) \in (-1,1)$ for
  $s \in (-\eta_{\eps,R}, \eta_{\eps,R})$ and
  $\xi_{\eps,R}(\pm \eta_{\eps,R}) = \pm 1$. For $R=+\infty$, we have
  $\xi_{\eps,\infty}=\tanh(\cdot/\eps)$ and the claim follows for
  $\eta_{\eps,\infty}=+\infty$. Estimate \eqref{eq:xi_eps_R_decay}
  follows immediately from $\xi_{\eps,\infty} \le \xi_{\eps,R} \le 1$
  for $x\ge 0$.

  We will now show that $\eta_{\eps,R} \le R$ holds. Indeed, since
  $\xi_{\eps,R}$ is strictly monotone on
  $(-\eta_{\eps,R},\eta_{\eps,R})$, the inverse function theorem
  yields
    \begin{align}\label{eq:eta_eps}
    \begin{aligned}
      \eta_{\eps,R} &= \lim_{s \to 1^-} \xi_{\eps,R}^{-1}(s) = \int_0^{1}
      \left(\xi_{\eps,R}^{-1}\right)'(s) \ds1\\
      &= \int_0^1 \frac{\eps}{\sqrt{\left(1-s^2\right) \left(1-s^2 +
            \frac{\pi^2\eps^2}{4 R^2}\right)} } \ds1 \le \int_0^1
      \frac{2 R}{\pi \sqrt{1-s^2 } } \ds1 = R.
    \end{aligned}
    \end{align}
    We turn to the proof of \eqref{eq:constr_local_energy}. By
    \eqref{eq:xi_eps_R}, we have
    \begin{align}
    \begin{aligned}\label{eq:constr_local_energy_integrand}
      & \frac{\eps |\xi_{\eps,R}'|^2}{1-\xi_{\eps,R}^2} +
      \frac{1-\xi_{\eps,R}^2}{\eps} \lupref{eq:xi_eps_R}= 2
      \xi_{\eps,R}' + \frac{1}{\eps}\left(\sqrt{1-\xi_{\eps,R}^2 +
          \left(\frac{\pi \eps}{2 R}\right)^2} -
        \sqrt{1-\xi_{\eps,R}^2}\right)^2\\
      &= 2 \xi_{\eps,R}' + \frac{1}{\eps}\left(\int_{0}^{\frac{\pi
            \eps}{2 R}} \frac{s}{\sqrt{1-\xi_{\eps,R}^2 + s^2}} \dd s
      \right)^2 \le 2 \xi_{\eps,R}' + \frac{\pi^2\eps}{4 R^2}
    \end{aligned}
    \end{align}
    and thus \eqref{eq:constr_local_energy} follows from
    \eqref{eq:constr_local_energy_integrand} by integration.

    It remains to prove \eqref{eq:constr_nonlocal_energy}. By symmetry
    of $\xi_{\eps,R}$ we have
    \begin{align}
    \begin{aligned}\label{eq:constr_nonlocal_energy_proof1}
      &\int_{-X}^{X} \int_{\{\eps \le |z| \le X\}\cap \{|x+z|\le X\}}
      \frac{|\xi_{\eps,R}(x+z)- \xi_{\eps,R}(x)|^2}{|z|^2} \dd z \dd x
      \\ 
      &= 2\int_{-X}^{0}\int_{\{\eps \le |z| \le X\}\cap \{|x+z|\le
        X\}} \frac{|\xi_{\eps,R}(x+z)- \xi_{\eps,R}(x)|^2}{|z|^2} \dd
      z \dd x
    \end{aligned}
    \end{align}
    As it turns out, it is sufficient to restrict the integral to a
    set where $|\xi_{\eps,R}(x+z)- \xi_{\eps,R}(x)| \gtrsim 1$ to
    obtain the correct leading order behavior
    \begin{align*}
      &\int_{-X}^{0}\int_{\{\eps \le |z| \le X\}\cap \{|x+z|\le X\}}
        \frac{|\xi_{\eps,R}(x+z)- \xi_{\eps,R}(x)|^2}{|z|^2} \dd z \dd
        x\\ 
      &\qquad \ge \int_{-X}^{0}\int_{\eps}^{x+X}
        \frac{|\xi_{\eps,R}(y)- \xi_{\eps,R}(x)|^2}{|y-x|^2} \dd y \dd
        x. 
    \end{align*}
    Since $|1-\xi_{\eps,R}|$ decays exponentially with rate $1/\eps$,
    we split the integral into the leading order and a lower order
    correction
    \begin{align}
    \begin{aligned}\label{eq:constr_nonlocal_energy_3}
      \int_{-X}^{0}\int_{\eps}^{x+X} &\frac{|\xi_{\eps,R}(y)-
        \xi_{\eps,R}(x)|^2}{|y-x|^2} \dd y \dd x =
      \int_{-X}^{0}\int_{\eps}^{x+X} \frac{4}{|y-x|^2} \dd y \dd x \\ 
      &- \int_{-X}^{0}\int_{\eps}^{x+X} \frac{4-|\xi_{\eps,R}(y)-
        \xi_{\eps,R}(x)|^2}{|y-x|^2} \dd y \dd x.
    \end{aligned}
    \end{align}
    The first term on the right hand side of
    \eqref{eq:constr_nonlocal_energy_3} yields
    \begin{align}
    \begin{aligned}\label{eq:constr_nonlocal_energy_proof4}
      \int_{-X}^{0}\int_{\eps}^{x+X} \frac{1}{|y-x|^2} \dd y \dd x =
      \log\left(\frac{\eps+X}{\eps}\right) - 1.
    \end{aligned}
    \end{align}
    Thus, it is sufficient to show that the second term on the right
    hand side of \eqref{eq:constr_nonlocal_energy_3} is bounded
    independently of $\eps$. Indeed, using the exponential decay of
    $|1-\xi_{\eps,R}|$, we get
    \begin{align}
    \begin{aligned}\label{eq:constr_nonlocal_energy_5}
      \int_{-X}^{0}\int_{\eps}^{x+X} \frac{4-|\xi_{\eps,R}(y)-
        \xi_{\eps,R}(x)|^2}{|y-x|^2} \dd y \dd x \lesssim
      \int_{0}^\infty \int_{1}^\infty \frac{e^{-2x} +
        e^{-2y}}{|x+y|^2} \dd x \dd y \lesssim 1.
    \end{aligned}
    \end{align}
    Together, \eqref{eq:constr_nonlocal_energy_proof1} --
    \eqref{eq:constr_nonlocal_energy_5} yield the claim
    \eqref{eq:constr_nonlocal_energy}.
\end{proof}

For the special case $\lambda = 0$, the $\Gamma$-convergence and in
particular the construction of a recovery sequence is a classical
result, relying on the optimal one-dimensional transition profiles to
smooth out the jump discontinuity in the limit configuration
\cite{ABV}.  As it turns out, this construction also works for
$\lambda >0$, where $F_{\eps, \lambda}$ is nonlocal. We will use a
construction based on the nearly optimal profile $\xi_{\eps,R}$ from
Lemma \ref{le:opt}. As the calculations for the local part of the
energy are well-known, our focus is on the contribution of the
homogeneous $H^{1/2}$-norm. Recall that we need to prove a lower bound
for the $H^{1/2}$-norm in order to obtain an upper bound for
$F$.

\begin{lem}[Construction of a recovery sequence in the subcritical and
  critical regime]\label{le:constr}
  Let $\lambda \le \lambda_c$ and $m \in L^1(\T^2;\S^2)$. Then there
  is a sequence $(m_\eps)$ in $H^1(\T^2;\S^2)$ with
    \begin{align}
      \limsup_{\eps \to 0}F_{\eps,\lambda}[m_\eps] \le F_{*,\lambda}[m],
    \end{align}
    where $F_{\eps,\lambda}$ is given by \eqref{eq:F}, and
    $F_{*,\lambda}$ is given by \eqref{eq:F_st} for
    $\lambda < \lambda_c$ or \eqref{eq:F_st_crit} for
    $\lambda = \lambda_c$, respectively.
\end{lem}

\begin{proof}[Proof of Lemma \ref{le:constr}]
  It is sufficient to prove the limsup inequality under the additional
  assumption that $m = (\chi_A -\chi_{\T^2 \setminus A}){e_3}$ for a
  set $A \subset \T^2$ with smooth boundary. By standard density results (see
  e.g. \cite[Prop. 12.20]{maggi2012sets}) and a diagonal argument, the
  limsup inequality then extends to arbitrary $A\subset \T^2$ with
  finite perimeter for $\lambda < \lambda_c$ or to measurable
  $A \subset \T^2$ for the $\lambda = \lambda_c$ case. Since
  $F_{*,\lambda}[m] = +\infty$ for $m \notin BV(\T^2, \{\pm e_3\})$
  when $\lambda<\lambda_c$ or for $m \notin L^1(\T^2, \{\pm e_3\})$)
  when $\lambda=\lambda_c$, this yields the claim.

  Our strategy is to adapt the optimal profiles $\xi_{\eps,R}$ from
  Lemma \ref{le:opt} to the two-dimensional setting by means of the
  signed distance function $d$, given by
  $d(x) := \dist(x, A^c) - \dist(x,A)$. Without loss of generality, we
  may assume $0 < |A|<1$ (otherwise take $m_\eps \equiv \pm e_3$). To
  fix the notation, let $\nu: \partial A \to \R^2$ denote the smooth
  inward normal to $A$ and $\tau: \partial A \to \R^2$,
  $\tau = \nu^{\perp}$ denote a smooth tangent vector field to
  $\partial A$ obtained by a counter-clockwise $90^\circ$ rotation of
  $\nu$. As $\partial A$ is assumed to be smooth, there exists a
  tubular neighborhood
  $\left(\partial A\right)_R = \bigcup_{x\in \partial A} B_R(x)
  \subset \T^2$ for some $R>0$ such that the projection
  $p:\left(\partial A\right)_R \to \partial A$,
  $p(x) := \argmin_{y \in \partial A} |x-y|$ is single-valued and
  hence well-defined. Furthermore, the projection $p$ and the signed
  distance function $d$ are smooth on $(\partial A)_R$ and the
  identity
  \begin{align}\label{eq:pg}
      x = p(x) + d(x) \nu(p(x))
  \end{align}
  holds for all $x \in (\partial A)_R$, see e.g. \cite[Lemma
  14.16]{gilbargtrudinger}.

  With the necessary notation at hand, we define the recovery sequence
  by
  \begin{align}\label{eq:m_eps}
    m_{\eps}(x) = \xi_{\eps,R}(d(x)) e_3 + \sqrt{1-\xi_{\eps,R}^2(d(x))}
    \, \tau(p(x)). 
  \end{align}
  Recall that $\eta_{\eps,R} \leq R$, (see \eqref{eq:eta_eps}) and
  hence the function $m_\eps$ is Lipschitz continuous and piecewise smooth.

  It is easy to see that $m_\eps \to m$ in $L^1(\T^2)$, and for the
  sake of completeness, we briefly mention how to compute the
  contribution of the local energy terms. Since $\tau \perp e_3$,
  $(\tau \circ p) \cdot \nabla (\tau \circ p) = 0$ and $|\nabla d|=1$
  almost everywhere, the squared gradient of $m_\eps$ can be estimated
  by 
\begin{align}\label{eq:m_eps_grad}
\begin{aligned}
  |\nabla m_\eps|^2 
  &= \frac{|\xi_{\eps,R}'(d)|^2}{1-\xi_{\eps,R}^2(d)} +
  (1-\xi_{\eps,R}^2(d))|\nabla (\tau \circ p) |^2\le
  \frac{|\xi_{\eps,R}'(d)|^2}{1-\xi_{\eps,R}^2(d)} + C_{A},
\end{aligned}
\end{align}
where $C_{A} > 0$ is a constant that depends only on $A$ for all
$R \leq R_A$, where $R_A > 0$ depends only on $A$. In the following,
$C_A$ may change from line to line.

We next employ the co-area formula, to reduce to the one-dimensional
case:
\begin{align}\label{eq:rec_rhs_a}
\begin{aligned}
  &\int_{\T^2} \left(\frac{\eps}{2}|\nabla m_\eps|^2 +
    \frac{1}{2\eps}(1-{m_{\eps,3}^2} ) \right)\d2\\
  &\lupref{eq:m_eps_grad}\le\int_{(\p A)_{\eta_{\eps,R}}}
  \left(\frac{\eps |\xi'_{\eps,R}(d)|^2}{2(1-\xi_{\eps,R}^2(d))} +
    \frac{1}{2\eps}(1-\xi_{\eps,R}^2(d))\right)\d2 + \eps C_{A}\\
  &\le\int_{-\eta_{\eps,R}}^{\eta_{\eps,R}} \left(\frac{\eps
      |\xi'_{\eps,R}(s)|^2}{2(1-\xi_{\eps,R}^2(s))} +
    \frac{1}{2\eps}(1-\xi_{\eps,R}^2(s))\right)\,\HH^1(\{d(x) = s
  \})\ds1 + \eps C_{A}.
\end{aligned}
\end{align}
Inserting the estimate for the one-dimensional profile from Lemma
\ref{le:opt}, we obtain
\begin{align}\label{eq:rec_rhs_b}
\begin{aligned}
  &\int_{-\eta_{\eps,R}}^{\eta_{\eps,R}} \left(\frac{\eps
      |\xi'_{\eps,R}(s)|^2}{2(1-\xi_{\eps,R}^2(s))} +
    \frac{1}{2\eps}(1-\xi_{\eps,R}^2(s))\right)\,\HH^1(\{d(x) = s
  \})\ds1\\
  &\lupref{eq:constr_local_energy}\le \sup_{-\eta_{\eps,R} \le s \le
    \eta_{\eps,R}}\HH^1(\{d(x) = s \})\left(2 +
    O\left(\frac{\eps}{R}\right)\right).
\end{aligned}
\end{align}
Since $\partial A$ and the signed distance function $d$ are smooth in
$(\partial A)_R$, we have
\begin{align}\label{eq:lim-AH}
    \lim_{s \to 0}\HH^1(\{d(x) = s \}) = \HH^1(\partial A).
\end{align}
In the limit
$\eps \to 0$, then $R \to 0$, estimates \eqref{eq:rec_rhs_a},
\eqref{eq:rec_rhs_b} and \eqref{eq:eta_eps} hence imply
\begin{align}
\begin{aligned}
\label{eq:rec_rhs}
&\limsup_{R \to 0} \limsup_{\eps \to 0}\int_{\T^2}
\left(\frac{\eps}{2}|\nabla m_\eps|^2 +
  \frac{1}{2\eps}\left(1-m_{\eps,3}^2 \right)\right) &\le 2
\HH^1(\partial A).
\end{aligned}
\end{align}

We now turn to the estimate of the nonlocal term in the energy $F$.
As for the local terms, our strategy is to use the one-dimensional
estimates from Lemma \ref{le:opt}. Invoking the coarea formula twice
and inserting \eqref{eq:m_eps}, we get
\begin{align}
\label{eq:I_nonl}
  &\Int_{\T^2} \Int_{\R^2} \frac{|m_{\eps,3}(x) -
    m_{\eps,3}(y)|^2}{|x-y|^3}\d[x]{2} \d[y]{2}\\
  &\ge\int_{-R}^{R}\int_{\{x:\,d(x)=\rho'\}}\left(\int_{-R}^{R}
    \int_{\{y:\,d(y)=\rho\}}\frac{|\xi_{\eps,R}(\rho')
      - \xi_{\eps,R}(\rho)|^2}{|x-y|^3}\dd \HH^1(y) \dd \rho \right) \dd
  {\mathcal H^1(x)} \dd \rho'. \notag
\end{align}
We claim that the integrals over curves tangential to the
boundary may be estimated as follows: For every $\delta > 0$,
there is an $R_{\delta,A}$ such that
\begin{align}\label{eq:tang-int}
  \int_{\{x:\,d(x)=\rho'\}}
  \int_{\{y:\,d(y)=\rho\}}\frac{1}{|x-y|^3}\dd \HH^1(y) \dd
  {\mathcal H^1(x)} \ge (1-\delta)\frac{2\HH^1(\p A)}{(\rho -
  \rho')^2}, 
\end{align}
for all $R\le R_{\delta,A}$ and all $\rho \neq \rho' \in (-R,R)$.
Assuming for a moment that \eqref{eq:tang-int} holds, we conclude by
inserting \eqref{eq:tang-int} into \eqref{eq:I_nonl} and applying the
one-dimensional estimate \eqref{eq:constr_nonlocal_energy}
\begin{align*}
\begin{aligned}
  \frac{\lambda}{|\log
    \eps|}\Int_{\T^2}|\nabla^{1/2}m_{\eps,3}|^2\d[x]{2}
  &\lupupref{eq:I_nonl}{eq:tang-int}\ge (1-\delta)\frac{\lambda
    \HH^1(\p A)}{2\pi |\log \eps|}\int_{-R}^{R}\int_{-R}^{R}
  \frac{|\xi_{\eps,R}(\rho) -
    \xi_{\eps,R}(\rho')|^2}{|\rho - \rho'|^2}\dd \rho' \dd \rho\\
  &\lupref{eq:constr_nonlocal_energy}\ge (1-\delta)2\HH^1(\p
  A)\frac{\lambda}{\lambda_c}\frac{\log(cR/\eps)}{|\log \eps|}.
    \end{aligned}
\end{align*}
Since $\delta$ was arbitrary, we obtain
\begin{align}\label{eq:nonl-sup}
  \liminf_{R \to 0}\liminf_{\eps \to 0}\frac{\lambda}{|\log
  \eps|}\Int_{\T^2}|\nabla^{1/2}m_{\eps,3}|^2\d[x]{2} \ge 2\HH^1(\p
  A)\frac{\lambda}{\lambda_c}. 
\end{align}
Together, \eqref{eq:rec_rhs} and \eqref{eq:nonl-sup} imply the limsup
inequality by a standard diagonal argument.

It remains to prove \eqref{eq:tang-int}, for which we fix
$x \in (\p A)_R$ with $d(x)=\rho'$ and pass to curvilinear coordinates
in a neighborhood of $\tilde x:=p(x) \in \p A$. More precisely, let
the curve $\gamma: (-R^{1/2},R^{1/2}) \to \p A$ be a parametrization
by arclength of a neighborhood of $\tilde x$ in $\p A$ with
$\gamma(0)=\tilde x$. Then, for all $R \le R_A$ with some
  $R_A > 0$ the function
\begin{align}
\begin{aligned}
   \Psi(\sigma,\rho) := \gamma(\sigma) + \nu(\gamma(\sigma)) \rho
    \end{aligned}
\end{align}
is a diffeomorphism from $(-R^{1/2},R^{1/2}) \times (-R,R)$ onto its
image, which we denote by $\Gamma_{\tilde x}$. The choice $R^{1/2}$
will become clear later. Note that due to compactness of $\p A$, we
may choose $R_A$ independent of $\tilde x$. A transformation of
variables then yields
\begin{align}\label{eq:trafo-trans}
  \Int_{\{y:\, d(y)=\rho\} \cap \Gamma_{p(x)}} \frac{1}{|x-y|^3} \dd
  \HH^1(y) = \int_{-R^{1/2}}^{R^{1/2}}
  \frac{(1+\kappa(\gamma(\sigma))\rho)}{|\Psi(0,\rho') -
  \Psi(\sigma,\rho)|^3} \dd \sigma, 
\end{align}
where $\kappa(\tilde y)$ denotes the signed curvature of $\p A$ at
$\tilde y$ (negative if $A$ is convex). Since the curvature of
$\p A$ is bounded, there is, for any $\delta >0$, an $ R_{\delta,A}>0$
such that for all $R \le R_{\delta,A}$ we have
\begin{align}\label{eq:bounded-curv}
  |\kappa|R \le \delta 
  && \text{and} 
  &&
     |\Psi(0,\rho') -
     \Psi(\sigma,\rho)| \le
     (1+\delta)\sqrt{\sigma^2 +
     (\rho-\rho')^2}. 
\end{align}
We conclude that, for any $\tilde \delta>0$, there is an
$\tilde R_{\tilde \delta,A}  > 0$ such that for all
$R \le \tilde R_{\tilde \delta,A}$ and all $\rho, \rho' \in (-R,R)$
we have
\begin{align}
\begin{aligned}
\label{eq:tang-int-cl}
\Int_{\{y:\, d(y)=\rho\} \cap \Gamma_{p(x)}}& \frac{1}{|x-y|^3} \dd
\HH^1(y) \lupupref{eq:trafo-trans}{eq:bounded-curv}\ge (1-\tilde
\delta)\int_{-R^{1/2}}^{R^{1/2}}\frac{1}{\left(\sigma^2 + (\rho -
    \rho')\right)^{3/2}} \dd \sigma\\
&= (1-\tilde \delta) \frac{2}{(\rho - \rho')^2}\frac{R^{1/2}}{\sqrt{R
    + (\rho-\rho')^2}} \ge (1-2\tilde \delta) \frac{2}{(\rho -
  \rho')^2}.
\end{aligned}
\end{align}
Integrating \eqref{eq:tang-int-cl} over $x$ and invoking
\eqref{eq:lim-AH} we obtain \eqref{eq:tang-int}.
\end{proof}

\subsection{Proof of Theorem \ref{th:SUPERcrit-F}}
\label{sec:SUPERcrit-proof-F}
We begin with the proof of the lower bound in Theorem
\ref{th:SUPERcrit-F}, which is the subject of Lemma
\ref{le:F-SUPERcrit-lb}. The proof of Theorem \ref{th:SUPERcrit-F} is
completed with the construction of the upper bound, carried out in
Lemma \ref{le:F-SUPERcrit-ub}.

\begin{lem}\label{le:F-SUPERcrit-lb}
  Let $\lambda_{c} := \frac{\pi}{2}$ and $F_{\eps,\lambda}$ as defined
  in \eqref{eq:F}. Then there is a universal constant
  $\delta >0$ such that for all $\eps < 1/2$ and all
  \begin{align}\label{eq:delta-lambda}
    \lambda_c \le \lambda < \delta|\log \eps| 
  \end{align}
  the family of functionals $\{F_{\eps, \lambda}\}$ is bounded below
  by
  \begin{align}
    \label{eq:scaling-opt-below}
    \min F_{\eps,\lambda} \gtrsim -\frac{\lambda
    \eps^{\frac{\lambda_c - \lambda}{\lambda}}}{|\log \eps|}. 
  \end{align}
  Moreover, the profiles achieving the optimal scaling can be
  characterized as follows: For any $\gamma >0$ and all
  $m \in H^1(\T^2;\S^2)$ which satisfy
  \begin{align}\label{eq:scaling-opt}
    F_{\eps, \lambda}[m] \le -\frac{\lambda \eps^{\frac{\lambda_c -
    \lambda}{\lambda}}}{|\log \eps|} \gamma, 
  \end{align}
  there holds
  \begin{align}\label{eq:scaling_var_supercrit_i-2}
    \begin{aligned}
      \Int_{\T^2} |\nabla m_{3}| \dd x \leq \Int_{\T^2} \left(
        \frac{\eps}{2} |\nabla m|^2 + \frac{1-m_{\eps,3}^2}{2\eps}
      \right) \dd x \leq \frac{\lambda}{|\log
        \eps|} \Int_{\T^2}|\nabla^{1/2} m_{3}|^2 \dd x,
    \end{aligned}
  \end{align}
  and the above quantities agree to leading order and scale like
  $\eps^{\frac{\lambda_c - \lambda}{\lambda}}$, i.e. if $A$ and $B$
  are any of the three quantities in
  \eqref{eq:scaling_var_supercrit_i-2}, we have
  \begin{align}\label{eq:scaling_var_supercrit-2}
    A  \sim \eps^{\frac{\lambda_c - \lambda}{\lambda}} \qquad
    \text{and} \qquad |A-B| \lesssim \frac{\lambda}{|\log
    \eps|}A, 
  \end{align}
  where the the constants may depend on $\gamma$.
\end{lem}

\begin{proof}
  By \eqref{eq:main_ineq_1d}, we may bound the energy from below by
  \begin{align}
    \begin{aligned}\label{eq:lower-bound}
      F_{\eps,\lambda}[m] &\lupref{eq:main_ineq_1d}\ge
      \left(1-\frac{\lambda}{|\log \eps|}\right)\int_{\T^2}
      \frac{\eps}{2} |\nabla m|^2 +\frac{1}{2\eps}(1-m_3^2) \dd x\\
      &\quad -\frac{\lambda}{\lambda_c}\frac{\log
        \left(c_*\max\left\{1,\min\left\{\frac{1}{\eps\int_{\T^2}|\nabla
                m_3| \dd x},
              \frac{1}{\eps}\right\}\right\}\right)}{|\log
        \eps|}\int_{\T^2} |\nabla m_3| \dd x.
    \end{aligned}
  \end{align}
  Without loss of generality, we may assume that
  $\int_{\T^2}|\nabla m_3| \dd x>0$. We first consider the case
  $\min \{\frac{1}{\eps \int_{\T^2}|\nabla m_3| \dd x},
  \frac{1}{\eps}\} \le 1$,
  for which, with the help of \eqref{eq:A2}, the estimate in
  \eqref{eq:lower-bound} turns into
  \begin{align}\label{eq:lower-bound-1}
    F_{\eps,\lambda}[m] \ge \left(1-\frac{\lambda \log(c_*^{1/\lambda_c})}{|\log
    \eps|}\right)\int_{\T^2}|\nabla m_3| \dd x
    \lupref{eq:delta-lambda} \ge
    \left(1- C \delta\right)\int_{\T^2}|\nabla m_3| \dd x 
  \end{align}
  for some universal constant $C>0$. For $\delta < 1/C$, the right
  hand side of \eqref{eq:lower-bound-1} is positive and the lower
  bound follows. Hence, we may assume
  $\min \{\frac{1}{\eps \int_{\T^2}|\nabla m_3| \dd x},
  \frac{1}{\eps}\} > 1$ so that \eqref{eq:lower-bound} implies
  \begin{align}
    \begin{aligned}\label{eq:lower-bound1}
      F_{\eps,\lambda}[m] &\ge \left(1-\frac{\lambda}{|\log
          \eps|}\right)\int_{\T^2} \frac{\eps}{2} |\nabla m|^2
      +\frac{1}{2\eps}(1-m_3^2) \dd x\\
      &\quad -\frac{\lambda}{\lambda_c}\frac{\log
        \left(\frac{c_*}{\eps\int_{\T^2}|\nabla m_3| \dd
            x}\right)}{|\log \eps|}\int_{\T^2} |\nabla m_3| \dd x.
    \end{aligned}
  \end{align}

  Abbreviating the energetic cost for $m$ to deviate from the optimal
  Bloch wall profile by
  \begin{align}
    D_\eps[m]:= \int_{\T^2} \frac{\eps}{2} |\nabla m|^2 +
    \frac{1}{2\eps}(1-m_3^2) \dd x - \int_{\T^2} |\nabla m_3| \dd x, 
  \end{align} and inserting $\mu :=
  \eps^{\frac{\lambda-\lambda_c}{\lambda}}\int_{\T^2}|\nabla m_3| \dd x
  $ and $c_{**}:=c_*e^{\lambda_c}$ into the lower bound in
  \eqref{eq:lower-bound1}, we get 
  \begin{align}
    \begin{aligned}\label{eq:lower-bound2}
      F_{\eps,\lambda}[m] &\ge \left(1-\frac{\lambda}{|\log
          \eps|}\right)D_\eps[m] -\frac{\lambda}{\lambda_c}\frac{\log
        \left(\frac{c_{**}}{\mu}\right)}{|\log \eps|}\mu
      \,\eps^{\frac{\lambda_c - \lambda}{\lambda}}.
    \end{aligned}
  \end{align}
  Since $\sup_{\mu > 0} \mu \log ( c_{**} / \mu) = c_{**} / e$,
  and since $D_{\eps}[m] \ge 0$ by \eqref{eq:A2}, the lower bound in
  \eqref{eq:scaling-opt-below} follows.

  We now turn to the proof of \eqref{eq:scaling_var_supercrit-2}. Note
  that \eqref{eq:A2} and $F_{\eps,\lambda}[m] \le 0$ yield
  \begin{align}
    \int_{\T^2} |\nabla m_3| \dd x \le \int_{\T^2} \left( \frac{\eps}{2}
    |\nabla m|^2 + \frac{1-m_3^2}{2\eps} \right) \dd x \le
    \frac{\lambda}{|\log \eps|} \int_{\T^2} |\nabla^{1/2}m_3| \dd x. 
  \end{align}
  For \eqref{eq:scaling_var_supercrit-2} it is hence sufficient to
  show
  \begin{align}\label{eq:scaling_var_supercrit_suff}
    \int_{\T^2} |\nabla m_3| \dd x \sim
    \eps^{\frac{\lambda_c-\lambda}{\lambda}} \quad \text{and} \quad
    \frac{\lambda}{|\log \eps|} \int_{\T^2} |\nabla^{1/2}m_3| \dd x -
    \int_{\T^2} |\nabla m_3| \dd x \lesssim \frac{\lambda
    \eps^{\frac{\lambda_c-\lambda}{\lambda}}}{|\log \eps|},
  \end{align}
  where here and in the rest of the proof the constants may depend on
  $\gamma$.  We combine the lower bound for the energy
  \eqref{eq:lower-bound2} with the upper bound \eqref{eq:scaling-opt}
  to obtain $\mu \log(c_{**}/\mu) \gtrsim 1$, which in turn implies
  $\mu \sim 1$. Hence, the first item in
  \eqref{eq:scaling_var_supercrit_suff} may be estimated as
  \begin{align*}
    \int_{\T^2} |\nabla m_3| \dd x = \mu
    \eps^{\frac{\lambda_c-\lambda}{\lambda}} \sim
    \eps^{\frac{\lambda_c-\lambda}{\lambda}}. 
  \end{align*}
  For $\delta >0$ sufficiently small universal and
  $\mu \sim 1$, the second item in
  \eqref{eq:scaling_var_supercrit_suff} follows from
  \eqref{eq:lower-bound2}:
  \begin{align}
    \frac{\lambda}{|\log \eps|} \int_{\T^2} |\nabla^{1/2}m_3| \dd x -
    \int_{\T^2} |\nabla m_3| \dd x = -F_{\eps,\lambda}[m] +D_{\eps}[m]
    \lupref{eq:lower-bound2} \lesssim  \frac{\lambda
    \eps^{\frac{\lambda_c-\lambda}{\lambda}}}{|\log \eps|}. 
  \end{align}
  This concludes the proof.
\end{proof}

\begin{lem}[Upper bound in the supercritical
  regime]\label{le:F-SUPERcrit-ub}
    There is a constant $0<K<1$ such that for every $(\eps,\lambda)$ with
    \begin{align}\label{eq:SUPERcrit-F-regime}
      \lambda_c< \lambda \qquad \text{and}\qquad 0<\eps^{\frac{\lambda
      - \lambda_c}{\lambda}} < K, 
    \end{align}
    there is $m_{\eps,\lambda} \in H^1(\T^2;\S^2)$ which satisfies
    \begin{align}
      F_{\eps,\lambda}[m_{\eps,\lambda}] \lesssim -\frac{\lambda
      \eps^{\frac{\lambda_c - \lambda}{\lambda}}}{|\log \eps|}. 
    \end{align}
\end{lem}

\begin{proof}
  We make an ansatz with $N$ transitions equally separated by
  $1/N$-sized regions of approximately constant magnetization. More
  precisely, we take the transitions as solutions of the optimal
  profile ODE and define
\begin{align}
m_{\eps,N}(x_1,x_2) =
\begin{cases}
  \xi_{\eps,\infty}\left(\tfrac{x_1-\tfrac{1}{2N}}{\eps}\right){e_3} +
  \sqrt{1-\xi_{\eps,\infty}^2\left(\tfrac{x_1-\tfrac{1}{2N}}{\eps}\right)}{e_2},
  &
  \text{for } x_1 \in \left[0,\tfrac1N\right]\\
  \xi_{\eps,\infty}\left(\tfrac{\tfrac{3}{2N}-x_1}{\eps}\right){e_3} +
  \sqrt{1-
    \xi_{\eps,\infty}^2\left(\tfrac{\tfrac{3}{2N}-x_1}{\eps}\right)}{e_2},
  & \text{for } x_1 \in \left[\tfrac1N,\tfrac2N\right]
\end{cases}
\end{align}
extended periodically to $\T^2$ (see Fig. \ref{fig:ansatz}). Applying
Lemma \ref{le:opt} with $X=\frac{1}{2N}$ and using symmetries of
$m_{\eps,N}$, we get
\begin{align}\label{eq:construction}
  \int_{\T^2} \left( \frac{\eps}{2} |\nabla m_{\eps,N}|^2 +
  \frac{1-m_{(\eps,N),3}^2}{2\eps} \right) \dd x \le 2N  
\end{align}
and, for all $\eps <\tfrac{1}{4N}$, we have 
\begin{align}
\label{eq:construction_2}
  \int_{\T^2}
  &|\nabla^{1/2}m_{(\eps,N),3}|^2 \d2
    \lupref{eq:A4}={\frac{1}{4\pi}\int_{\T} \int_{\R} \int_{\R}
    \frac{|m_{(\eps,N),3}(x_1)-m_{(\eps,N),3}(y_1)|^2}{(|x_1-y_1|^2
    + s^2)^{3/2} }\dd s \, \dd x_1 \dd y_1} \notag\\
  &\ge\frac{1}{2\pi}\sum_{k=1}^N
    \int_{\frac{k-1}{N}}^{\frac{k}{N}}\int_{\frac{k-1}{N}}^{\frac{k}{N}}
    \frac{|m_{(\eps,N),3}(x_1)-m_{(\eps,N),3}(y_1)|^2}{|x_1-y_1|^2}\dd  
    x_1 \dd y_1 \\
  &=
    \frac{N}{4\lambda_c}\int_{-\frac{1}{2N}}^{\frac{1}{2N}}
    \int_{-\frac{1}{2N}}^{\frac{1}{2N}}  
    \frac{|\xi_{\eps,\infty}(x)-
    \xi_{\eps,\infty}(y)|^2}{|x-y|^2} \dd x \dd y
  \lupref{eq:constr_nonlocal_energy}\ge
  2N\frac{\log(\tfrac{c}{2 \eps N})}{\lambda_c}.\notag
\end{align}
To obtain the upper bound, we combine estimates
\eqref{eq:construction} and \eqref{eq:construction_2} and optimize in
$N \in \N$. The choice
$N :=2\left\lfloor{K \eps^{\frac{\lambda_c -
        \lambda}{\lambda}}}\right\rfloor$ is admissible because
$N \lupref{eq:SUPERcrit-F-regime}\ge 2$ and
$\eps N \le 2K\le \tfrac{1}{4}$ for
$K\le\tfrac{1}{8}\min\{1,c\}$. Since $0<\eps<1$, we get
\begin{align}
\begin{aligned}\label{eq:e-c}
  F_{\eps,\lambda}[m_{\eps,N}] \le 2N\left(1-\frac{\lambda \log
      (\tfrac{c}{2\eps N})}{\lambda_c |\log \eps|}\right) {\leq}
  -\frac{{C} \lambda \eps^{\frac{\lambda_c -
        \lambda}{\lambda}}}{|\log \eps|},
\end{aligned}
\end{align}
for some universal $C > 0$, which is the desired estimate.
\end{proof}

\subsection{Proof of Theorem \ref{th:crit-F}}
\begin{proof}[Proof of Theorem \ref{th:crit-F}]
We start by proving item $(i)$. Inserting \eqref{eq:beta_1} into
the lower bound \eqref{eq:F_SUBcrit_lb}, we get for sufficiently small
$\eps>0$
\begin{align}
  \begin{aligned}
    F_{\eps,\lambda}[m] &\ge \left(1- \frac{\log(\eps
        c)\log(\eps/\beta_1)}{\log(\eps)^2}\right) \int_{\T^2} |\nabla
    m_3| \dx2\\
    &\ge\left(\frac{|\log(\eps)|\log(c/\beta_1) +
        \log(c)\log(\beta_1)}{|\log(\eps)|^2}\right) \int_{\T^2}
    |\nabla m_3| \dx2.
  \end{aligned}
\end{align}
For $\beta_1<c$, the bracket is positive, which shows that the minimal
value of $\min F_{\eps,\lambda}=0$ is only attained for
$m\equiv \pm e_3$.  Since
$\eps^{\frac{\lambda_+(\eps)-\lambda_c}{\lambda_+(\eps)}} \le
\frac{2}{\beta_2}$ for sufficiently small $\eps>0$, the second part
follows from Lemma \ref{le:F-SUPERcrit-ub}.

To proceed, we next establish the estimate
\begin{align}\label{eq:estos}
  \int_{\T^2} |\nabla m_3| \dx2 \lesssim \max\left\{1,|\log
  \eps|F_{\eps,\lambda_c}[m]\right\}. 
\end{align}
It is enough to show that there are constants $C,\eps_0 >0$ such that
for all $\eps \in (0, \eps_0)$ we have
\begin{align}\label{eq:F-grow}
  \int_{\T^2} |\nabla m_3| \dx2 \ge C \quad \implies \quad
  F_{\eps,\lambda_c}[m] \gtrsim \frac{1}{|\log\eps|}\int_{\T^2}
  |\nabla m_3| \dx2. 
\end{align}
Indeed, by \eqref{eq:main_ineq_1d}, we may bound the energy from below
by
\begin{align}
  \begin{aligned}\label{eq:lower-bound-c}
    F_{\eps,\lambda_c}[m_\eps] &\lupref{eq:main_ineq_1d}\ge
    \left(1-\frac{\lambda_c}{|\log \eps|}\right)\int_{\T^2} \left(
      \frac{\eps}{2} |\nabla m_\eps|^2
      +\frac{1}{2\eps}(1-m_{\eps,3}^2) \right)
    \dd x\\
    &\quad -\frac{\log
      \left(c_*\max\left\{1,\min\left\{\frac{1}{\eps\int_{\T^2}|\nabla
              m_{\eps,3}| \dd x},
            \frac{1}{\eps}\right\}\right\}\right)}{|\log
      \eps|}\int_{\T^2} |\nabla m_{\eps,3}| \dd x.
  \end{aligned}
\end{align}
We first consider the case
$\min \{\frac{1}{\eps \int_{\T^2}|\nabla m_{\eps,3}| \dd x},
\frac{1}{\eps}\} \le 1$, for which \eqref{eq:lower-bound-c} turns into
\begin{align}\label{eq:lower-bound-1-c}
  F_{\eps,\lambda_c}[m] \ge \left(1-\frac{\lambda_c + \log(c_*)}{|\log
  \eps|}\right)\int_{\T^2}|\nabla m_{\eps,3}| \dd x \gtrsim
  \int_{\T^2}|\nabla m_3| \dd x.
\end{align}
For the remaining case, we have
$\min \{\frac{1}{\eps \int_{\T^2}|\nabla m_3| \dd x}, \frac{1}{\eps}\}
\ge 1$ and \eqref{eq:lower-bound-c} implies
\begin{align}
  \begin{aligned}\label{eq:lower-bound1-c}
    F_{\eps,\lambda_c}[m_\eps] &\ge \left(1-\frac{\lambda_c}{|\log
        \eps|}\right)\int_{\T^2} \left( \frac{\eps}{2} |\nabla
      m_\eps|^2
      +\frac{1}{2\eps}(1-m_{\eps,3}^2) \right) \dd x\\
    &\quad -\frac{\log \left(\frac{c_*}{\eps\int_{\T^2}|\nabla
          m_{\eps,3}| \dd x}\right)}{|\log \eps|}\int_{\T^2} |\nabla
    m_{\eps,3}| \dd x\\
    &\lupref{eq:A2}\ge -\frac{\log
      \left(\frac{c_{**}}{\int_{\T^2}|\nabla m_3| \dx2}\right)}{|\log
      \eps|}\int_{\T^2}|\nabla m_3| \dx2,
  \end{aligned}
\end{align}
where we have inserted $c_{**}:=c_*e^{\lambda_c}$. The estimate
\eqref{eq:F-grow} follows with the choice $C=
2c_{**}$.
    
With \eqref{eq:estos} at hand, we now prove item $(ii)$, starting with
the lower bound. Let $m_{\eps} \to m$ in $L^1(\T^2)$ for some
$m \in L^1(\T^2;\R^3)$. Lemma \ref{le:F-SUPERcrit-lb} yields 
\begin{align}
  \liminf_{\eps \to 0}F_{\eps,\lambda_c}[m_\eps] 
  \geq 0, 
\end{align}
which proves the lower bound in case that
$m \in L^1(\T^2;\{\pm e_3 \})$. For the remaining case, we may assume
$\int_{\T^2} (1-m_{\eps,3}^2) \dx2 \gtrsim 1$. For sufficiently small
$\eps$, estimates \eqref{eq:main_ineq_1d} and \eqref{eq:estos} then
yield 
\begin{align}
  \begin{aligned}\label{eq:noop}
    \int_{\T^2} (1-m_{\eps,3}^2) \dx2 &\lesssim \eps
    \left(F_{\eps,\lambda}[m_\eps] + \frac{\lambda_c}{|\log
        \eps|}\int_{\T^2}|\nabla^{1/2}m_{\eps,3}|^2 \dx2 \right)\\
    &\lupref{eq:main_ineq_1d}\lesssim \eps
    \left(F_{\eps,\lambda}[m_\eps] + \int_{\T^2}|\nabla m_{\eps,3}|
      \dx2 \right) \lupref{eq:estos} \lesssim \eps \left( 1 +
        |\log \eps|F_{\eps,\lambda_c}[m_\eps] \right),
  \end{aligned}
\end{align}
which implies
$\liminf_{\eps \to 0} F_{\eps,\lambda_c}[m_\eps] = +\infty$ for
$m \in L^1(\T^2; \R^3) \setminus L^1(\T^2;\{\pm e_3 \})$. Since the
construction of the upper bound was already carried out in Lemma
\ref{le:constr}, the proof is complete.
    
To prove item $(iii)$, we again make use of the construction
in Lemma \ref{le:F-SUPERcrit-ub}. However, this time we take
$N = \lfloor \log(|\log \eps|) \rfloor$. Analogous to \eqref{eq:e-c},
we get for sufficiently small $\eps$
\begin{align}
  \begin{aligned}\label{eq:e-c-2}
    F_{\eps,\lambda}[m_{\eps,N}] \le 2N\left(1-\frac{\log
        (\tfrac{2\eps N}{c})}{\log \eps}\right) \lesssim 
    \frac{N \log N}{|\log \eps|} \longrightarrow 0, \quad \text{for }
    \eps \to 0.
    \end{aligned}
    \end{align}
    Therefore, it remains to show that $m_{\eps,N}$ is not
    compact in the strong $L^1$-topology. Since
    $\int_{\T^2} |m_{\eps,N}|^2 \dx2 =1$, any possible limit
    $\tilde m$ of (a subsequence of) $m_{\eps,N}$ in the strong
    topology needs to satisfy $\int_{\T^2} |\tilde m|^2 \dx2
    =1$.
    However, since $\eps N \to 0$ as $\eps \to 0$, it is clear
      that $m_{\eps,N}$ converges weakly to zero in $L^2(\T^2)$,
      leading to a contradiction.

    Finally, item $(iv)$ follows directly from \eqref{eq:noop},
    \eqref{eq:estos} and the compact embedding
    $BV(\T^2) \hookrightarrow L^1(\T^2)$.
\end{proof}

\section{Stray field estimates and reduction of the full energy}
\label{sec:reduction-E-to-F}

The goal of this section is to make the heuristic reduction in section
\ref{sec:heuristic} rigorous. We prove the following
\begin{lem}[Reduction of the energy]\label{le:reduction-of-E}
  There is a universal constant $C>0$ such that energy $E$ is bounded
  below by
  \begin{align}
    \begin{aligned}\label{eq:E-m2d}
      E[m] &\ge \ell^2 t + \left(1-Ct^2\right)\Int_{\T_\ell^2\times
        (0,t)} |\nabla m|^2
      + (Q-1) (m_1^2 + m_2^2)\d3 \\
      &\quad - 2 \Int_{\T_\ell^2\times (0,t)}
      m_3h_{\mathrm{ext}}\d3 - \frac{t^2}{2}
      \Int_{\T_\ell^2}|\nabla^{1/2}\overline{m}_3|^2\d2,
    \end{aligned}
  \end{align}
  where $\overline m(x') = \frac{1}{t}\int_0^t m(x',x_3) \d[x_3]{1}$
  denotes the ${e_3}$-average of the magnetization over $(0,t)$.
\end{lem}

Note that for two-dimensional magnetizations \eqref{eq:E-m2d} also
holds in the reversed direction if $-C$ is replaced by $C$. Hence the
lower bound is asymptotically sharp. We also remark that a similar
  sharp estimate for the three-dimensional dipolar energy holds for
  thin three-dimensional domains in the whole space \cite{m17}.

For the proof of Lemma \ref{le:reduction-of-E}, which is deferred
until the end of this section, we need several estimates presented in
the following sections.

\subsection{Approximation of $m$ by its $e_3$-average $\overline m$}
\label{sec:m-approx-average}

Since the thickness $t$ of the film is small, the exchange energy
strongly penalizes oscillations of the magnetization in the normal
direction of the film. Hence the averaged magnetization $\overline m$
is a good approximation of $m$, and Assumption $\eqref{eq:I}$ in
section \ref{sec:heuristic} can be made rigorous by the following
Poincaré-type inequality
\begin{align}\label{eq:poincare_3}
    \Int_{\T^2_\ell \times (0,t)} |m - \chi_{(0,t)}\overline m|^2\d3
    &\lesssim t^2 \Int_{\T_\ell^2 \times (0,t)}  |\p_{3} m|^2\d3,
\end{align}
which holds for all $m \in H^1(\T^2_\ell \times (0,t); \R^3)$ and
{can be proved} by standard methods.

\subsection{Approximation of the stray field energy}
\label{sec:stray-field-approx}

In this section, we establish an approximation of the stray field,
i.e. a rigorous version of Assumption $\eqref{eq:I}$. In particular,
we show that for thin films, the difference between the stray field
energy of the averaged magnetization and the stray field energy of the
full magnetization may be estimated by the exchange energy at lower
order. We mention that an estimate of this type already occured in \cite{KohnSlastikov}, however, it is not strong enough for our purpose. The statement of Theorem \ref{th:h-m2d} below is slightly
stronger than what is necessary to prove Lemma \ref{le:reduction-of-E}
and might be of independent interest for other thin film regimes.

\begin{theorem}
    \label{th:h-m2d}
    Let $m \in H^1(\T^2_\ell \times (0,t);\R^3)$,
    then the stray field energy (see \eqref{eq:h}) satisfies
    \begin{align}
      \label{eq:h-split-approx}
      \Bigg| \int_{\T_\ell^2\times \R} |h[m]|^2 \d3 
      &- \int_{\T_\ell^2 \times
        \R} |h[m_3 e_3]|^2 \d3
        - \int_{\T_\ell^2
        \times \R} |h[m']|^2
        \d3 \Bigg|  \notag \\ 
      &\qquad \qquad \lesssim t^2 \int_{\T_\ell^2 \times (0,t)} |\nabla m|^2
        \d3,\\ 
      \Bigg| \int_{\T_\ell^2\times \R} |h[m]|^2 \d3 
      &- \int_{\T_\ell^2} |h[\chi_{(0,t)}\overline m]|^2 \d3 \Bigg| \lesssim
        t^2\int_{\T_\ell^2 \times (0,t)} |\nabla m|^2
        \d3,\label{eq:h-av} 
    \end{align}
  where $m' = m-m_3e_3$ is understood to have values in $\R^3$ with
  $e_3$-component 0.  Moreover, the contributions due to $m_3$ and
  $m'$ may be approximated by 
    \begin{align}
      \label{eq:h-m_3}
      \Bigg| \int_{\T_\ell^2 \times \R} |h[m_3 e_3]|^2 \d3
      &-
        \int_{\T_\ell^2\times
        (0,t)}  m_3^2 \d3
        + \frac{t^2}{2}
        \int_{\T_\ell^2}
        |\nabla^{1/2}\overline
        m_3|^2 \dx2 \Bigg| \notag \\ 
      &\qquad \qquad \lesssim t^2\int_{\T_\ell^2 \times (0,t)} |\nabla m|^2
        \d3, \\ 
      \label{eq:h-m_p}
      \Bigg| \int_{\T_\ell^2 \times \R} |h[m']|^2 \d3 
      &- \frac{t^2}{2}
        \int_{\T_\ell^2}
        |\nabla^{-1/2}
        \nabla'\cdot
        \overline m'|^2 \d2 \Bigg| \lesssim
        t^2 \int_{\T_\ell^2\times
        (0,t)} |\nabla m|^2
        \d3,\\ 
      \int_{\T_\ell^2 \times \R} |h[m']|^2 \d3 
      & \lesssim
        t^2 \int_{\T_\ell^2
        \times (0,t)} \left( |\nabla 
        m|^2 + |m'|^2 \right) \d3. 
     \end{align}
   
\end{theorem}

\begin{proof}
  It is sufficient to argue for
  $m \in C_c^\infty(\T_\ell^2 \times \R; \R^3)$, because the general
  case follows by an approximation argument, as we now explain. Since
  $\T_\ell^2 \times (0,t)$ is an extension domain, there exists, for
  every $m \in H^1(\T_\ell^2 \times (0,t);\R^3)$, a sequence
  $(m_n)_{n \in \N}$ with
  $m_n \in C_c^\infty(\T_\ell^2 \times \R; \R^3)$ such that
  $\|m-m_n\|_{L^2(\T_\ell^2 \times \R)} + \|\nabla m - \nabla
  m_n\|_{L^2(\T_\ell^2 \times (0,t))} \to 0$.  It remains to check
  that all terms in \eqref{eq:h-split-approx} -- \eqref{eq:h-m_p} are
  continuous. Note that by \eqref{eq:poincare_3}, we also have
  $\|\overline m_n - \overline m\|_{L^2(\T_\ell^2)} \to 0$.  Moreover,
  $t\int_{\T_\ell} |\nabla \overline m_n|^2 \d2 \lesssim \int_{\T_\ell
    \times (0,t)}|\nabla m_n|^2 \d3$ (see \eqref{eq:A3} in the
  Appendix for a proof). Hence the convergence follows from the
  elliptic estimate
  $\int_{\T_\ell^2 \times \R}|h[m_n-m]|^2 \d3 \le \int_{\T_\ell^2
    \times \R}|m_n-m|^2\d3$ and by interpolation for the terms
  involving fractional derivatives.

  We write the stray field energy in terms of the magnetostatic
  potential $\phi$
  \begin{align}
    \int_{\T_\ell^2 \times\R } |h[m]|^2 \dx3 = -\int_{\T_\ell^2 \times
    \R} \phi\, \nabla \cdot m \d3 \quad \text{where }\Delta \phi =
    \nabla \cdot m \text{ in }\DD'(\T_\ell^2 \times \R). 
    \end{align}
    Upon passing to Fourier series (with respect to the in-plane
    variables), we get
    \begin{align}\label{eq:Fourier-h-energy}
      \int_{\T_\ell^2 \times \R} \phi\, \nabla \cdot m \d3 =
      \frac{1}{\ell^2}\int_{\R }\sum_{k \in
      \frac{2\pi}{\ell}\Z^2}\widehat\phi_k^*(z)\left(\p_z\widehat
      m_{3,k}(z) -ik\cdot \widehat m'_k(z)\right) \, \dz1,   
    \end{align}
    where the Fourier coefficients $\widehat \phi_k:\R \to \mathbb{C}$
    for $k \in {2 \pi \over \ell} \Z^2$ of $\phi$ solve
    \begin{align}
      \p_z^2 \widehat \phi_k - |k|^2 \widehat \phi_k = \p_z\widehat
      m_{3,k} -ik\cdot \widehat m'_k. 
    \end{align}
    We introduce the fundamental solution
    \begin{align}
      H_k(z) = \begin{cases}
        \frac{e^{-|k||z|}}{|k|} & \text{for }k \neq 0,\\
        -|z| & \text{for }k=0,
        \end{cases}
    \end{align}
    which satisfies
    \begin{align}\label{eq:H-id}
      - \p_z^2 H_k + |k|^2H_k = 2\delta \qquad \text{in }\DD'(\R)
      \text{ for all }k \in \Z^2, 
    \end{align}
    where $\delta$ denotes the Dirac measure at 0. The fundamental
    solution allows to rewrite $\widehat \phi_k(z)$ as
    \begin{align}
      \widehat \phi_k(z) = -\frac{1}{2} \int_{\R} H_k(z-z') \left(\p_z
      \widehat m_{3,k}(z')-ik\cdot \widehat m'_k(k,z')\right) \dd
      {z'},  
    \end{align}
    which by \eqref{eq:Fourier-h-energy} leads to the following
    expression for the stray field energy
    \begin{align}
        \begin{aligned}\label{eq:h-fourier}
          \int_{\T_\ell^2 \times\R } |h[m]|^2 \dx3
          &= \frac{1}{2\ell^2} \int_{\R}\int_{\R} \sum_{k \in
            \frac{2\pi}{\ell }\Z^2}(\p_z \widehat m_{3,k}(z)-ik\cdot
          \widehat m'_k(z))^*\\ 
          &\qquad \qquad \qquad \times H_k(z-z')(\p_z \widehat
          m_{3,k}(z')-ik\cdot \widehat m'_k(z')) \dd z \dd z'.
         \end{aligned}
    \end{align}
    To prove \eqref{eq:h-split-approx}, we need to show that the mixed
    terms in \eqref{eq:h-fourier}, i.e. terms of the form
    \begin{align}\label{eq:h-I}
      I:= \frac{1}{\ell^2} \int_{\R}\int_{\R} \sum_{k \in
      \frac{2\pi}{\ell }\Z^2} \p_z \widehat m_{3,k}^* (z)
      H_k(z-z')(ik\cdot \widehat m'_k(z')) \dd z \dd z' 
    \end{align}
    satisfy
    $|I| \lesssim t^2 \int_{\T_\ell^2 \times (0,t)} |\nabla m|^2 \dd
    x$. Integrating by parts in \eqref{eq:h-I}, we get
    \begin{align}\label{eq:p-ip}
      I= -\frac{1}{\ell^2} \int_{\R}\int_{\R} \sum_{k \in
      \frac{2\pi}{\ell }\Z^2} \widehat m_{3,k}^* (z)
      \p_zH_k(z-z')(ik\cdot \widehat m_k'(z')) \dd z \dd z'. 
    \end{align}
    We write $m = \chi_{(0,t)}\overline m + u$ where as usual
    $\overline m(x')= \frac{1}{t}\int_0^t m(x',x_3) \dd x_3$ denotes
    the average of $m$ over in the ${e_3}$-direction.  With this
    notation, \eqref{eq:p-ip} turns into
    \begin{align}
        \begin{aligned}\label{eq:m-av-ins}
          I= -\frac{1}{\ell^2} \int_{\R}\int_{\R} \sum_{k \in
            \frac{2\pi}{\ell }\Z^2}&(\chi_{(0,t)}(z)\widehat{
            \overline m}_{3.k} + \widehat{u}_{3,k}(z))^*
          \p_zH(k,z-z')\\
          &\times \left(ik\cdot \chi_{(0,t)}(z')\widehat{ \overline
              m}'_k + ik\cdot \widehat u'_k(z')\right) \dd z \dd z'.
        \end{aligned}
    \end{align}

    Since $\p_z H_k(z) = -\frac{z}{|z|}e^{-|k||z|}$ is anti-symmetric
    in $z$, we have $\int_0^t \int_0^t \p_zH_k(z-z') \dd z \dd z' = 0$
    which means that upon expanding \eqref{eq:m-av-ins}, the term
    involving $\overline m_3$ and $\overline m'$
    vanishes. Furthermore, we have $|\p_z H_k| \le 1$ and hence the
    remaining terms in \eqref{eq:m-av-ins} may be estimated by
    \begin{align}\label{eq:I-rem-terms}
      |I|&\leq  \frac{1}{\ell^2}\int_{\R} \int_{\R}\sum_{k \in
           \frac{2\pi}{\ell }\Z^2} \left(|\widehat u_{3,k}(z)|\, |k
           \cdot \widehat m'_k(z')| +
           |\chi_{(0,t)}(z)\widehat{\overline m}_{3,k}|\, |k\cdot
           u'_k(z')|\right) \dd z \dd z'. 
    \end{align}
    Note that passing to Fourier series in the in-plane variables
    commutes with taking $e_3$-averages. Thus $\widehat u_{j,k}$
    has ${e_3}$-average
    zero for all $j=1,2,3$ and the intermediate value theorem yields
    $\tau_{j,k}, \rho_{j,k} \in (0,t)$ such that
    $\Re \widehat u_{j,k}(\tau_{j,k}) = 0$ and
    $\Im \widehat u_{j,k}(\rho_{j,k})=0$. By the fundamental
    theorem of calculus, we hence get the estimate
    \begin{align}
        \begin{aligned}\label{eq:ftc-u}
          |\widehat{u}_{j,k}(z)| &\lesssim \int_{0}^t |\p_z
          \widehat m_{j,k}(\tau)| \dd \tau \qquad \text{ for all
          } z \in (0,t) \text{ and }j=1,2,3.
        \end{aligned}
    \end{align}
    Inserting \eqref{eq:ftc-u} into \eqref{eq:I-rem-terms} and using
    Jensen's inequality yields the rough estimate
    \begin{align}
      |I| & \lesssim \sum_{n,j=1}^{3}\frac{t}{\ell^2} \int_{0}^t
            \int_{0}^t \sum_{k \in \frac{2\pi}{\ell }\Z^2}|\p_z
            \widehat m_{j,k}(z)|\, |k|\, |\widehat{ m}_{n,k}(z')| \dd
            z \dd z'. 
    \end{align}
    By Young's inequality and Parseval's identity, we conclude
    \begin{align}
        \begin{aligned}
          |I|&\lesssim \sum_{n,j=1}^{3}\frac{t}{\ell^2}
          \int_{0}^t\int_{0}^t \sum_{k \in \frac{2\pi}{\ell
            }\Z^2}\left(| \p_z \widehat m_{j,k}(z)|^2 +
            |k|^2|\widehat
            m_{n,k}(z')|^2\right) \dd z \dd z'\\
            &\lesssim t^2 \int_{\T_\ell^2 \times (0,t)} |\nabla m|^2 \dd x,
        \end{aligned}
    \end{align}
    which completes the proof of \eqref{eq:h-split-approx}.  Assuming
    for a moment that $\eqref{eq:h-m_3}$ and $\eqref{eq:h-m_p}$ hold,
    identity \eqref{eq:h-av} is obtained as follows. Applying
    \eqref{eq:h-split-approx} to $m$ and $\chi_{(0,t)}\overline m$, we
    get
    \begin{align}
        \begin{aligned}\label{eq:hxt}
          & \Bigg| \int_{\T_\ell^2} |h[m]|^2 \d3 - \int_{\T_\ell^2}
          |h[\chi_{(0,t)}\overline m]|^2 \d3\\
          &- \int_{\T_\ell^2 \times \R} |h[m_3 e_3]|^2 \d3 +
          \int_{\T_\ell^2 \times \R}
          |h[\chi_{(0,t)}\overline m_3 e_3]|^2 \d3\\
          &\qquad - \int_{\T_\ell^2 \times \R} |h[m']|^2 \d3 +
          \int_{\T_\ell^2 \times \R} |h[\chi_{(0,t)}\overline m']|^2
          \d3 \Bigg| \lupref{eq:h-split-approx} \lesssim
          t^2\int_{\T_\ell^2 \times (0,t)} |\nabla m|^2 \d3,
    \end{aligned}
    \end{align}
    where we have also used (see \eqref{eq:A3} in the appendix for a
    proof)
    \begin{align*}
      \int_{\T_\ell^2 \times (0,t)} |\nabla
      \left(\chi_{(0,t)}\overline m\right)|^2 \d3 =
      t\int_{\T_\ell^2} |\nabla' \overline m|^2 \d2 \lupref{eq:A3}\le
      \int_{\T_\ell^2 \times (0,t)} |\nabla m|^2 \d3. 
    \end{align*}
    Applying \eqref{eq:h-m_3} and \eqref{eq:h-m_p} to \eqref{eq:hxt}
    yields the claim
    \begin{align}
    \begin{aligned}
      &\Bigg| \int_{\T_\ell^2} |h[m]|^2 \d3 - \int_{\T_\ell^2}
      |h[\chi_{(0,t)}\overline m]|^2 \d3\\
      &- \int_{\T_\ell^2 \times (0,t)} m_3^2\d3 + \int_{\T_\ell^2
        \times (0,t)}(\chi_{(0,t)}\overline m_3)^2 \d2 \Bigg|
      \lupupref{eq:h-m_3}{eq:h-m_p} \lesssim t^2 \int_{\T_\ell^2
        \times (0,t)} |\nabla m|^2
      \d3\\
      & \Bigg| \int_{\T_\ell^2} |h[m]|^2 \d3 - \int_{\T_\ell^2}
      |h[\chi_{(0,t)}\overline m]|^2 \d3 \Bigg|
      \lupref{eq:poincare_3}\lesssim t^2 \int_{\T_\ell^2 \times
        (0,t)} |\nabla m|^2 \d3.
    \end{aligned}
    \end{align}
    We turn to the proof of \eqref{eq:h-m_3}. Integrating by parts
    twice and inserting \eqref{eq:H-id}, we get
    \begin{align*}
      \begin{aligned}
        \int_{\T^2_\ell \times \R}|h[m_3]|^2 \d3
        \lupref{eq:h-fourier}=&\frac{1}{2\ell^2} \int_{\R}\int_{\R}
        \sum_{k \in \frac{2\pi}{\ell }\Z^2}\p_z \widehat m_{3,k}^*(z)
        H_k(z-z')\p_z \widehat m_{3,k}(z') \dd z \dd z'\\ 
        &= -\frac{1}{2\ell^2}\int_{\R}\int_{\R} \sum_{k \in
          \frac{2\pi}{\ell }\Z^2} \widehat m_{3,k}^*(z)
        \p_z^2H_k(z-z') \widehat m_{3,k}(z') \dd z \dd z'\\ 
        &\lupref{eq:H-id}= \frac{1}{\ell^2}\int_{\R}\sum_{k \in
          \frac{2\pi}{\ell }\Z^2} | \widehat m_{3,k}(z)|^2 \dd z\\ 
        &\qquad - \frac{1}{2\ell^2}\int_{\R}\int_{\R} \sum_{k \in
          \frac{2\pi}{\ell }\Z^2} \widehat m_{3,k}^*(z)
        |k|e^{-|k||z-z'|} \widehat m_{3,k}(z') \dd z \dd z'
        \end{aligned}
    \end{align*}
    Since $|1-e^{-|k||z|}| \le |k|t$ for $z \in (-t,t)$, the last line
    above 
    \begin{align}
      J:=\frac{1}{2\ell^2}\int_{\R}\int_{\R} \sum_{k \in
         \frac{2\pi}{\ell }\Z^2} \widehat m_{3,k}^*(z)
         |k|e^{-|k||z-z'|} \widehat m_{3,k}(z') \dd z \dd z'      
    \end{align}
    may be estimated, with the help of Young's inequality, by
    \begin{align}
    \begin{aligned}
      \Bigg| J - \frac{t^2}{2\ell^2} \sum_{k \in \frac{2\pi}{\ell
        }\Z^2} |k| |\widehat{ \overline{m}}_{3,k}(z)|^2 \Bigg| &
      \lesssim \frac{t}{\ell^2}\int_{0}^{t}\int_{0}^{t} \sum_{k \in
        \frac{2\pi}{\ell }\Z^2} |\widehat m_{3,k}(z)| |k|^2 |\widehat
      m_{3,k}(z')| \dd z \dd z'  \notag \\
      & \lesssim \frac{t^2}{\ell^2}\int_{0}^{t} \sum_{k \in
        \frac{2\pi}{\ell }\Z^2} |k|^2 |\widehat m_{3,k}(z)|^2 \dd
      z, 
       \end{aligned}
    \end{align}
    which by Parseval's identity is equivalent to
    \begin{align}
      \Bigg| J - \frac{t^2}{2} \int_{\T^2_\ell}
      |\nabla^{1/2}\overline m_3|^2 \d2 \Bigg| \lesssim t^2
      \int_{\T_\ell^2 \times(0,t)} |\nabla' m_3|^2 \d3.
    \end{align}
  In total, we get 
    \begin{align}
    \begin{aligned}
      \Bigg| \int_{\T_\ell^2 \times \R} |h[m_3 e_3]|^2 \d3 &-
      \int_{\T_\ell^2 \times (0,t)} m_3^2 \d3 + \frac{t^2}{2}
      \int_{\T^2_\ell} |\nabla^{1/2}\overline m_3|^2 \d3 \Bigg| \\
      &\qquad \lesssim t^2 \int_{\T_\ell^2 \times(0,t)} |\nabla'
      m_3|^2 \d3,
        \end{aligned}
    \end{align}
  which proves \eqref{eq:h-m_3}. We continue with the proof of
  \eqref{eq:h-m_p}. Since $|1-e^{-|k||z|}| \le |k|t$ for
  $z \in (0,t)$, we may insert $|H_k(z-z') - \frac{1}{|k|}| \le t$ for
  $k \neq 0$ into \eqref{eq:h-fourier}
    \begin{align}
      \int_{\T_\ell^2 \times \R} |h[m']|^2 \d3 
      &\lupref{eq:h-fourier}=
        \frac{1}{2\ell^2}
        \int_{\R}\int_{\R}
        \sum_{k \in
        \frac{2\pi}{\ell
        }\Z^2 \setminus
        \{0\}} (k\cdot
        \widehat m'_k(z))^*
        H_k(z-z') k\cdot
        \widehat m'_k(z') \dd
        z \dd z'.
    \end{align}
    This yields
    \begin{align}
      \Bigg| \int_{\T_\ell^2 \times \R} |h[m']|^2 \d3 -
      \frac{t}{2\ell^2} \sum_{k \in \frac{2\pi}{\ell }\Z^2 
      \setminus \{0\}} \frac{|k\cdot \widehat{ \overline
      {m}}'_k|^2}{|k|} \Bigg| \lesssim \frac{t^2}{2\ell^2}
      \int_{\R}\sum_{k \in 
      \frac{2\pi}{\ell }\Z^2} |k\cdot \widehat
      m'_k(z)|^2 \dd z, 
    \end{align}
    which proves the first equality. The second equality follows as in
    \eqref{eq:lo}.
\end{proof}

\begin{proof}[Proof of Lemma \ref{le:reduction-of-E}]
  We invoke Theorem \ref{th:h-m2d} to obtain a lower bound for the
  stray field energy. Combining \eqref{eq:h-split-approx} with
  \eqref{eq:h-m_3} and neglecting the non-negative term
  $\int_{\T_\ell^2 \times \R} |h[m']|^2 \d3$, we get 
    \begin{align}
        \begin{aligned}\label{eq:h-lb}
          \int_{\T_\ell^2} |h[m]|^2 \d3 &\lupref{eq:h-split-approx}\ge
          \int_{\T_\ell^2 \times \R} |h[m_3 e_3]|^2 \d3
          -Ct^2\int_{\T_\ell^2 \times (0,t)} |\nabla m|^2 \d3\\
          &\lupref{eq:h-m_3}\ge \int_{\T_\ell^2\times (0,t)} m_3^2 \d3
          - \frac{t^2}{2} \int_{\T_\ell^2} |\nabla^{1/2}\overline
          m_3|^2 \dx2\\
          & \qquad \qquad -Ct^2\int_{\T_\ell^2 \times (0,t)} |\nabla
          m|^2 \d3,
        \end{aligned}
    \end{align}
    for some universal constant $C>0$.
    Note that estimating $\int_{\T_\ell^2 \times \R} |h[m']|^2 \d3$ by
    zero is reasonable, since \eqref{eq:h-m_p} shows that the term is
    controlled by the exchange and anisotropy energy at lower order.
    Inserting \eqref{eq:h-lb} into the energy $E$ yields
    \begin{align}
    	\begin{aligned}\label{eq:E_lb_1}
          E[m] &\lupref{eq:E}= \Int_{\T_{\ell}^2\times (0, t)}
          \left(|\nabla m|^2 + Q (m_1^2 +m_2^2) - 2
            m_3h_{\mathrm{ext}}\right)\d3 + \Int_{\T_{\ell}^2 \times
            \R} |h|^2 \d3.\\
          &\lupref{eq:h-lb}\ge \Int_{\T_\ell^2\times (0,t)}
          \left(|\nabla m|^2 + Q (m_1^2 + m_2^2) - 2
            m_3h_{\mathrm{ext}}\right)\d3 +\Int_{\T_\ell^2\times
            (0,t)}{m}_3^2\d2\\
          &\quad - \frac{t^2}{2}
          \Int_{\T_\ell^2}|\nabla^{1/2}\overline{m}_3|^2\d2 -
          Ct^2\Int_{\T_{\ell}^2\times (0, t)} |\nabla
          m|^2 \d3.
    	\end{aligned}
	\end{align}
	The constraint $|m|=1$ allows to combine the leading order of
        the stray field energy with the anisotropy energy which leads
        to constant contribution and a renormalized anisotropy term
    \begin{align}\label{eq:regroup}
        \begin{aligned}
          \int_{\T_\ell^2 \times (0,t)}Q(m_1^2 + m_2^2) \dd x +
          t\int_{\T_\ell^2} m_3^2 \dd x &= \ell^2 t + \int_{\T_\ell^2
            \times (0,t)} (Q-1)(m_1^2+ m_2^2) \dd x.
        \end{aligned}
    \end{align}
    Finally, we insert \eqref{eq:regroup} into \eqref{eq:E_lb_1} to
    extract the leading order constant $\ell^2 t$ and conclude the
    claim of Lemma \ref{le:reduction-of-E}
	\begin{align}
	\label{eq:E_lb_2}
    	\begin{aligned}
          E[m] &\ge \ell^2 t + \Int_{\T_\ell^2\times (0,t)}
          \left(|\nabla m|^2 + (Q-1) (m_1^2 + m_2^2) - 2
            m_3h_{\mathrm{ext}}\right)\d3\\
          &\qquad - \frac{t^2}{2}
          \Int_{\T_\ell^2}|\nabla^{1/2}\overline{m}_3|^2\d2 
          -Ct^2\Int_{\T_{\ell}^2\times (0, t)} |\nabla m|^2 \d3,
    	\end{aligned}
	\end{align}
        which completes the proof.
\end{proof}

\section{Proofs for the full energy $E$}
\label{sec:proofs-E}
The proofs for the full energy $E$ are based on the arguments in the
proofs for the reduced energy $F$. We recommend to read section
\ref{sec:proofs-F} first.

Under mild assumptions on $\ell, t,Q$ and $h_{\rm ext}$, weaker than those of Theorems \ref{th:SUBcrit-E} -- \ref{th:CRIT-E}, Lemma \ref{le:main_ineq} and Theorem \ref{th:h-m2d} yield the following estimates for the rescaled energy $J$.
\begin{lem}
    There are universal constants $C,\delta > 0$ such that for $(\ell,t,Q, h_{\rm ext})$ which satisfy
    \begin{align}\label{eq:71-param}
        Q>1,  &&t < \delta \min\{1,\ell\}&& \text{and} &&\frac{\ell}{\sqrt{Q-1}} h_{\rm ext}(\ell x') = g(x')
    \end{align}
    for some $g \in L^1(\T^2)$, the rescaled energy $J$ (see \eqref{eq:J}) satisfies
    \begin{align}
    J[m] &\ge
    \left(1-Ct^2 - \frac{t}{4\sqrt{Q-1}}\right)\Int_{\T^2\times (0,1)} \Bigg(
    \eps|\nabla m|^2 + \frac{1}{\eps}
    (m_1^2 + m_2^2) \Bigg) \d3 \notag\\
    &\qquad + \frac{1}{2\eps t^2(Q-1)}\int_{\T^2 \times (0,1)} |\p_3 m|^2 \dx3 - 2\Int_{\T^2} g\overline
    m_3\d2 \label{eq:71-lb}\\
    &\qquad  - \frac{t}{\pi\sqrt{Q-1}}
    \log
    \left(c_*\max\left\{1,\min\left\{\frac{1}{\eps
        \int_{\T^2}|\nabla m_3| \dd x},
    \frac{1}{\eps}\right\}\right\}\right)\Int_{\T^2}|\nabla\overline{m}_3|\d2, \notag
    \end{align}
    for all $m \in H^1(\T^2\times(0,1);\S^2)$, where we have abbreviated $\eps:= \frac{1}{\ell \sqrt{Q-1}}$. Furthermore, for any $\overline m \in H^1(\T^2;\S^2)$ we have the upper bound
    \begin{align}
        \begin{aligned}
        \label{eq:71-ub}
            J[\chi_{(0,1)}\overline m] &\le (1+Ct^2) \Int_{\T^2} \left(
            \eps|\nabla \overline m|^2 + \frac{1}{\eps} (\overline m_1^2 + \overline m_2^2) \right)
            \d3 \\
            & \qquad - 2 \Int_{\T^2} \overline
            m_3 g\d2 - \frac{t}{2\sqrt{Q-1}}
            \Int_{\T^2}|\nabla^{1/2}\overline{m}_3|^2\d2.
        \end{aligned}
    \end{align}
\end{lem}

\begin{proof}
    The lower bound for $E$ in Lemma \ref{le:reduction-of-E} implies
    a lower bound for the rescaled energy $J$ 
    \begin{align}
    \begin{aligned}\label{eq:71-a}
    J[m] &= \frac{E[m(\ell \cdot, \ell \cdot, t \cdot)] - \ell^2
        t}{\ell t \sqrt{Q-1}} \lupref{eq:E-m2d}\ge
    \left(1-Ct^2\right)\Int_{\T^2\times (0,1)} \Bigg(
    \frac{1}{\ell\sqrt{Q
            -1}}|\nabla' m|^2\\
    &\qquad + \frac{\ell}{t^2\sqrt{Q-1}}|\p_3 m|^2 + \ell\sqrt{Q-1}
    (m_1^2 + m_2^2) \Bigg) \d3\\
    &- \frac{2\ell}{\sqrt{Q-1}}\Int_{\T^2} \overline
    m_3(x')h_{\mathrm{ext}}(\ell x')\d2 - \frac{t}{2\sqrt{Q-1}}
    \Int_{\T^2}|\nabla^{1/2}\overline{m}_3|^2\d2.
    \end{aligned}
    \end{align}
    We insert
    \begin{align*}
    \frac{\ell}{\sqrt{Q-1}} h_{\rm ext}(\ell x') = g(x') \qquad
    \text{and} \qquad \eps= \frac{1}{\ell \sqrt{Q-1}}
    \end{align*}
    to obtain
    \begin{align}
    \begin{aligned}\label{eq:71-b}
    J[m] &\lupref{eq:E-m2d}\ge
    \left(1-Ct^2\right)\Int_{\T^2\times (0,1)} \Bigg(
    \eps|\nabla' m|^2 + \frac{1}{\eps t^2(Q-1)}|\p_3 m|^2 + \frac{1}{\eps}
    (m_1^2 + m_2^2) \Bigg) \d3\\
    &\qquad - 2\Int_{\T^2} g\overline
    m_3\d2 - \frac{t}{2\sqrt{Q-1}}
    \Int_{\T^2}|\nabla^{1/2}\overline{m}_3|^2\d2.
    \end{aligned}
    \end{align}
    In view of \eqref{eq:71-param} we may assume that
    \begin{align}\label{eq:paramss}
        (1-Ct^2)\left(\frac{1}{\eps t^2(Q-1)} - \eps\right) \ge \frac{1}{2\eps t^2(Q-1)}.
    \end{align}
    Hence, applying Lemma \ref{le:main_ineq} to the last term in \eqref{eq:71-b} and inserting \eqref{eq:paramss} we arrive at \eqref{eq:71-lb}. The proof for the upper bound \eqref{eq:71-ub} is simpler and analogous to the arguments that led to \eqref{eq:71-b}.
\end{proof}

\subsection{Proof of Theorem \ref{th:SUBcrit-E}}
    \label{sec:SUBcrit-proof-E}

It is possible to invoke the lower bound for $F$ on slices $\{x_3=const\}$ to
obtain the lower bound for the full (rescaled) energy $J$. However, we will
not pursue this option. Instead, we apply the $H^{1/2}$-bound of
Lemma \ref{le:main_ineq} directly and extend the arguments of the
previous section. The reason is related to the fact that
$C^\infty(\T^2\times(0,1); \S^2)$ is not dense in
$H^1(\T^2\times(0,1); \S^2)$, which can be seen by considering $f(x)=\frac{x}{|x|}$ (see
\cite{bethuel1988density,bethuel91sobolev,hang2001topology}). Hence,
evaluating Sobolev functions on slices $\{x_3=const\}$ and
confirming that the constraint $|m|=1$ still holds requires to use
the precise representative of a Sobolev function and gets
rather technical.

    \begin{proof}[Proof of the lower bound and compactness in Theorem
      \ref{th:SUBcrit-E}]
      Our starting point is the lower bound \eqref{eq:71-lb}. It turns out to be more convenient to use the parameter $\eps=\frac{1}{\ell \sqrt{Q-1}}$ instead of $\ell$. We first note that for $\eps<1$ the last term in \eqref{eq:71-lb} may be estimated with the aid of \eqref{eq:A3} and \eqref{eq:A2} by
      \begin{align}
      \begin{aligned}
      \label{eq:h12-term}
          \log
          \left(c_*\max\left\{1,\min\left\{\frac{1}{\eps
              \int_{\T^2}|\nabla m_3| \dd x},
            \frac{1}{\eps}\right\}\right\}\right)\Int_{\T^2}|\nabla\overline{m}_3|\d2\\
            \le
            \log
            \left(c_*/\eps\right)\Int_{\T^2\times (0,1)} \Bigg(
            \eps|\nabla m|^2 + \frac{1}{\eps}
            (m_1^2 + m_2^2) \Bigg) \d3.
            \end{aligned}
      \end{align}
      For $Q$ and $(\ell_k, t_k, h_{{\rm ext},k})$ satisfying \eqref{eq:SUBcrit-E-regime}, we abbreviate
      \begin{align}\label{eq:regi}
       \eps_k:= \frac{1}{\ell_k\sqrt{Q-1}} \to 0 && \text{and} && g_k:=\frac{\ell_k}{\sqrt{Q-1}} h_{{\rm ext},k}(\ell_k \cdot) \to g,
      \end{align} and note that
      \begin{align}\label{eq:t_k_null}
          t_k^2 + \frac{t_k}{\sqrt{Q-1}} \lupref{eq:SUBcrit-E-regime}\longrightarrow 0.
      \end{align}
      Inserting \eqref{eq:SUBcrit-E-regime} and \eqref{eq:h12-term} -- \eqref{eq:t_k_null} into the lower bound \eqref{eq:71-lb}, we deduce that for any $\gamma > 0$ and sufficiently large $k \ge k_0(\gamma)$, we have
    \begin{align}
        \begin{aligned}
        J_k[m]&\ge \left(1- \frac{\lambda}{\lambda_c} - \gamma\right)\Int_{\T^2\times (0,1)} \left(\eps_k|\nabla m|^2 + \frac{1}{\eps_k} (m_1^2 + m_2^2) \right) \d3 \label{eq:pr31-lb-c}\\
        &\qquad + \frac{1}{2\eps_k t_k^2(Q-1)}\int_{\T^2 \times(0,1)}
        |\p_3m|^2 \d3 -2 \Int_{\T^2} \overline m_3
        g_k\d2.
        \end{aligned}
    \end{align}
Note that \eqref{eq:pr31-lb-c} for $2\gamma \le 1-\frac{\lambda}{\lambda_c}$ and sufficiently large $k$ implies
\begin{align}\label{eq:asdf}
\int_{\T^2 \times (0,1)} \left( m_1^2 + m_2^2
\right) \d3 \lesssim
\frac{\eps_k}{(\lambda_c - \lambda)}\left(J_k[m] +
2\|g_k\|_{L^1}\right).
\end{align}
Using Poincaré's inequality and \eqref{eq:pr31-lb-c} for $\gamma < 1-\frac{\lambda}{\lambda_c}$ again, we get
\begin{align}
\begin{aligned}\label{eq:I-lb-p7}
\int_{\T^2\times(0,1)} |m-\chi_{(0,1)}\overline m|^2 \d3 &\lesssim \int_{\T^2
    \times (0,1)} |\p_3 m|^2 \d3 \\
&\lupref{eq:pr31-lb-c}\lesssim \eps_k t_k^2(Q-1)\left(\limsup_{k
    \to \infty} J_k[m] + 2\|g_k\|_{L^1}\right).
\end{aligned}
\end{align}
Furthermore, applying \eqref{eq:A2} and \eqref{eq:A3} to \eqref{eq:pr31-lb-c} again implies the lower bound
\begin{align}\label{eq:I-lb-p88}
J[m]\ge  2\left(1 - \frac{\lambda}{\lambda_c} -
\gamma \right)\int_{\T^2} |\nabla' \overline m_3| \d2 - 2 \int_{\T^2}
\overline m_3 g_k\d2.
\end{align}

In order to prove compactness, let
$ m^{(k)} \in H^1(\T^2 \times (0,1); \S^2)$ with
$\limsup_{k \to \infty} J[m_k] < \infty$. Since $\lambda < \lambda_c$
and $g_k \to g $ in $L^1(\T^2)$, inequality \eqref{eq:I-lb-p88} implies a uniform
bound on $\overline{m}_3^{(k)}$ in $BV(\T_2)$. A standard compactness
argument implies that $\overline{ m}_3^{(k)} \to \overline m_3^{}$ in
$L^1(\T^2)$ for a subsequence (not relabelled) and some
$\overline m_3 \in BV(\T^2)$.  We will now show that in fact
$m^{(k)} \to \chi_{(0,1)}\overline m_3 e_3$ in
$L^1(\T^2\times(0,1);\R^3)$. Indeed, the triangle inequality yields
\begin{align}
\begin{aligned}\label{eq:I-lb-p10}
&\Int_{\T^2\times (0,1)}|m^{(k)} - \chi_{(0,1)}\overline m_3 e_3|
\d3 \le \Int_{\T^2\times (0,1)} \left( |{m}_1^{(k)}|^2 +
|{m}_2^{(k)}|^2 \right)^{1/2}
\d3\\
&\qquad + \Int_{\T^2\times (0,1)} |m_3^{(k)} - \chi_{(0,1)}\overline
m_3^{(k)}| \d3 + \Int_{\T^2}|\overline m^{(k)}_3 -
\overline m_3| \d3,
\end{aligned}
\end{align}
and we already know that the last term on the right hand side of \eqref{eq:I-lb-p10} vanishes. Furthermore, the first term vanishes due to \eqref{eq:asdf} and the second one due to \eqref{eq:I-lb-p7} and \eqref{eq:SUBcrit-E-regime}. This completes the proof of the compactness statement.

The liminf inequality is easily obtained from the lower bound \eqref{eq:I-lb-p88}. Indeed, let $ m^{(k)} \in H^1(\T^2 \times (0,1); \S^2)$ with $ m^{(k)} \to m$ in
$L^1(\T^2 \times (0,1))$. By Jensen's inequality, we also have
$\overline{m}^{(k)} \to \overline m$ in $L^1(\T^2)$. By lower
semicontinuity of the BV seminorm and since $\gamma$ was arbitrary, we obtain from \eqref{eq:I-lb-p88} in the limit
\begin{align*}
\liminf_{k \to \infty}J_k[m^{(k)}] \ge
\left(1-\frac{\lambda}{\lambda_c}\right) \Int_{\T^2}|\nabla'
\overline {m}_3|\d2 - 2 \Int_{\T^2} \overline{m}_3 g \d2. 
\end{align*}
\end{proof}

It remains to prove the upper bound for the $\Gamma$-convergence. As
it turns out, we may use the recovery sequence for the reduced energy
$F$ also for the full energy $E$ (up to thickening).

\begin{proof}[Construction of the recovery sequence in Theorem \ref{th:SUBcrit-E}]
  Let $\lambda \le \lambda_c$ and $\overline m \in BV(\T^2; \{\pm e_3\})$. Furthermore, let $\overline m_{\eps} \in H^1(\T^2; \S^2)$ denote the recovery sequence
  for $F_{\eps,\lambda}$ from Lemma \ref{le:constr}. With the notation \eqref{eq:regi} we set
  \begin{align}
      m^{(k)}(x',x_3) := \chi_{(0,1)}(x_3)\overline m_{\eps_k}(x') \qquad \text{for }(x',x_3) \in \T^2 \times (0,1)
  \end{align}
  and claim that
  \begin{align}
      \limsup_{k \to \infty} J_{k}[m^{(k)}] \le J_*[\overline m].
  \end{align}
 Inserting the abbreviation $\lambda_k:= \frac{t_k |\log(\eps_k)|}{4\sqrt{Q-1}}$ into the upper bound \eqref{eq:71-ub}, we obtain
\begin{align}\label{eq-Elt-ub}
  J_k[m^{(k)}] &\le \left(1+Ct_k^2\right)\Int_{\T^2} \left(\eps_k|\nabla  \overline m_{\eps_k}|^2 + \frac{1}{\eps_k} (\overline m_{\eps_k,1}^2 + \overline m_{\eps_k,2}^2) \right) \d2 \notag\\
  &\quad - \frac{2\lambda_k}{|\log \eps_k|} \Int_{\T^2}|\nabla^{1/2}\overline
  m_{{\eps_k},3}|^2 - 2 \Int_{\T^2} g_k\overline m_{{\eps_k},3}\d2\\
  & = 2F_{\eps_k, \lambda_k}[\overline
  m_{\eps_k}] -2 \Int_{\T^2} g_k\overline{m}_{\eps_k,3} \d2+ Ct_k^2 \Int_{\T^2} \left(\eps_k|\nabla  \overline m_{\eps_k}|^2 + \frac{1}{\eps_k} (\overline m_{\eps_k,1}^2 + \overline m_{\eps_k,2}^2) \right) \d2.\notag
\end{align}
We have shown in Lemma 5.3 that
\begin{align}
    \Int_{\T^2} \left(\eps_k|\nabla  \overline m_{\eps_k}|^2 + \frac{1}{\eps_k} (\overline m_{\eps_k,1}^2 + \overline m_{\eps_k,2}^2) \right) \d2 \to 2\int_{\T^2}|\nabla \overline m_3| \d2  < \infty.
\end{align}
Since \eqref{eq:SUBcrit-E-regime} implies $t_k \to 0$, $\lambda_k \to \lambda < \lambda_c$ and $g_k \to g$ in $L^1(\T^2)$, the claim follows upon applying Lemma \ref{le:constr} to \eqref{eq-Elt-ub}
\begin{align}
  \limsup_{k \to \infty}J_k[
  m^{(k)}]&\le 2 F_{*, \lambda}[\overline m]
  -2
  \Int_{\T^2} g\overline{m}_3 \d2.
\end{align}
\end{proof}

\subsection{Proof of Theorem \ref{th:SUPERcrit-E}}
    \label{sec:SUPERcrit-proof-E}
    \begin{proof}[Proof of Theorem \ref{th:SUPERcrit-E}]
        We begin with the proof of the lower bound for which we use \eqref{eq:71-lb} with $g=0$. For sufficiently small $\delta$, the regime \eqref{eq:re2} implies
        \begin{align}\label{eq:smallpar}
        Ct^2 + \frac{t}{\sqrt{Q-1}} \lupref{eq:re2}{\lesssim} C\delta^2 +
        \delta \lesssim \delta. 
        \end{align}
        Analogous to the argument that lead from \eqref{eq:lower-bound} to
        \eqref{eq:lower-bound1}, but now with \eqref{eq:smallpar} instead of
        \eqref{eq:delta-lambda}, we reduce \eqref{eq:71-lb} to the case
        \begin{align}
        \begin{aligned}\label{eq:J-lb-p6}
        J[m]&\ge \left(1-Ct^2 -
        \frac{t}{\sqrt{Q-1}}\right)\Int_{\T^2\times (0,1)} \left(
        \eps|\nabla
        m|^2 + \frac{1}{\eps} (m_1^2 + m_2^2) \right) \d3 \\
        &\qquad + \frac{1}{2\eps t^2(Q-1)} |\p_3m|^2 \d3 -\frac{t\log \left(c_*\frac{1}{\eps\int_{\T^2}|\nabla
                \overline m_3| \dd x}\right)}{\pi\sqrt{Q-1}}\,
        \Int_{\T^2}|\nabla \overline m_3|\d2.
        \end{aligned}
        \end{align}
        Abbreviating the energetic cost for $m$ to deviate from the optimal
        Bloch wall profile by
        \begin{align}\label{eq:Deps}
        D_\eps[m]:= \int_{\T^2 \times (0,1)} \left( \eps |\nabla m|^2 +
        \frac{1}{\eps}(1-m_3^2) \right) \d3 - 2\int_{\T^2} |\nabla
        \overline m_3| 
        \d3, 
        \end{align} 
        and inserting
        $\mu := \eps e^{2\pi t^{-1}\sqrt{Q-1}} \int_{\T^2} |\nabla
        \overline m_3| \dd x$
        and
        $c_{**}: = c_*e^{2\pi(1+C t \sqrt{Q-1})} \lupref{eq:re2} \sim 1$
        into the lower bound \eqref{eq:J-lb-p6} we get
        \begin{align}
        \begin{aligned}\label{eq:J-lb-p8}
        J[m]&\ge  \left(1-Ct^2 - \frac{t}{\sqrt{Q-1}}\right)D_\eps[m]+
        \frac{1}{2\eps t^2(Q-1)}\int_{\T^2 \times(0,1)} |\p_3m|^2 \d3\\ 
        &\qquad-\frac{\log \left(c_{**}/\mu\right)}{\pi}\, \mu\, t \ell e^{-2\pi t^{-1}\sqrt{Q-1}}.
        \end{aligned}
        \end{align}
        Minimizing in $\mu >0$ then yields the lower bound
        \begin{align}
        \begin{aligned}\label{eq:J-lb-p10}
        J[m]&\gtrsim -c_{**}t\ell e^{-2\pi t^{-1}\sqrt{Q-1}} \gtrsim -t\ell
        e^{-2\pi t^{-1}\sqrt{Q-1}}.
        \end{aligned}
        \end{align}
        It remains to construct a sequence that achieves the optimal
        scaling. Let $m_{\eps, N}$ denote the function constructed in Lemma \ref{le:F-SUPERcrit-ub} and define $m_{\eps, N} := \chi_{(0,1)}\overline m_{\eps,N}$. We insert \eqref{eq:construction} and \eqref{eq:construction_2} into \eqref{eq:71-ub} and use that \eqref{eq:re2} implies
        $t^2 \lesssim \frac{t}{\sqrt{Q-1}}$ to deduce
        \begin{align}\label{eq:construction-E}
        J[m_{\eps,N}] \le 4N\left(1+Ct^2 - \frac{t\log(\tfrac{c}{2 \eps
                N})}{2\pi\sqrt{Q-1}}\right)\lupref{eq:re2} \le 4N\left(1 - \frac{t\log(\tfrac{\tilde c}{2 \eps
                N})}{2\pi\sqrt{Q-1}}\right)
        \end{align}
        for some universal $\tilde c > 0$. Optimizing in $N$ leads to
        \begin{align}\label{eq:N-E}
	        N := 2\left\lfloor \ell \sqrt{Q-1}\frac{e^{-2\pi {t^{-1}} \sqrt{Q-1}}}{K}\right\rfloor,
        \end{align}
        which satisfies $N \ge 2$ due to \eqref{eq:re2} and is hence admissible. Inserting \eqref{eq:N-E} into \eqref{eq:construction-E}, and taking $K\ge \frac{8}{\tilde c}$, we conclude that the function $m_{\eps,N}$
        indeed achieves the optimal scaling
        \begin{align*}
        \begin{aligned}
         J[m_{\eps,N}]\lesssim - t\ell e^{-2\pi {t^{-1}}
            \sqrt{Q-1}}.
        \end{aligned}
        \end{align*}
    \end{proof}
    
    \subsection{Proof of Proposition \ref{prop:properties}}
    \begin{proof}[Proof of Proposition \ref{prop:properties}]
        Let $m$ satisfy \eqref{eq:t-opt-gamma}. Then \eqref{eq:smallpar} and \eqref{eq:J-lb-p8} imply $\mu \sim 1$
        and hence \eqref{eq:wl}
        \begin{align*}
        \Int_{\T^2} |\nabla \overline m_3| \d2 \sim \ell \sqrt{Q-1}
        e^{-2\pi {t^{-1}}\sqrt{Q-1}},
        \end{align*}
        where here and throughout the rest of this proof, the constants associated with $\lesssim, \gtrsim$ and $\sim$ may depend on $\gamma$.
        In turn, inserting \eqref{eq:t-opt-gamma}, \eqref{eq:wl} and \eqref{eq:smallpar} into
        \eqref{eq:J-lb-p8} implies \eqref{eq:bw}
        \begin{align*}
        D_\eps[m] \lupref{eq:J-lb-p8}\lesssim
        \frac{t}{\sqrt{Q-1}}\Int_{\T^2} |\nabla \overline m_3| \d2. 
        \end{align*}
        Furthermore, Poincaré's inequality, \eqref{eq:J-lb-p8}, \eqref{eq:t-opt-gamma} and $\mu \sim 1$ yield \eqref{eq:t2-i1}
        \begin{align*}
        \begin{aligned}
        \int_{\T^2 \times (0,1)} |m - \chi_{(0,1)}\overline m|^2 \d3
        &\lesssim \int_{\T^2 \times (0,1)} |\p_3m|^2 \d3 \lupref{eq:J-lb-p8}\lesssim
	    t^3 \sqrt{Q-1}\,e^{-2\pi {t^{-1}}\sqrt{Q-1}}.
        \end{aligned}
        \end{align*}
        Finally, we deduce \eqref{eq:oop} from \eqref{eq:Deps}, \eqref{eq:wl} and \eqref{eq:bw}
        \begin{align}
        \begin{aligned}
        \int_{\T^2 \times (0,1)} \left( m_1^2 + m_2^2 \right) \d3
        &\lupref{eq:Deps}\lesssim \eps \left(\int_{\T^2} |\nabla
        \overline m_3| \d2 + D_\eps[m]\right) \lesssim e^{-2\pi
            {t^{-1}}\sqrt{Q-1}},
        \end{aligned}
        \end{align}
        which completes the proof.
    \end{proof}

\subsection{Proof of Theorem \ref{th:CRIT-E}}
\begin{proof}[Proof of Theorem \ref{th:CRIT-E}]
    The proof is analogous to the proof of Theorem \ref{th:crit-F}.
\end{proof}

\paragraph{Acknowledgements}
The work of CBM was supported, in part, by NSF via grants DMS-1313687
and DMS-1614948. FN thanks the New Jersey Institute of Technology for
its hospitality during a visit in Newark and the Heidelberg Graduate
School of Mathematical and Computational Methods for the Sciences for
financial support. Furthermore, the authors thank Christof Melcher for
stimulating questions that led to Remark \ref{eq:nat-scales} and
Pierre Bousquet for pointing us to \cite{hang2001topology}.

\begin{appendices}
\numberwithin{equation}{section}
\section{}
    \label{sec:appendix}

    We give a proof for the continuity of
    $\eps \mapsto \lambda_c(\eps)$, the critical value of
    $\lambda$. Furthermore, we record a few well-known results that
    are used in the paper. For the convenience of the reader, we also
    give the proofs.

    For $0 < \eps < 1$, we define the critical value of $\lambda$
    where $\min F_{\eps,\lambda}$ becomes negative as
\begin{align}\label{eq:def}
\lambda_c(\eps):= \inf\{\lambda: \, \min F_{\eps,\lambda}<0\}.
\end{align}
\begin{lem}\label{le:cont-lambda}
    The function $\lambda_c:(0,1) \to \R$ (see \eqref{eq:def}) is Lipshitz-continuous on compact subsets of $(0,1)$.
\end{lem}
\begin{proof}
    The main idea is to express $\lambda_c$ as the infimum over $\lambda_{c,m}$, where $m$ is held fixed (see \eqref{eq:def2}) and to deduce regularity of $\lambda_c$ from the regularity of $\lambda_{c,m}$. We define
    \begin{align}
    X:= \{m \in H^1(\T^2;\S^2):\, m \text{ is not constant}\}
    \end{align} and introduce, for any $m \in X$, the function
    \begin{align}\label{eq:def2}
    \lambda_{c,m}: (0,1) \to \R, \qquad \eps \mapsto \lambda_{c,m}(\eps):= \inf\{\lambda: \, F_{\eps,\lambda}[m]<0\}.
    \end{align}
    Note that $F_{\eps,\lambda}[m] \ge 0$ if $m$ is constant and that $\lambda \mapsto F_{\eps,\lambda}[m]$ is strictly monotone (for $\eps$ and $m \in X$ fixed). Hence, we may rewrite
    \begin{align}
    \begin{aligned}\label{eq:la}
    \lambda_c(\eps) &= \inf\{\lambda: \, \exists m \in X \text{ s.t. } F_{\eps,\lambda}[m]<0\}\\
    &=\inf\{\lambda: \, \exists m \in X \text{ s.t. } \lambda > \lambda_{c,m}(\eps)\} =\inf_{m \in X} \lambda_{c,m}(\eps).
    \end{aligned}
    \end{align}
    \textit{Step 1:} Regularity of $\lambda_{c,m}$. We claim that
    \begin{align}\label{eq:est}
    \left| \frac{d}{d \eps}\lambda_{c,m}(\eps)\right| \le \left(1 + \frac{1}{|\log \eps|}\right)\frac{\lambda_{c,m}(\eps)}{\eps} \qquad \text{for all }m \in X.
    \end{align}
    To prove \eqref{eq:est}, fix $m \in X$ and abbreviate
    \begin{align*}
      a = \int_{\T^2} |\nabla m|^2 \dd x, \qquad b:= \int_{\T^2} (1-m_3^2) \dd x \qquad c:= \int_{\T^2} |\nabla^{1/2}m_3|^2 \dd x,
    \end{align*}
    so that $F_{\eps,\lambda}[m] = \frac{\eps}{2}a + \frac{b}{2\eps} - \frac{\lambda}{|\log \eps|}c$ with partial derivatives
    \begin{align*}
    \p_\eps F_{\eps,\lambda}[m] = \frac{a}{2} - \frac{b}{2\eps^2} - \frac{\lambda c}{\eps|\log \eps|^2} && \text{and} && \p_\lambda F_{\eps,\lambda}[m] = - \frac{c}{|\log \eps|}.
    \end{align*}
    By continuity of $(\eps,\lambda) \mapsto F_{\eps,\lambda}[m]$ and strict monotonicity in $\lambda$, we deduce from \eqref{eq:def2} that $\lambda_{c,m}$ satisfies $F_{\eps,\lambda_{c,m}(\eps)}[m]=0$ for all $\eps \in (0,1)$ and, furthermore, that it is the only function with this property. Then the implicit function theorem asserts that $\lambda_{c,m}$ is $C^1((0,1))$ with
    \begin{align}
    \begin{aligned}\label{eq:A8}
    \frac{d}{d \eps}\lambda_{c,m}(\eps) &= - \left(\p_\lambda F_{\eps,\lambda}[m] \right)^{-1}\p_\eps F_{\eps,\lambda}[m]\\
    &= \frac{|\log \eps|}{c}\left(\frac{a}{2} - \frac{b}{2\eps^2} - \frac{\lambda c}{\eps|\log \eps|^2}\right).
    \end{aligned}
    \end{align}
    Inserting the identity $F_{\eps,\lambda_{c,m}(\eps)}[m] = \frac{\eps}{2}a + \frac{b}{2\eps} - \frac{\lambda_{c,m}(\eps)}{|\log \eps|}c = 0$ into \eqref{eq:A8}, we obtain the estimate
    \begin{align}
    \left| \frac{d}{d \eps}\lambda_{c,m}(\eps)\right| \le  \frac{|\log \eps|}{\eps}\left( \frac{\frac{\eps a}{2} + \frac{b}{2 \eps}}{c} \right) +  \frac{\lambda_{c,m}(\eps)}{\eps|\log \eps|} \le \left(1 + \frac{1}{|\log \eps|}\right)\frac{\lambda_{c,m}(\eps)}{\eps}
    \end{align}
    which completes the proof of \eqref{eq:est}.
    
    \textit{Step 2: Regularity of $\lambda_c$.}  The metric space
    $(X,\|\cdot\|_{H^1})$ is separable as a subset of the separable
    metric space $H^1(\T^2; \R^3)$ and hence there exists a dense
    countable subset $\{m_n: \, n \in \N\} \subset X$. Let
    $\delta \in (0,1/2)$ and define
    $M:= \sup_{\eps \in [\delta,1-\delta]}|\lambda_{c,m_1}(\eps)|
    {< +\infty}$. Then the functions
    \begin{align}
    g_n: [\delta, 1-\delta] \to \R, \qquad \eps \mapsto g_n(\eps) = \min \{\lambda_{c,m_n}(\eps),M\}
    \end{align}
    are Lipschitz-continuous for all $n \in \N$. Furthermore, by \eqref{eq:est}, their Lipschitz-constant is bounded by $ \delta^{-1}(1 + \frac{1}{|\log \delta|})M$ (independent of $n \in \N$). Define the sequence of functions $f_k:= \min_{1 \le n \le k}g_n$ and observe that
    \begin{itemize}
        \item[(i)] $\| f_k \|_{C^0([\delta, 1-\delta])} \le M $ for all $k \in \N$,
        \item[(ii)] $f_k$ is Lipschitz continuous with Lipschitz
          constant bounded by
          $\delta^{-1}(1 + \frac{1}{|\log \delta|})M$,
        \item[(iii)] $f_k(\eps) \to \lambda_c(\eps)$ as $k \to \infty$
          for all $\eps \in [\delta,1-\delta]$.
    \end{itemize}
    The last point follows from \eqref{eq:la}, the density of
    $\{m_n: \, n \in \N\} \subset X$ and continuity of
    $m \mapsto F_{\eps,\lambda}[m]$.  Now the compact embedding
    $C^{0,1}([\delta,1-\delta]) \hookrightarrow
    C^0([\delta,1-\delta])$ implies that $f_k {\to} f$
    {uniformly} for some
    $f \in C^{0,1}([\delta,1-\delta])$ with Lipschitz constant bounded
    by $\delta^{-1}(1 + \frac{1}{|\log \delta|})M$. By uniqueness of
    the limit we conclude that $f=\lambda_c$, which completes the
    proof.
\end{proof}
It is well-known that if $m \in H^1$ takes values in $\S^2$, this
implies certain estimates for the gradient $\nabla m$ {(see, e.g.,
  \cite{kohn07icm})}. Since these estimates are used frequently
throughout our paper, we record them in the following Lemma.

\begin{lem}
  Let $\Omega \subset \R^n$ be open and $m \in H^1(\Omega,\S^2)$. Then
  for every $\eps > 0$ we have 
    \begin{align}
        \label{eq:A1}
      &(i)
      &\frac{|\nabla m_3|^2}{1-m_3^2} 
      &\le |\nabla m|^2  
      &&\text{for
         a.e. }x
         \in \Omega
         \text{
         with }
         |m_3(x)|<1,
      \\ 
      &(ii)&\quad |\nabla m_3| 
      &\le {\eps \over 2} |\nabla m|^2 +
        {1-m_3^2 \over 2 \eps} \label{eq:A2}  
      &&\text{for
         a.e. }x
         \in
         \Omega. 
    \end{align}
\end{lem}
\begin{proof}
  To prove $(i)$, we apply the weak chain rule to the constraint $|m|^2 = 1$, which yields
\begin{align*}
  - m_3 \nabla m_3 = m_1 \nabla m_1 + m_2 \nabla m_2
\end{align*}
a.e. in $\Omega$.  After squaring both sides and applying the
$n$-dimensional Cauchy-Schwarz inequality, we obtain
\begin{align*}
  m_3^2 |\nabla m_3|^2 
  \le\left(m_1^2 + m_2^2\right)(|\nabla m_1|^2 +
    |\nabla m_2|^2). 
\end{align*}
Finally we add $\left(m_1^2+m_2^2\right) |\nabla m_3|^2$ to both
sides. Since $|m|^2=1$, this yields
\begin{align*}
  |\nabla m_3|^2 \le (1-m_3^2)|\nabla m|^2,
\end{align*}
and hence proves \eqref{eq:A1}.

We turn to the proof of $(ii)$. Since $\nabla m_3=0$ almost everywhere
on the set $\{x \in \Omega:\, |m_3(x)|=1\}$, it remains to prove
\eqref{eq:A2} on $\{x \in \Omega, |m_3(x)|<1\}$. This follows from
\eqref{eq:A1} upon an application of Young's inequality
\begin{align*}
  \quad 2|\nabla m_3| \le \frac{\eps |\nabla m_3|^2}{1-m_3^2} +
  {1-m_3^2 \over \eps}  \lupref{eq:A1}\le \eps |\nabla m|^2
  + {1-m_3^2 \over \eps},  
\end{align*}
which concluded the proof.
\end{proof}

In the following Lemma, we record a consequence of Jensen's inequality
for the gradients of $e_3$-averages.

\begin{lem}\label{le:jensen_gradient}
  For every $p \in [1,\infty)$ and every
  $f \in W^{1,p}(\T^2\times (0,1))$, we have
    \begin{align}\label{eq:A3}
      \Int_{\T^2}\left|\nabla'\int_0^1f(x',x_3)\d[x_3]{1}\right|^p
      \d[x']{2} \le \Int_{\T^2\times (0,1)}\left|\nabla'f \right|^p
      \d{3}. 
    \end{align}
\end{lem}

\begin{proof}
  Assume for a moment that $f \in C^\infty(\T^2\times (0,1))$. Since
  $|\cdot|^p:\R^2 \to \R$ (the $p$-th power of the euclidean norm) is
  a convex function, an application of Jensen's inequality (for
  two-dimensions) then yields
    \begin{align}\label{eq:jensen_simple}
      \left|\int_0^1 \nabla' f(x',x_3)\d[x_3]{1}\right|^p \d[x']{2}
      \le \Int_0^1\left|\nabla'f(x',x_3) \right|^p \d[x_3]1 \qquad
      \text{for every }x' \in \T^2. 
    \end{align}
    For $f \in C^\infty(\T^2\times (0,1))$, we can change the order of
    integration and differentiation, so that \eqref{eq:A3} follows
    from \eqref{eq:jensen_simple} after integration over $\T^2$
    \begin{align}
    \begin{aligned}\label{eq:jensen_gradient_smooth}
      \Int_{\T^2}\left|\nabla'\int_0^1f(x',x_3)\d[x_3]{1}\right|^p
      \d[x']{2}
      &= \Int_{\T^2}\left|\int_0^1 \nabla'
        f(x',x_3)\d[x_3]{1}\right|^p \d[x']{2}\\ 
      &\lupref{eq:jensen_simple}\le
      \Int_{\T^2}\Int_0^1\left|\nabla'f(x',x_3)\right|^p
      \d[x_3]{1}\d[x']{2}.
    \end{aligned}
    \end{align}
    Finally, \eqref{eq:jensen_gradient_smooth} extends to any $f\in W^{1,p}(\T^2\times(0,1))$ by a standard approximation argument using lower semi-continuity of the $W^{1,p}(\T^2)$ norm with respect to weak
        convergence of the $e_3$-averages.
\end{proof}

The next Lemma relates the real space formulation of the homogeneous
$H^{1/2}$-norm to its Fourier representation.
\begin{lem}\label{le-H12fourier}
    For every smooth function $f:\T_\ell^2 \to \R$, the following holds
    \begin{align}\label{eq:A4}
      \Int_{\T^2_\ell} |\nabla^{1/2} f|^2 \d2
      := \frac{1}{\ell^2}\sum_{k\in \frac{2\pi}{\ell}\Z^2} |k||\widehat f_k|^2 
      = \frac1{4\pi}\Int_{\T^2_\ell} \Int_{\R^2} \frac{|f(x) -
      f(y)|^2}{|x-y|^3}\dx2\dy2.  
    \end{align}
\end{lem}

\begin{proof}
  First we prove the identity
    \begin{align}\label{eq-fi}
      \Int_{\R^2}|e^{ik\cdot x} -1|^2 \frac{1}{|x|^3}\dx2 = 4\pi |k|
      \qquad \text{for every }k \in \tfrac{2\pi}{\ell} \Z^2. 
    \end{align}
    By scaling and rotational symmetry, we have
    \begin{align}\label{eq-fi-1}
      \Int_{\R^2}|e^{ik\cdot x} -1|^2 \frac{1}{|x|^3}\dx2 =
      |k|\Int_{\R^2}|e^{ix_1} -1|^2 \frac{1}{|x|^3}\dx2. 
    \end{align}
    We evaluate the last integral in polar coordinates. On
    substituting $\rho = \frac{r \cos \theta}{2}$, we obtain 
    \begin{align}
    \label{eq-fi-2}
    \Int_{\R^2}|e^{ix_1} -1|^2 \frac{1}{|x|^3}\dx2 &=
    \Int_{\R^2}|e^{\frac{ix_1}{2}} - e^{-\frac{ix_1}{2}}|^2
    \frac{1}{|x|^3}\dx2 =\Int_{0}^{2\pi} \Int_{0}^{\infty} 4
    \sin^2\left(\frac{r\cos
        \theta}{2}\right) \frac{1}{r^3} r \d[\theta]{1} \d[r]{1} \notag \\
    &= 2\Int_{0}^{2\pi} |\cos \theta| \d[\theta]{1} \Int_{0}^{\infty}
    \frac{\sin^2 \rho}{\rho^2} \d[\rho]{1} = 4\pi. 
    \end{align}
    Together, \eqref{eq-fi-1} and \eqref{eq-fi-2} prove \eqref{eq-fi}.\\
    \\
    With \eqref{eq-fi} at hand, we will now prove \eqref{eq:A4}. By a
    variable transformation and Fubini's Theorem, we obtain
    \begin{align*}
    &\Int_{\T^2_\ell} \Int_{\R^2} \frac{|f(x) - f(y)|^2}{|x-y|^3}\dx2 \dy2
    =\Int_{\R^2} \Int_{\T^2_\ell} |f(z+y) - f(y)|^2 \dy2\frac{1}{|z|^3}\dz2.
    \end{align*}
    Rewriting the inner integral in Fourier space and using Fubini's
    Theorem again yields
    \begin{align*}
      &\Int_{\R^2} \Int_{\T^2_\ell} |f(z+y) - f(y)|^2 \dy2
        \frac{1}{|z|^3}\dz2 = \frac{1}{\ell^2}\Int_{\R^2} \sum_{k\in
        \frac{2\pi}{\ell}\Z^2} |e^{-ik\cdot z} -1|^2 |\widehat
        f_k|^2\frac{1}{|z|^3}\dz2\\ 
      &=\frac{1}{\ell^2}\sum_{k\in \frac{2\pi}{\ell}\Z^2} |\widehat
        f_k|^2 \Int_{\R^2}|e^{-ik\cdot z} -1|^2 \frac{1}{|z'|^3}\dz2
        \lupref{eq-fi}= \frac{4\pi}{\ell^2}\sum_{k\in
        \frac{2\pi}{\ell}\Z^2}  |k||\widehat f_k|^2,
    \end{align*}
    which gives the desired formula.
\end{proof}

\end{appendices}

\bibliographystyle{siam}
\bibliography{bib}

\end{document}